\documentclass[11pt,twoside,reqno]{amsart}
\usepackage{amsmath}
\usepackage{amsfonts}
\usepackage{amssymb}
\usepackage{color}
\usepackage{mathrsfs}
\usepackage{cite}

\usepackage{tikz}
\usepackage{pgfplots}
\pgfplotsset{compat=1.10}
\usepgfplotslibrary{fillbetween}
\usetikzlibrary{patterns}
\allowdisplaybreaks
\newcommand{\R}{\mathbb R} 
\newcommand{\rd}{\mathrm{d}}
\definecolor{cadmiumgreen}{rgb}{0.0, 0.42, 0.24}
\newtheorem{theorem}{Theorem}[section]
\newtheorem{corollary}[theorem]{Corollary}
\newtheorem{lemma}[theorem]{Lemma}
\newtheorem{proposition}[theorem]{Proposition}

\newtheorem{remark}[theorem]{Remark}
\newtheorem{example}[theorem]{Example}
\numberwithin{equation}{section}

\newcommand{\hb}[1]{H_B^1(#1)}
\begin{document}
\title{Energy Minimizers for an Asymptotic MEMS Model with Heterogeneous Dielectric Properties}
 \thanks{Partially supported by the CNRS Projet International de Coop\'eration Scientifique PICS07710}

\author{Philippe Lauren\c{c}ot}
\address{Institut de Math\'ematiques de Toulouse, UMR~5219, Universit\'e de Toulouse, CNRS \\ F--31062 Toulouse Cedex 9, France}
\email{laurenco@math.univ-toulouse.fr}
\author{Katerina Nik}
\address{Leibniz Universit\"at Hannover\\ Institut f\" ur Angewandte Mathematik \\ Welfengarten 1 \\ D--30167 Hannover\\ Germany}
\email{nik@ifam.uni-hannover.de}
\author{Christoph Walker}
\address{Leibniz Universit\"at Hannover\\ Institut f\" ur Angewandte Mathematik \\ Welfengarten 1 \\ D--30167 Hannover\\ Germany}
\email{walker@ifam.uni-hannover.de}
\keywords{Minimizers, shape derivative, obstacle problem, bilaplacian operator}
\subjclass{35J50 - 49Q10 - 49J40 - 35R35 - 35Q74}

\date{\today}

\begin{abstract}
A model for a MEMS device, consisting of a fixed bottom plate and an elastic plate, is  studied. It was derived in a previous work as a reinforced limit  when the thickness of the insulating layer covering the bottom plate tends to zero. This asymptotic model inherits the dielectric properties of the insulating layer. It involves the electrostatic potential in the device and the deformation of the elastic plate defining the geometry of the device. The electrostatic potential is given by an elliptic equation with mixed boundary conditions in the possibly non-Lipschitz region between the two plates. The deformation of the elastic plate is supposed to be a critical point of an energy functional which, in turn, depends on the electrostatic potential due to the force exerted by the latter on the elastic plate. The energy functional is shown to have a minimizer giving the geometry of the device. Moreover, the corresponding Euler-Lagrange equation is computed and the maximal regularity of the electrostatic potential is established.
\end{abstract}

\maketitle

%
%
\pagestyle{myheadings}
\markboth{\sc Ph. Lauren\c cot, K. Nik, and Ch. Walker}{\sc Energy Minimizers for an Asymptotic MEMS Model}

\section{Introduction}\label{sec.IntP}

The modeling and analysis of microelectromechanical systems (MEMS) has attracted a lot of interest in recent years, see, e.g., \cite{EGG10, FMCCS05, LWBible, LW18, Pe01, PeB03,You11} and the references therein. Idealized devices often consist of a rigid dielectric ground plate above which an elastic dielectric plate is suspended. Applying a voltage difference between the two plates induces a competition between attractive electrostatic Coulomb forces and restoring mechanical forces, the latter resulting from the elasticity of the upper plate. When electrostatic forces dominate mechanical forces, the two plates may come into contact, a phenomenon usually referred to as pull-in instability or touchdown. From a mathematical point of view, this phenomenon may be accounted for in different ways. In fact, in most mathematical models considered so far in the MEMS literature, the pull-in instability is revealed as a singularity in the corresponding mathematical equations which coincides with a breakdown of the model, see \cite{EGG10, LWBible,PeB03} and the references therein. There is a close connection between the singular character of the touchdown and the fact that the modeling does not account for the thickness of the plates. Indeed, coating the ground plate with a thin insulating layer prevents a direct contact of the plates, so that a touchdown of the elastic plate on the insulating layer does not interrupt the operation of the device \cite{BG01, LLG14, LLG15, LW19}. Due to the presence of this layer, the MEMS device features heterogeneous dielectric properties (with a jump of the permittivity at the interface separating the coated ground plate and the free space beneath the elastic plate) and the electrostatic potential solves a free boundary transmission problem in the non-smooth domain enclosed between the two plates \cite{LW19}. The shape of the domain itself is given by a partial differential equation governing the deflection of the elastic plate from rest, which, in turn, involves the electrostatic force exerted on the latter.
The mathematical treatment of such a model is rather complex, see \cite[Section~5]{LW19} and \cite{LW19b}.  It is thus desirable to derive simpler and more tractable models. As the modeling involves two small spatial scales -- the aspect ratio $\varepsilon$ of the device and the thickness $d$ of the insulating layer -- a variety of reduced models may be obtained. For instance, the assumption of a vanishing aspect ratio of the device, when either the ratio $d/\varepsilon$ has a positive finite limit \cite{AmEtal, BG01, LW17, LLG14, LLG15} or converges to zero, see \cite{EGG10, Pe01, PeB03} and the references therein, leads to a model which no longer involves a free boundary. Indeed, in that case, the electrostatic potential can be computed explicitly in terms of the deflection of the elastic plate and the model reduces to a single equation for the deflection, with the drawback that some important information on the electrostatic potential may thus be lost. 

For this reason an intermediate model is derived in \cite{LNW19} by letting only the thickness of the insulating  layer $d$ go to zero (keeping the aspect ratio of the device of order one). Assuming an appropriate scaling of the dielectric permittivity in dependence on the layer's thickness (in order to keep relevant information of the dielectric heterogeneity of the device) and using a Gamma convergence approach, the resulting energy, which is the building block of the model, is computed. The next step is the mathematical analysis of the thus derived model, in which stationary solutions correspond to critical points of the energy, while the dynamics is described by the gradient flow associated with the energy. The aim of the present work is to show the existence of a particular class of stationary solutions, which are additionally energy minimizers, and to identify the corresponding Euler-Lagrange equations.

Let us provide beforehand a more precise description of the MEMS configuration under study. We consider an idealized MEMS device composed of two rectangular two-dimensional dielectric plates: a fixed ground plate above which an elastic plate, with the same shape at rest, is suspended and clamped in only one direction while free in the other. We assume that the device is homogeneous in the free direction and that it is thus sufficient to consider only a cross-section of the device orthogonal to the free direction. The shape of the ground plate and that of the elastic plate at rest are then represented by $D:=(-L,L)\subset\R$, the ground plate being located at $z=-H$ with $H>0$ and covered with an infinitesimally thin dielectric layer (in consistency with the aforementioned limit). The vertical deflection of the elastic plate from its rest position at $z=0$ is described by a function  $u:\bar D\rightarrow [-H, \infty)$ satisfying the clamped boundary conditions 
\begin{equation}\label{ubc}
u(\pm L)= \partial_x u(\pm L)=0\,,
\end{equation}
so that its graph
$$
\mathfrak{G}(u):=\{(x,u(x))\,:\, x\in \bar D\}
$$
represents the elastic plate and
$$
 \Omega(u):=\left\{(x,z)\in D\times \mathbb{R}\,:\, -H<  z <  u(x)\right\}
$$
is the free space between the elastic plate and the ground plate. Since we do not exclude the possibility of contact between the two plates, we introduce the {\it coincidence set}
$$
\mathcal{C}(u):=\{x\in D\, :\, u(x)=-H\}
$$
and let
$$
\Sigma (u):=\{(x,-H)\,:\, x\in D,\, u(x)>-H\} = \big(D\setminus\mathcal{C}(u)\big) \times \{-H\}
$$
be the part of the ground plate which is not in contact with the elastic plate. A touchdown of the elastic plate on the ground plate corresponds to a non-empty coincidence set, in which case $\Sigma (u)$ is a strict subset of $D\times \{-H\}$. Note that the free space $\Omega(u)$ then has a  different geometry with at least two connected components, which may not be Lipschitz domains due to cusps (independent of the smoothness of the function $u$).  In Figure~\ref{F1} the different situations with empty and non-empty coincidence sets are depicted.

 \begin{figure}
 	\begin{tikzpicture}[scale=0.9]
 	\draw[black, line width = 1.5pt, dashed] (-7,0)--(7,0);
 	\draw[black, line width = 2pt] (-7,0)--(-7,-4);
 	\draw[black, line width = 2pt] (7,-4)--(7,0);
 	\draw[black, line width = 2pt] (-7,-4)--(7,-4);
 	\draw[cadmiumgreen, line width = 2pt] plot[domain=-7:7] (\x,{-1-cos((pi*\x/7) r)});
	\draw[blue, line width = 2pt] plot[domain=-7:-1] (\x,{-2-2*cos((pi*(\x+1)/6) r)});
 	\draw[blue, line width = 2pt] (-1,-4)--(3,-4);
	\draw[blue, line width = 2pt] plot[domain=3:7] (\x,{-2-2*cos((pi*(\x-3)/4) r)});
 	\draw[cadmiumgreen, line width = 1pt, arrows=->] (3,0)--(3,-1.15);
 	\node at (3.2,-0.6) {${\color{cadmiumgreen} v}$};
 	\draw[blue, line width = 1pt, arrows=->] (5,0)--(5,-1.85);
 	\node at (5.2,-1) {${\color{blue} w}$};
 	\node at (3,-2) {${\color{cadmiumgreen} \Omega(v)}$};
 	\node at (-5.5,-2) {${\color{blue} \Omega(w)}$};
 	\node at (6,-2) {${\color{blue} \Omega(w)}$};
 	\node at (3.75,-4.75) {$D$};
	\draw (3.55,-4.75) edge[->,bend left, line width = 1pt] (2.3,-4.1);
 	\node at (7.5,-5) {$\Sigma(w)$};
 	\draw (7,-5) edge[->,bend left, line width = 1pt] (5.2,-4.1);
 	\node at (-6,-5) {$\Sigma(w)$};
 	\draw (-5.5,-5) edge[->,bend right, line width = 1pt] (-4,-4.1);
 	\node at (-7.8,1) {$z$};
 	\draw[black, line width = 1pt, arrows = ->] (-7.5,-5)--(-7.5,1);
 	\node at (-8,-4) {$-H$};
 	\draw[black, line width = 1pt] (-7.6,-4)--(-7.4,-4);
 	\node at (-7.8,0) {$0$};
 	\draw[black, line width = 1pt] (-7.6,0)--(-7.4,0);
 	\node at (1,-3) {${\color{blue} \mathcal{C}(w)}$};
 	\draw (0.45,-3) edge[->,bend right,blue, line width = 1pt] (-0.5,-3.95);
 	\end{tikzpicture}
 	\caption{Geometry of $\Omega(u)$ for a state $u=v$ with empty coincidence set (\textcolor{cadmiumgreen}{green}) and a state $u=w$ with non-empty coincidence set (\textcolor{blue}{blue}).}\label{F1}
 \end{figure}
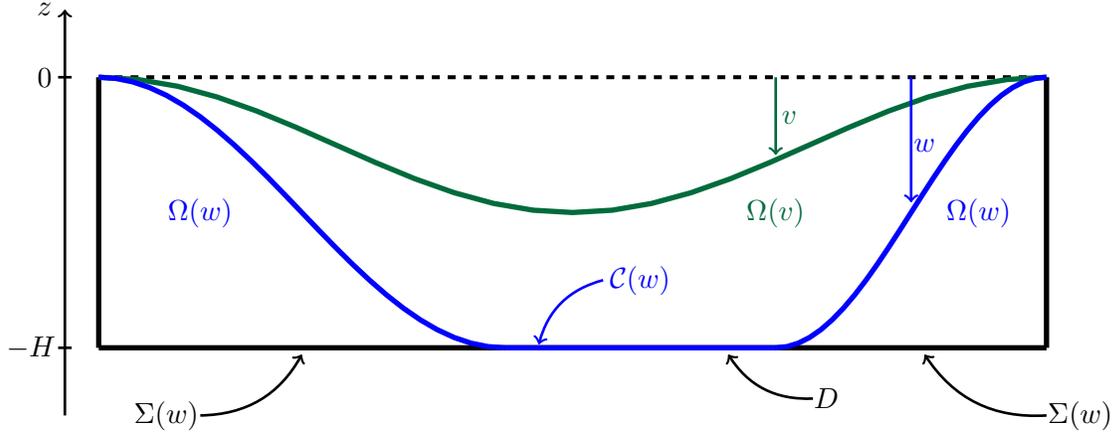

As already mentioned, the building block of the model studied in this paper is the total energy $E(u)$ of the device at a state $u$ given by
$$
E(u):= E_m(u) + E_e(u)
$$ 
and derived in \cite{LNW19} in the limit of an infinitesimally small insulating layer. It 
consists of the mechanical energy $E_m(u)$ and the electrostatic energy $E_e(u)$. The former is given by
$$
E_m(u):=\frac{\beta}{2}\|\partial_x^2u\|_{L_2(D)}^2 +\left(\frac{\tau}{2}+\frac{\alpha}{4}\|\partial_x u\|_{L_2(D)}^2\right)\|\partial_x u\|_{L_2(D)}^2
$$
with $\beta>0$ and $\tau,\alpha \ge 0$, taking into account bending and external- and self-stretching effects of the elastic plate. The electrostatic energy is
\begin{equation}\label{Ee}
E_e(u):=-\dfrac{1}{2}\displaystyle\int_{\Omega(u)}  \big\vert\nabla \psi_u\big\vert^2\,\rd (x,z)
 -\dfrac{1}{2}\displaystyle\int_{ D} \sigma(x) \big\vert \psi_u(x,-H) - \mathfrak{h}_{u}(x) \big\vert^2\,\rd x\,,
\end{equation}
where $\psi_u$ is the electrostatic potential in the device and solves the elliptic equation with mixed boundary conditions
\begin{subequations}\label{MBP0}
\begin{align}
\Delta\psi_u&=0 \quad\text{in }\ \Omega(u)\,, \label{MBP1}\\
\psi_u &=h_u\quad\text{on }\ \partial\Omega(u)\setminus  \Sigma(u)\,,\label{MBP2}\\
- \partial_z\psi_u +\sigma (\psi_u-\mathfrak{h}_u)&=0\quad\text{on }\    \Sigma(u)\label{MBP3}\,.
\end{align}
\end{subequations}
In \eqref{MBP0}, the function $\sigma$ represents the properties of the dielectric permittivity inherited from the insulating layer while the functions $h_u$ and $\mathfrak{h}_u$ determining the boundary values of $\psi_u$ on $\partial\Omega(u)$  are of the form
\begin{equation}\label{hnot}
h_u(x,z):=h(x,z,u(x))\,,\quad (x,z)\in \bar{D}\times [-H,\infty)\,,\qquad \mathfrak{h}_u(x):=\mathfrak{h}(x,u(x))\,,\quad x\in \bar D\,,
\end{equation}  
for some prescribed functions
\begin{equation*}
h:\bar D\times [-H,\infty)\times [-H,\infty)\rightarrow \R \,,\qquad \mathfrak{h}: \bar D\times [-H,\infty)\rightarrow \R\,.
\end{equation*} 

The main results of this work are the existence of at least one minimizer of the total energy~$E$ and the derivation of the corresponding Euler-Lagrange equation. This requires, of course, first to study the well-posedness of the elliptic problem \eqref{MBP0} subject to its mixed boundary conditions. A first step in that direction is to guarantee that the electrostatic energy~$E_e$ is well-defined, which turns out to require some care. Indeed, it should be pointed out that $\Omega(u)$ is a non-smooth domain with corners and possibly features turning points, for instance when $\mathcal{C}(u)$ includes an interval, see Figure~\ref{F1}. Thus, $\Omega(u)$ might consist of several components no longer having a Lipschitz boundary, so that traces have first to be given a meaning. Once this matter is settled, the existence of a variational solution $\psi_u$ to \eqref{MBP0} readily follows from the Lax-Milgram Theorem and the electrostatic energy is then well-defined. This paves the way to the proof of the existence of minimizers of the total energy by the direct method of calculus of variations but does not yet allow us to conclude. Indeed, since $E$ involves two contributions with opposite signs, it might be unbounded from below. We overcome this difficulty by adding a penalization term to the total energy. This additional term can be removed afterwards, thanks to an \textit{a priori} upper bound on the minimizers which follows from the corresponding Euler-Lagrange equation. However, it turns out that the derivation of the latter requires additional regularity of the electrostatic potential $\psi_u$. Such a regularity is actually not obvious, as the highest expected smoothness of the boundary of $\Omega(u)$ is Lipschitz regularity (when the coincidence set $\mathcal{C}(u)$ is empty). Consequently, one needs to establish sufficient regularity for $\psi_u$ both for states $u$ with empty and with non-empty coincidence sets $\mathcal{C}(u)$. In particular, this will ensure a well-defined  normal trace of the gradient of $\psi_u$ on $\Sigma(u)$ as required by \eqref{MBP3} and on the part of $\mathfrak{G}(u)$ lying above $\Sigma(u)$ as required by \eqref{g} below. The above mentioned difficulties are actually not the only ones that we face in the forthcoming analysis. To name but a few, the electrostatic energy $E_e(u)$ features a nonlocal and intricate dependence upon the state $u$ and appropriate continuity properties are needed in the minimizing procedure. This requires a thorough understanding of the dependence of $\psi_u$ on the state $u$, this dependence being due to the domain $\Omega(u)$ as well as the functions $h_u$ and~$\mathfrak{h}_u$. Also, due to the prescribed constraint $u\ge -H$, the Euler-Lagrange equation solved by minimizers is in fact a variational inequality.

\section{Main Results}\label{sec.IntP2}

Throughout this work we shall assume that 
\begin{subequations}\label{hH}
\begin{equation}
\sigma\in C^2(\bar{D})\,,\qquad \sigma(x)>0\,,\quad x\in \bar D\,. \label{beep}
\end{equation}
As for the functions $h_u$ and $\mathfrak{h}_u$ appearing in \eqref{MBP0} we 
shall assume in the following that
\begin{equation} \label{zucchero}
h\in C^2(\bar D\times [-H,\infty)\times [-H,\infty))\,,\qquad \mathfrak{h}\in C^1(\bar D\times [-H,\infty))\,,
\end{equation} 
satisfy
\begin{equation}
\partial_z h(x,-H,w) = \sigma(x) \big[ h(x,-H,w) - \mathfrak{h}(x,w) \big]\,, \qquad (x,w)\in D\times [-H,\infty)\, . \label{bryanadams}
\end{equation}
\end{subequations}
Assumption~\eqref{bryanadams} allows us later to rewrite \eqref{MBP0} as an elliptic equation with homogeneous boundary conditions. In the following, we shall use the notation introduced in \eqref{hnot}.

A simple example of boundary functions $(h,\mathfrak{h})$ satisfying \eqref{zucchero} and \eqref{bryanadams} may be derived from \cite[Example~5.5]{LW19} with the scaling from \cite{LNW19}:

\begin{example}\label{ex.h}
 Let $V>0$ and set 
\begin{equation*}
h(x,z,w) := V \frac{1+\sigma(x)(H+z)}{1+\sigma(x)(H+w)}\,, \qquad (x,z,w)\in \bar{D}\times [-H,\infty)\times [-H,\infty)\,, 
\end{equation*}
and $\mathfrak{h}\equiv 0$. Then $(h,\mathfrak{h})$ clearly satisfies \eqref{zucchero} and \eqref{bryanadams}, the former being a consequence of the regularity \eqref{beep} of $\sigma$. Note that $h_u(x,u(x))=V, x\in D$, for a given state $u$; that is, in this example the electrostatic potential is kept constant to the value $V$ along the elastic plate, see \eqref{MBP2}.
\end{example}

\subsection{The Electrostatic Potential}

We first turn to the existence of an electrostatic potential for a given state $u$.
To have an appropriate functional setting for $u$ we introduce
\begin{equation}
\bar{S}  := \{ u\in H^2(D)\cap H_0^1(D)\ :\ -H\le u \;\text{ in }\; D \}\,,
\label{SP}
\end{equation}
and point out that $\mathcal{C}(u)=\emptyset$ if and only if $u$ belongs to the interior of $\bar S$; that is,  $u\in S$, where
$$
S := \{ u\in H^2(D)\cap H_0^1(D)\ :\ -H<u \;\text{ in }\; D \}\,.
$$
Note that $H^2(D)$ is  embedded in $C(\bar D)$ so that $\Omega(u)$ is well-defined for $u\in \bar S$. Regarding the well-posedness of \eqref{MBP0} we shall prove the following result.

\begin{theorem}\label{Thmpsi}
Suppose \eqref{hH}. For each $u\in \bar S$ there exists a unique strong solution $\psi_u\in  H^2(\Omega(u))$ to \eqref{MBP0}. Moreover, given $\kappa>0$ and  $r\in [2,\infty)$, there are $c(\kappa)>0$ and $c(r,\kappa)>0$ such that
\begin{equation*}
\|\psi_u\|_{H^2(\Omega(u))} + \|\partial_x\psi_u(\cdot,-H)\|_{L_2(D\setminus\mathcal{C}(u))} \le c(\kappa)\,, \qquad \|\partial_z\psi_u(\cdot,u)\|_{L_r(D\setminus\mathcal{C}(u))} \le c(r,\kappa)
\end{equation*}
 for each  $u\in \bar S$ with $\|u\|_{H^2(D)} \le \kappa$.
\end{theorem}

Theorem~\ref{Thmpsi} is an immediate consequence of Lemma~\ref{h}, Theorem~\ref{thmt1PP}, and \eqref{chi} below.

\subsection{Existence of Energy Minimizers}

Owing to Theorem~\ref{Thmpsi}, the total energy is well-defined on the set
$$
\bar S_0 := \{ u\in H^2(D)\, :\, u(\pm L)=\partial_x u(\pm L)=0\,,\, -H\le u \;\text{ in }\; D \}\subset\bar S \,,
$$
taking into account the clamped boundary conditions \eqref{ubc}. We shall now focus on the existence of energy minimizers on $\bar S_0$. We have already observed that the total energy $E$ is the sum of two terms $E_m$ and $E_e$ with different signs. Hence, the coercivity of $E$ is not obvious.  However, if $\alpha>0$, the first order term in the mechanical energy $E_m$ is quartic and thus dominates the negative contribution coming from the electrostatic energy $E_e$. This property allows us to follow the lines of \cite[Section~5]{LW19} to derive the coercivity of $E$ based on the following growth assumption for $h$: there is a constant $K>0$ such that 
\begin{subequations}\label{bb5}
\begin{equation} 
\vert \partial_x h(x,z,w)\vert +\vert\partial_z h(x,z,w)\vert \le K \sqrt{\frac{1+w^2}{H+w}}\,,\quad \vert\partial_w h(x,z,w)\vert \le \frac{K}{\sqrt{H+w}}\,, \label{bb6}
\end{equation}
for $(x,z,w)\in \bar D \times [-H,\infty) \times [-H,\infty)$ and
\begin{equation}\label{bb7}
\vert h(x,-H,w)\vert +\vert\mathfrak{h}(x,w)\vert \le K\,, \quad (x,w)\in \bar D\times [-H,\infty)\,.
\end{equation}
\end{subequations}
This approach no longer works if $\alpha=0$ and the coercivity of $E$ is not granted. To remedy this drawback, we shall use a regularized energy functional (see \eqref{Ek} below), which includes a penalization term ensuring its coercivity if, in addition to \eqref{bb5}, we assume that 
\begin{subequations}\label{hbound}
\begin{equation}\label{hbound1}
\vert h(x,w,w)\vert + \vert h(\pm L,z,w)\vert \le K\,,  \quad (x,z,w)\in \bar D\times [-H,\infty)\times[-H,\infty)\,,
\end{equation}
and
\begin{equation}\label{hbound2}
|\partial_x h(x,w,w)| + |\partial_z h(x,w,w)| + |\partial_w h(x,w,w)| + \vert\partial_w \mathfrak{h}(x,w)\vert \le K\,, \quad (x,w)\in \bar D\times [-H,\infty)\,.
\end{equation}
\end{subequations}
We complete the analysis when $\alpha=0$ by showing that minimizers of the regularized energy functional for a suitable choice of the penalization parameter give rise to a minimizer of $E$, establishing indirectly that $E$ is bounded from below in that case as well. Consequently, in both cases we can prove the existence of at least one energy minimizer as stated in the next result.

\begin{theorem}\label{Thm22}
Assume  \eqref{hH} and \eqref{bb5} and, either  $\alpha>0$, or
 $\alpha=0$ and \eqref{hbound}. Then the total energy $E$ has at least one minimizer $u_*$ in $\bar{S_0}$; that is,  $u_*\in\bar{S_0}$ and
\begin{equation}\label{B}
E(u_*)=\min_{\bar{S_0}}E\,.
\end{equation}
\end{theorem}

At this point, no further qualitative information on energy minimizers $u_*$ is available, and a particularly interesting question, which is yet left unanswered by our analysis, is whether the coincidence set $\mathcal{C}(u_*)$ is empty or not. Another interesting open issue is the uniqueness of minimizers. The proof of Theorem~\ref{Thm22} is given in Section~\ref{proof} for $\alpha=0$ and in Section~\ref{proof2} for $\alpha>0$.

\subsection{Euler-Lagrange Equation}

We next aim at deriving the Euler-Lagrange equation satisfied by minimizers of the total energy $E$. Recalling the prescribed constraint $u\ge -H$ for $u\in \bar{S_0}$, we are dealing with an obstacle problem and the resulting equation is actually a variational inequality. For the precise statement we introduce, for a given $u \in \bar{S}$, the function \mbox{$g(u):D\rightarrow\R$} by setting 
\begin{subequations}\label{GG}
\begin{equation}
\label{g}
\begin{split}
 g(u)(x):= &
\frac{1}{2} (1+\vert\partial_x u(x)\vert^2)\,\big[\partial_z\psi_{u}-(\partial_z h)_{u}-(\partial_w h)_{u}\big]^2(x, u(x))\\
& + \sigma(x)\big[\psi_{u}(x,-H)-\mathfrak{h}_{u}(x)\big](\partial_w \mathfrak{h})_{u}(x)\\
&-\frac{1}{2} \left[ \big\vert(\partial_x h)_u\big\vert^2+ \big((\partial_z h)_u+(\partial_w h)_u\big)^2 \right](x, u(x))
\end{split}
\end{equation}
for $x\in D\setminus\mathcal{C}(u)$ while setting
\begin{equation}
\label{gg}
\begin{split}
g(u)(x):= & \frac{1}{2} \vert (\partial_w h)_{u}\vert^2(x, -H)+  \sigma(x)\big[h(x,-H,-H)-\mathfrak{h}_u(x)\big](\partial_w \mathfrak{h})_{u}(x)\\
&-\frac{1}{2} \left[ \big\vert(\partial_x h)_u\big\vert^2+ \big((\partial_z h)_u+(\partial_w h)_u\big)^2 \right](x, -H)
\end{split}
\end{equation}
\end{subequations}
for $x\in \mathcal{C}(u)$. In fact, $g(u)$ represents the electrostatic force exerted on the elastic plate and is computed as the differential (in a suitable sense) of the electrostatic energy $E_e(u)$ with respect to $u$. We emphasize here that the regularity properties of $\psi_u$ established in Theorem~\ref{Thmpsi} are of utmost importance to guarantee that $g(u)$ is well-defined on $D\setminus\mathcal{C}(u)$,  since it features the trace of $\partial_z \psi_u$ on $\mathfrak{G}(u)$. With this notation, we are able to identify the variational inequality solved (in a weak sense) by energy minimizers.

\begin{theorem}\label{Thm33}
Assume  \eqref{hH}. Assume that $u\in \bar{S_0}$ is a minimizer of $E$ on $\bar S_0$. Then $g(u)\in L_2(D)$ and $u$ is an $H^2$-weak solution to the variational inequality
\begin{equation}
\beta\partial_x^4u-(\tau+\alpha\|\partial_x u\|_{L_2(D)}^2)\partial_x^2 u+\partial\mathbb{I}_{\bar{S_0}}(u) \ni -g(u) \;\;\text{ in }\;\; D \label{bennygoodman0}
\end{equation}
where $\partial\mathbb{I}_{\bar{S_0}}$ denotes the subdifferential of the indicator function $\mathbb{I}_{\bar S_0}$ of the closed convex subset $\bar{S_0}$ of $H^2(D)$; that is, 
\begin{equation*}
\begin{split}
\int_D &\Big\{\beta\partial_x^2 u\,\partial_x^2 (w-u)+\big[\tau+\alpha \|\partial_x u\|_{L_2(D)}^2\big] \partial_x u\, \partial_x(w-u)\Big\}\,\rd x\ge -\int_D g(u) (w-u)\, \rd x  
\end{split}
\end{equation*}
for all $w\in \bar{S_0}$. 
\end{theorem}

At this point, we do not know whether minimizers of $E$ in $\bar{S_0}$ are the only solutions to \eqref{bennygoodman0}, a question closely connected to the uniqueness issue for \eqref{bennygoodman0}. It is, however, expected that the set of solutions to \eqref{bennygoodman0} exhibits a complex structure. Indeed, in the much simpler situation studied in \cite{LW17}, the minimizer may coexist with other steady states,  depending on the boundary values of the electrostatic potential.

The proof of Theorem~\ref{Thm33} is given in Section~\ref{proof} for $\alpha=0$ and in Section~\ref{proof2} for $\alpha>0$. It relies on the computation of the shape derivative of the electrostatic energy $E_e(u)$, which is performed in Section~\ref{Sec5}.

\begin{remark}\label{rem.NavierBC}
It is also possible to minimize the total energy $E$ on the set $\bar S$ (instead on $\bar S_0$). Then the corresponding minimizer in $\bar S$ satisfies instead of the clamped boundary conditions \eqref{ubc} the  Navier or pinned boundary conditions $u(\pm L)=\partial_x^2 u(\pm L)=0$. With this change, the statements of Theorem~\ref{Thm22}  and Theorem~\ref{Thm33}  remain true when $\bar S_0$ is replaced everywhere by $\bar S$.
\end{remark}

Now, combining Theorem~\ref{Thm22} and Theorem~\ref{Thm33} we obtain the existence of a stationary configuration of the MEMS device given as a solution to the  force balance~\eqref{bennygoodman0}:

\begin{corollary}\label{Cor44}
Assume  \eqref{hH} and \eqref{bb5} and, either  $\alpha>0$, or
 $\alpha=0$ and \eqref{hbound}. Then there is a solution $u_*\in \bar S_0$ to the variational inequality~\eqref{bennygoodman0}.
\end{corollary}

The subsequent sections are dedicated to the proofs of the results stated in this section.

\bigskip

Throughout the paper, we impose assumptions \eqref{hH} and set
\begin{equation}
\sigma_{min} := \min_{\bar{D}}\{\sigma\}>0 \,,\qquad  \bar{\sigma} := \|\sigma\|_{C^2(\bar{D})} <\infty\,.  \label{s0P}
\end{equation}

\section{Existence and $H^2$-Regularity of the Electrostatic Potential $\psi_u$}\label{sec.hP}

This section is dedicated to the proof of Theorem~\ref{Thmpsi}; that is, to the existence and regularity of a unique solution $\psi_u$ to \eqref{MBP0}. We first recall some basic properties of the boundary function $h_v$ which are established in \cite[Lemma~3.10]{LW19} and rely on the properties \eqref{zucchero} and \eqref{bryanadams} of $h$ and $\mathfrak{h}$.

\begin{lemma}\label{h}
Let $M>0$.

\noindent {\bf (a)} Given $v\in \bar S$  satisfying $-H\le v(x)\le M-H$ for $x\in D$, the function $h_v$ belongs to $H^2(\Omega(v))$ and
\begin{equation}\label{hh}
\begin{split}
& \|h_v\|_{H^2(\Omega(v))}\le C(M) \big(1 +\|\partial_x^2 v\|_{L_2(D)}^2\big)\,, \\
&  \|\partial_x h_v(\cdot,-H)\|_{L_2(D)} \le C(M) \big(1+\|\partial_x v\|_{L_2(D)}\big)\,, \\
&  \|\partial_z h_v(\cdot,v)\|_{L_r(D)} \le C(M)\,, \qquad r\in [1,\infty]\,.
\end{split}
\end{equation}

\noindent {\bf (b)}  Consider a sequence $(v_n)_{n\ge 1}$ in $\bar{S}$ and $v\in \bar S$ such that 
\begin{equation}\label{p}
-H\le v_n(x)\,,\, v(x)\le M-H\,,  \quad x\in D\,, \qquad v_n\rightarrow v \text{ in }\ H_0^1(D)\,.
\end{equation}
Let $\Omega(M) := D\times (-H,M)$. Then
\begin{align}
h_{v_n}&\rightarrow h_v\quad\text{in }\ H^1(\Omega(M))\,, \label{Gcz5} \\
h_{v_n}(\cdot,-H)&\rightarrow h_v(\cdot,-H) \quad\text{in }\ L_2(D)\,, \label{Gcz6} \\
\mathfrak{h}_{v_n}&\rightarrow \mathfrak{h}_v \quad\text{in }\ L_2(D)\,. \label{Gcz7}
\end{align}
\end{lemma}

\begin{proof}
Integrating
$$
\partial_x v(x)=\partial_x v(y)+\int_y^x\partial_x^2 v(z)\,\rd z\,, \qquad  (x,y)\in [-L,L]^2\,,
$$
with respect to $y\in [-L,L]$ and taking into account the boundary condition $v(\pm L)=0$, we obtain
$$
2L\partial_x v(x)=\int_{-L}^L \int_y^x\partial_x^2 v(z)\,\rd z\,\rd y\,,\quad x\in [-L,L]\,.
$$
Hence, by H\"older's inequality we get
$$
\|\partial_x v\|_{L_\infty(D)}\le \sqrt{2L}\|\partial_x^2v\|_{L_2(D)}\,.
$$
Using this inequality and the fact that $h$ and its derivatives up to second order are bounded on $\bar D\times [-H,M]\times [-H,M]$ we derive
\begin{align*}
\|h_v\|_{H^2(\Omega(v))} & \le C(M) \big(1+\|\partial_x v\|_{L_2(D)}+\|\partial_x v\|_{L_\infty(D)} \|\partial_x v\|_{L_2(D)} + \|\partial_x^2 v\|_{L_2(D)}\big) \\
&  \le C(M) \big(1+\|\partial_x^2 v\|_{L_2(D)} + \|\partial_x^2 v\|_{L_2(D)}^2\big)\,,
\end{align*}
which yields {\bf (a)}.  As for {\bf (b)} we first note that \eqref{p} and the compact embedding of $H^1(D)$ in $C(\bar{D})$ ensure that
\begin{equation*}
v_n \rightarrow v \;\;\text{ in }\;\; C(\bar{D})\,.
\end{equation*}
Combining this convergence with \eqref{p} and the continuity properties \eqref{zucchero} of $h$ and $\mathfrak{h}$ readily gives \eqref{Gcz6} and \eqref{Gcz7}, as well as \eqref{Gcz5} with the additional use of \eqref{p}, see \cite[Lemma~3.10]{LW19}.
\end{proof}

We shall now prove Theorem~\ref{Thmpsi} and thus focus on \eqref{MBP0}, which is more conveniently considered with homogeneous boundary conditions. To this end, we introduce  
\begin{equation}\label{chi}
\chi_v:=\psi_v-h_v
\end{equation}
for a given and fixed function $v\in \bar S$. Due to assumption \eqref{bryanadams}, problem  \eqref{MBP0} (with $v$ instead of $u$) is  then equivalent to
\begin{subequations}\label{bbb}
\begin{align}
-\Delta\chi_v & = \Delta h_v \;\text{ in }\; \Omega(v)\, , \\
\chi_v & = 0 \;\text{ on }\; \partial\Omega(v)\setminus \Sigma(v)\, , \\
-\partial_z \chi_v + \sigma \chi_v & = 0 \;\text{ on }\; \Sigma(v)\,.
\end{align}
\end{subequations}
Hence, the next result can be seen as a reformulation of Theorem~\ref{Thmpsi} in terms of $\chi_v$.

\begin{theorem}\label{thmt1PP}
Consider a function $v\in \bar S$ and let $\kappa>0$ be such that 
\begin{equation}
\|v\|_{H^2(D)} \le \kappa\ . \label{t2PP}
\end{equation}
Then there exists a unique strong solution $\chi_v\in H^2(\Omega(v))$ to \eqref{bbb}
 and there is $C(\kappa)>0$ depending only on $\sigma$ and $\kappa$ such that
\begin{equation}
\|\chi_v\|_{H^2(\Omega(v))} + \|\partial_x\chi_v(\cdot,-H)\|_{L_2(D\setminus\mathcal{C}(v))} \le C(\kappa)\ . \label{t6PP}
\end{equation}
Moreover, for any $r\in [2,\infty)$, there is $C(\kappa) >0$ depending only on $\sigma$ and $\kappa$ such that
\begin{equation}
\|\partial_z\chi_v(\cdot,v)\|_{L_r(D\setminus\mathcal{C}(v))} \le r  C(\kappa) \ .  \label{t6rPP}
\end{equation}
\end{theorem}

The remainder of this section is devoted to the proof of Theorem~\ref{thmt1PP}. 
	
\subsection{Variational Solution to \eqref{bbb}}
 	
We first establish the existence of a variational solution to \eqref{bbb}. To this end, we introduce for $v\in \bar S$ the space $H_B^1(\Omega(v))$ as the closure in $H^1(\Omega(v))$ of the set
\begin{equation*}
C_B^1\big( \overline{\Omega(v)} \big) := \Big\{\theta\in C^1\big( \overline{\Omega(v)} \big)\,:\, \theta(x,v(x))=0\,,\ x\in D\,,  \,\theta(\pm L,z)=0\,,\ z\in (-H,0 ) \Big\}\,,
\end{equation*}
and shall then minimize the functional
\begin{equation}\label{GFunc}
\mathcal{G}(v)[\vartheta] := \frac{1}{2} \int_{\Omega(v)} |\nabla (\vartheta+h_v)|^2 \,\mathrm{d}(x,z) + \frac{1}{2} \int_D \sigma(x) |\vartheta(x,-H)+h_v(x,-H) - \mathfrak{h}_v(x)|^2   \,\mathrm{d}x
\end{equation}
with respect to $\vartheta\in H_B^1(\Omega(v))$. Let us recall from \cite[Lemma~2.2]{LNW19} that the trace $\vartheta(\cdot,-H)\in L_2(D)$ is well-defined for $\vartheta\in H_B^1(\Omega(v))$ (see also Lemma~\ref{lemT2H} below for a complete statement), while Lemma~\ref{h} ensures that $h_v\in  H^1(\Omega(v))$ and that $h_v(\cdot,-H)$ and $\mathfrak{h}_v$ belong to $L_2(D)$. Thus, $\mathcal{G}(v)[\vartheta]$ is well-defined for $\vartheta\in H_B^1(\Omega(v))$.

\begin{proposition}\label{lemt9P}
Let $v\in\bar{S}$. There is a unique variational solution $\chi_v\in H_B^1(\Omega(v))$ to \eqref{bbb}
given as the unique minimizer of the functional $\mathcal{G}(v)$ on $H_B^1(\Omega(v))$. Moreover, $\chi_v$ is also the unique minimizer on $H_B^1(\Omega(v))$ of the functional $G_D(v)$ defined by
\begin{equation*}
G_D(v)[\vartheta] := \frac{1}{2} \int_{\Omega(v)} |\nabla\vartheta|^2\ \mathrm{d}(x,z) + \frac{1}{2} \int_D \sigma |\vartheta(\cdot,-H)|^2\ \mathrm{d}x - \int_{\Omega(v)} \vartheta\Delta h_v\ \mathrm{d}(x,z) \,.
\end{equation*}
\end{proposition}

\begin{proof}
As noted above, $\mathcal{G}(v)$ and $G_D(v)$ are both well-defined on $H_B^1(\Omega(v))$. Moreover, owing to the Poincar\'e inequality  established in \cite[Lemma~2.2]{LNW19}, the functional $\mathcal{G}(v)$ is coercive on $H_B^1(\Omega(v))$. It thus readily follows from the Lax-Milgram Theorem that there is a unique minimizer $\chi_v\in H_B^1(\Omega(v))$ of the functional $\mathcal{G}(v)$ on $H_B^1(\Omega(v))$. Let $\vartheta\in H_B^1(\Omega(v))$.  Since each connected component of $\Omega(v)$ has at most two singular points, we infer from \cite[Folgerung~7.5]{Ko63} that we may apply Gau\ss' Theorem on each connected component of $\Omega(v)$ and deduce from \eqref{bryanadams} that
\begin{align*}
\mathcal{G}(v)[\vartheta] & = \frac{1}{2} \int_{\Omega(v)} |\nabla \vartheta|^2 \,\mathrm{d}(x,z) +  \int_{\Omega(v)} \nabla \vartheta\cdot \nabla h_v \,\mathrm{d}(x,z)+ \frac{1}{2} \int_{\Omega(v)} |\nabla h_v|^2 \,\mathrm{d}(x,z) \\
& \quad + \frac{1}{2} \int_D \sigma |\vartheta(\cdot,-H)|^2 \,\mathrm{d}x + \int_D \sigma \vartheta(\cdot,-H) [h_v(\cdot,-H)-\mathfrak{h}_v]\,\mathrm{d}x \\
& \quad + \frac{1}{2} \int_D \sigma [h_v(\cdot,-H)-\mathfrak{h}_v]^2 \,\mathrm{d}x \\
& = G_D(v)[\vartheta] + \int_{\Omega(v)} \vartheta \Delta h_v\,\mathrm{d}(x,z) - \int_D (\vartheta\partial_z h_v)(x,-H)\,\mathrm{d}x -  \int_{\Omega(v)} \vartheta \Delta h_v\,\mathrm{d}(x,z) \\
& \quad + \frac{1}{2} \int_{\Omega(v)} |\nabla h_v|^2 \,\mathrm{d}(x,z) + \int_D \sigma \vartheta(\cdot,-H) [h_v(\cdot,-H)-\mathfrak{h}_v]\,\mathrm{d}x  \\
& \quad + \frac{1}{2} \int_D \sigma [h_v(\cdot,-H)-\mathfrak{h}_v]^2 \,\mathrm{d}x \\
& = G_D(v)[\vartheta] + \frac{1}{2} \int_{\Omega(v)} |\nabla h_v|^2 \,\mathrm{d} (x,z) + \frac{1}{2} \int_D \sigma [h_v(\cdot,-H)-\mathfrak{h}_v]^2 \,\mathrm{d}x \,.
\end{align*}	
Consequently, $\chi_v$ is also the unique minimizer of the functional $G_D(v)$ on $H_B^1(\Omega(v))$.
\end{proof}

For further use we state the following weak maximum principle.

\begin{lemma}\label{MP}
Let $v\in \bar S$. Then $h_v\in C(\overline{\Omega(v)})$ and $\mathfrak{h}_v\in C(\bar{D})$ and
\begin{equation*}
\min\Big\{\min_{\partial\Omega(v)} h_v\,,\, \min_{\bar D}\mathfrak{h}_v\Big\}\le \chi_v+h_v\le  
\max\Big\{\max_{ \partial\Omega(v)} h_v\,,\, \max_{\bar D}\mathfrak{h}_v\Big\}\,.
\end{equation*}
\end{lemma}

\begin{proof}
We first observe that $v\in C(\bar{D})$ which ensures, together with \eqref{zucchero}, that
\begin{equation*}
\mu_* := \min\Big\{ \min_{\partial{\Omega(v)}} h_v\,,\, \min_{\bar D}\mathfrak{h}_v\Big\} \;\;\text{ and }\;\; \mu^* := \max\Big\{\max_{\partial\Omega(v)} h_v\,,\, \max_{\bar D}\mathfrak{h}_v\Big\}
\end{equation*}
are well-defined and finite. Next, since $\chi_v$ is the minimizer of $\mathcal{G}(v)$ on $H_B^1(\Omega(v))$, it satisfies
\begin{equation}
\int_{\Omega(v)} \nabla(\chi_v+h_v)\cdot \nabla\vartheta \,\mathrm{d}(x,z) + \int_D \sigma [(\chi_v+h_v)(\cdot,-H) - \mathfrak{h}_v] \vartheta(\cdot,-H)\,\mathrm{d}x = 0 \label{yes1}
\end{equation}
for all $\vartheta\in H_B^1(\Omega(v))$. 

Now, it follows from the definition of $\mu^*$ that $\vartheta^* := (\chi_v + h_v - \mu^*)_+$ belongs to $H_B^1(\Omega(v))$ with $\nabla\vartheta^* = \mathrm{sign}_+ (\chi_v + h_v - \mu^*) \nabla (\chi_v + h_v - \mu^*)$. Consequently, by \eqref{yes1},
\begin{align*}
0 & = \int_{\Omega(v)} \nabla(\chi_v+h_v)\cdot \nabla\vartheta^* \,\mathrm{d}(x,z) + \int_D \sigma [(\chi_v+h_v)(\cdot,-H) - \mathfrak{h}_v] \vartheta^*(\cdot,-H) \,\mathrm{d}x \\
& = \int_{\Omega(v)} |\nabla\vartheta^*|^2 \,\mathrm{d}(x,z) + \int_D \sigma [(\chi_v+h_v)(\cdot,-H) - \mu^* + \mu^*- \mathfrak{h}_v] \vartheta^*(\cdot,-H) \,\mathrm{d}x \\
& \ge \int_{\Omega(v)} |\nabla\vartheta^*|^2 \,\mathrm{d}(x,z) + \int_D \sigma [\vartheta^*(\cdot,-H)]^2 \,\mathrm{d}x\,,
\end{align*}
where we have used the non-negativity of both $\mu^* - \mathfrak{h}_v$ and $\vartheta^*$ to derive the last inequality. We have thereby proved that $\nabla\vartheta^*=0$ in $L_2(\Omega(v))$, which implies that $\vartheta^*=0$ in $L_2(\Omega(v))$ thanks to the Poincar\'e inequality established in \cite[Lemma~2.2]{LNW19}. In other words, $\chi_v + h_v - \mu^*\le 0$ a.e. in $\Omega(v)$ as claimed. 

Finally, a similar argument with $\vartheta_* := (\mu_*-\chi_v-h_v)_+$ leads to the inequality $\mu_*-\chi_v-h_v\le 0$ a.e. in $\Omega(v)$ and completes the proof.
\end{proof}

We  now improve the regularity of $\chi_v$ as stated in Theorem~\ref{thmt1PP} and show that $\chi_v$ belongs to $H^2(\Omega(v))$. Once this is shown, it then readily follows that $\chi_v$ is a strong solution to \eqref{bbb} (see \cite[Theorem~3.5]{LNW19}). 

As pointed out previously, for a general $v\in \bar S$, the set $\Omega(v)$ may consist of several connected components without Lipschitz boundaries when the coincidence set $\mathcal{C}(v)$ is non-empty. The global $H^2(\Omega(v))$-regularity of $\chi_v$ is thus clearly not obvious.
The main idea is to  write the open set $D\setminus\mathcal{C}(v)$ as a countable union of disjoint open intervals  $(I_j)_{j\in J}$, see \cite[IX.Proposition~1.8]{AEIII}, and to establish  the $H^2$-regularity for $\chi_v$ first locally on each component $\left\{ (x,z)\in I_j\times\mathbb{R}\ :\ -H < z < v(x) \right\}$. This local regularity is performed in Section~\ref{Sec3.1}. The global $H^2(\Omega(v))$-regularity is subsequently established in Section~\ref{Sec3.2}.

\subsection{Local $H^2$-Regularity}\label{Sec3.1}

Let $I:=(a,b)$ be an open interval in $D$ and consider
\begin{equation}
v\in H^2(I) \;\text{ with }\; v(x)> -H\ , \quad x\in I\ . \label{t0P}
\end{equation}
We define the open set $\mathcal{O}_I(v)$ by
\begin{equation}
\mathcal{O}_I(v) := \left\{ (x,z)\in I\times\mathbb{R}\ :\ -H < z < v(x) \right\} \label{t3P}
\end{equation}
and split its boundary $\partial\mathcal{O}_I(v) = \partial \mathcal{O}_{I,D}(v) \cup \overline{\Sigma_I}$ with
\begin{align}
\partial\mathcal{O}_{I,D}(v) & := \big(\{a\} \times [-H,v(a)]\big) \cup \big(\{b\}\times [-H,v(b)]\big) \cup \overline{\mathfrak{G}_I(v)}\ , \label{t3aP} \\
\overline{\Sigma_I} & := [a,b]\times \{-H\}\ , \label{t3bP}
\end{align}
where $\Sigma_I := I \times \{-H\}$, and $\overline{\mathfrak{G}_I(v)}$ denotes the closure of the graph $\mathfrak{G}_I(v)$ of $v$, defined by
\begin{equation}
\mathfrak{G}_I(v) := \{ (x,v(x))\ :\ x\in I \}\ . \label{t4P}
\end{equation} 
We emphasize that $\mathcal{O}_I(v)$ has no Lipschitz boundary when $v(a)+H= \partial_x v(a)=0$ or $v(b)+H= \partial_x v(b)=0$, as these correspond to cuspidal boundary points, see Figure~\ref{F2}.

\begin{figure}
	\begin{tikzpicture}
	\draw[black, line width = 1.5pt, dashed] (-2.5,0)--(3.5,0);
	\draw[blue, line width = 2pt] (-2,0)--(3,0);
	\node at (-2,-0.5) {$a$};
	\draw[black, line width =1.5pt] (-2,-0.2)--(-2,0.2);
	\node at (3,-0.5) {$b$};
	\draw[black, line width =1.5pt] (3,-0.2)--(3,0.2);
	\draw[blue, line width = 2pt] plot[domain=-2:0] (\x,{(\x+2)^2*(2-\x)^2/8});
	\draw[blue, line width = 2pt] plot[domain=0:3] (\x,{2*(\x+3)^2*(3-\x)^2/81});
	\node at (0,1) {${\color{blue} \mathcal{O}_I(v)}$};
	\end{tikzpicture}
	\hspace{1cm}
	\begin{tikzpicture}
	\draw[black, line width = 1.5pt, dashed] (-2.5,0)--(2.5,0);
	\draw[blue, line width = 2pt] (-2,0)--(2,0);
	\draw[blue, line width = 2pt] (-2,0)--(-2,1);
	\draw[blue, line width = 2pt] (2,0)--(2,1);
	\node at (-2,-0.5) {$a$};
	\draw[black, line width =1.5pt] (-2,-0.2)--(-2,0.2);
	\node at (2,-0.5) {$b$};
	\draw[black, line width =1.5pt] (2,-0.2)--(2,0.2);
	\draw[blue, line width = 2pt] plot[domain=-2:2] (\x,{1+(\x+2)^2*(2-\x)^2/16});
	\node at (0.5,1) {${\color{blue} \mathcal{O}_I(v)}$};
	\end{tikzpicture}\\
	\vspace{0.5cm}
	\begin{tikzpicture}
	\draw[black, line width = 1.5pt, dashed] (-2.5,0)--(3.5,0);
	\draw[blue, line width = 2pt] (-2,0)--(3,0);
	\draw[blue, line width = 2pt] (-2,0)--(-2,1);
	\node at (-2,-0.5) {$a$};
	\draw[black, line width =1.5pt] (-2,-0.2)--(-2,0.2);
	\node at (3,-0.5) {$b$};
	\draw[black, line width =1.5pt] (3,-0.2)--(3,0.2);
	\draw[blue, line width = 2pt] plot[domain=-2:0] (\x,{1+(\x+2)^2*(2-\x)^2/12});
	\draw[blue, line width = 2pt] plot[domain=0:3] (\x,{7*(\x+3)^2*(3-\x)^2/243});
	\node at (0,1) {${\color{blue} \mathcal{O}_I(v)}$};
	\end{tikzpicture}
	\caption{Geometry of $\mathcal{O}_{I}(v)$ according to the boundary values of $v$.}\label{F2}
\end{figure}
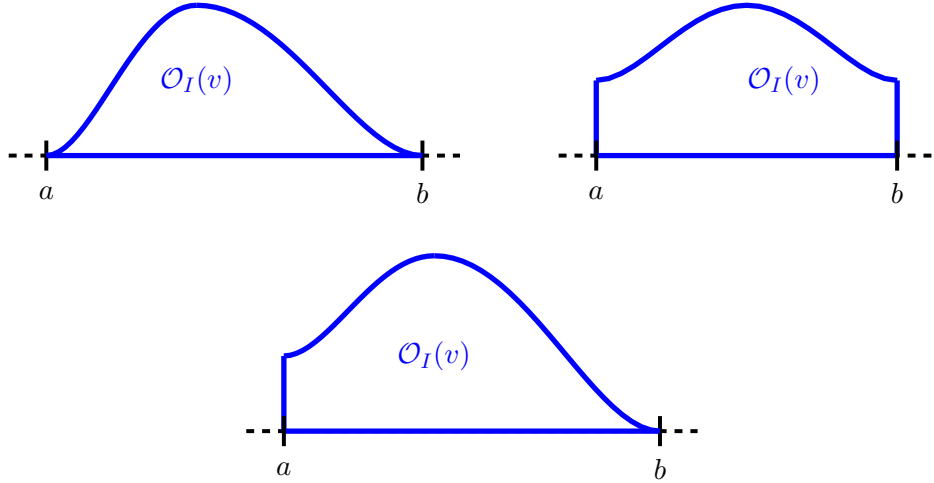

Let  $f\in L_2(\mathcal{O}_I(v))$ be a fixed function. The aim is to investigate the auxiliary problem
\begin{subequations}\label{t5P}
\begin{align}
-\Delta\zeta_v & = f \;\text{ in }\; \mathcal{O}_I(v)\ , \label{t5aP} \\
\zeta_v & = 0 \;\text{ on }\; \partial\mathcal{O}_{I,D}(v)\ , \label{t5bP} \\
-\partial_z \zeta_v + \sigma \zeta_v & = 0 \;\text{ on }\; \Sigma_I\ . \label{t5cP}
\end{align}
\end{subequations}
We shall show the existence and uniqueness of a variational solution $\zeta_v:=\zeta_{I,v}\in H^1(\mathcal{O}_I(v))$ to \eqref{t5P} and then prove its $H^2$-regularity. The main difficulty encountered here is the just mentioned possible lack of Lipschitz regularity of $\mathcal{O}_I(v)$. Indeed, the trace of functions in $H^1(\mathcal{O}_I(v))$ on $\partial\mathcal{O}_I(v)$ have no meaning yet in that case, and so \eqref{t5bP} and \eqref{t5cP} are not well-defined. We shall thus first give a precise meaning to traces for functions in $H^1(\mathcal{O}_I(v))$.

\begin{remark}
Clearly, if $v\in S$, $I=D$, and $f=h_v$, then $\chi_v=\zeta_{D,v}$, so that Theorem~\ref{thmt1PP} follows from Theorem~\ref{thmt1P} below in that case. Furthermore, if $I=(a,b)$ is a strict subinterval of $D$, $f=h_v$, and  $v\in \bar{S}$ is such that $v(a)=v(b) = -H$, or $a=-L$ and $v(-L)=v(b)+H=0$, or $b=L$ and $v(a)+H=v(L)=0$, then $\zeta_{I,v}$ coincides -- at least formally -- with the restriction of $\chi_v$ to $I$ and we shall also deduce Theorem~\ref{thmt1PP} from Theorem~\ref{thmt1P}. We thus do not impose that $v(a)=-H$ or $v(b)=-H$ in \eqref{t0P}, so as to be able to handle simultaneously the above mentioned different cases also depicted in Figure~\ref{F2}.
\end{remark}

\subsubsection{Traces}\label{sec.trP}

As already noticed in \cite{MNP00}, one can take advantage of the particular geometry of $\mathcal{O}_I(v)$, which lies between the graphs of two continuous functions, in order to define traces for functions in $H^1(\mathcal{O}_I(v))$ along these graphs. More precisely, one can derive the following result \cite[Lemma~2.1]{LNW19}.

\begin{lemma}[cite[Lemma~2.1]{LNW19}]\label{lemT1H} Assume that $v$ satisfies \eqref{t0P} and set $M_v := \|H+v\|_{L_\infty(I)}$.
	\begin{itemize}
		\item [{\bf (a)}] There is a linear bounded operator 
		\begin{equation*}
		\Gamma_{I,v} \in \mathcal{L}\left( H^1(\mathcal{O}_I(v)) , L_2(  I,(H+v)\mathrm{d}x) \right)
		\end{equation*}
		such that $\Gamma_{I,v} \vartheta = \vartheta(\cdot,v)$ for $\vartheta\in C^1(\overline{\mathcal{O}_I(v)})$ and
		\begin{equation}
		\int_I |\Gamma_{I,v} \vartheta|^2 (H+v)\ \mathrm{d}x \le \|\vartheta\|_{L_2(\mathcal{O}_I(v))}^2 + 2 M_v \|\vartheta\|_{L_2(\mathcal{O}_I(v))} \|\partial_z \vartheta\|_{L_2(\mathcal{O}_I(v))}\ . \label{t100P}
		\end{equation}
		\item [{\bf (b)}] There is a linear bounded operator 	
		\begin{equation*}
		\gamma_{I,v} \in \mathcal{L}\left( H^1(\mathcal{O}_I(v)) , L_2( I,(H+v)\mathrm{d}x) \right)
		\end{equation*}
		such that $\gamma_{I,v} \vartheta = \vartheta(\cdot,-H)$ for $\vartheta\in C^1(\overline{\mathcal{O}_I(v)})$ and
		\begin{equation}
		\int_I |\gamma_{I,v}\vartheta|^2 (H+v)\ \mathrm{d}x \le \|\vartheta\|_{L_2(\mathcal{O}_I(v))}^2 + 2 M_v \|\vartheta\|_{L_2(\mathcal{O}_I(v))} \|\partial_z \vartheta\|_{L_2(\mathcal{O}_I(v))}\ . \label{t200P}
		\end{equation}
	\end{itemize}
\end{lemma}

For simplicity, for $\vartheta\in H^1(\mathcal{O}_I(v))$, we use the notation
\begin{equation*}
\vartheta(x,v(x)) := \Gamma_{I,v}\vartheta(x)\ , \quad \vartheta(x,-H) := \gamma_{I,v}\vartheta(x)\ , \qquad x\in  I\ .
\end{equation*}

We next introduce the variational setting associated with \eqref{t5P} and define the space $H_B^1(\mathcal{O}_I(v))$ as the closure in $H^1(\mathcal{O}_I(v))$ of the set
\begin{equation*}
\begin{split}
C_B^1\left( \overline{\mathcal{O}_I(v)} \right) := \Big\{\theta\in C^1\left( \overline{\mathcal{O}_I(v)} \right) :\, & \, \theta(x,v(x))=0\,,\ x\in I\,,\\ 
& \text{ and }\theta(x,z)=0\,,\ (x,z)\in \{a,b\}\times (-H,0] \Big\}\,.
\end{split}
\end{equation*}
Note that this is consistent with the previous definition of $H_B^1(\Omega(v))$ when $I=D$ and $v\in \bar S$. We have already established in \cite[Lemma~2.2]{LNW19} a Poincar\'e inequality in $H_B^1(\mathcal{O}_I(v))$, as well as refined properties of the trace on $I\times \{-H\}$, which we recall now. 

\begin{lemma}[{\cite[Lemma~2.2]{LNW19}}] \label{lemT2H}
Assume that $v$ satisfies \eqref{t0P} and consider $\vartheta\in H_B^1(\mathcal{O}_I(v))$. Setting $M_v := \|H+v\|_{L_\infty(I)}$, there holds
	\begin{equation}
	\|\vartheta\|_{L_2(\mathcal{O}_I(v))} \le 2 M_v \|\partial_z \vartheta\|_{L_2(\mathcal{O}_I(v))}\,, \label{t300P}
	\end{equation}
	and the trace operator $\vartheta\mapsto \vartheta(\cdot,-H)$ maps $H_B^1(\mathcal{O}_I(v))$ to $L_2(I)$ with
	\begin{equation}
	\|\vartheta(\cdot,-H)\|_{L_2(I)}^2 \le 2  \|\vartheta\|_{L_2(\mathcal{O}_I(v))} \|\partial_z \vartheta\|_{L_2(\mathcal{O}_I(v))}\ . \label{t400P}
	\end{equation}
\end{lemma}

\subsubsection{Variational solution to \eqref{t5P}}\label{sec.trP2}
Thanks to Lemma~\ref{lemT2H}, the trace on $I\times \{-H\}$ of a function in $H_B^1(\mathcal{O}_I(v))$ is well-defined in $L_2(I)$ and, thus, so is the functional
\begin{equation}
G_I(v)[\vartheta] := \frac{1}{2} \int_{\mathcal{O}_I(v)} |\nabla\vartheta|^2\ \mathrm{d}(x,z) + \frac{1}{2} \int_I \sigma \vert\vartheta(\cdot,-H)\vert^2\ \mathrm{d}x - \int_{\mathcal{O}_I(v)} f \vartheta\ \mathrm{d}(x,z) \label{t101P}
\end{equation}
for $\vartheta\in H_B^1(\mathcal{O}_I(v))$. We now derive the existence of a unique variational solution to \eqref{t5P}, or, equivalently, of a unique minimizer of $G_I(v)$ on $\hb{\mathcal{O}_I(v)}$. 

\begin{lemma}\label{lemt0P}
	There is a unique variational solution $\zeta_v:= \zeta_{I,v}\in \hb{\mathcal{O}_I(v)}$ to \eqref{t5P} which satisfies
	\begin{equation}
	\|\zeta_v\|_{H^1(\mathcal{O}_I(v))}^2 + 2 \|\sqrt{\sigma} \zeta_v(\cdot,-H)\|_{L_2(I)}^2 \le 16 M_v^2 \left( 1 + 4 M_v^2 \right) \|f\|_{L_2(\mathcal{O}_I(v))}^2 \ , \label{h1P}
	\end{equation}
	where $M_v := \|H+v\|_{L_\infty(I)}$.
\end{lemma}

\begin{proof}
	It readily follows from \eqref{s0P}, Lemma~\ref{lemT2H}, and the Lax-Milgram Theorem that there is a unique variational solution $\zeta_v\in \hb{\mathcal{O}_I(v)}$ to \eqref{t5P} in the sense that
	\begin{equation}
	G_I(v)[\zeta_v] \le G_I(v)[\vartheta]\,, \qquad \vartheta\in \hb{\mathcal{O}_I(v)}\,. \label{t102P}
	\end{equation}
	Taking $\vartheta\equiv 0$ in the previous inequality, we deduce from \eqref{t300P} and H\"older's and Young's inequalities that
	\begin{align*}
	\|\nabla\zeta_v\|_{L_2(\mathcal{O}_I(v))}^2 + \|\sqrt{\sigma} \zeta_v(\cdot,-H)\|_{L_2(I)}^2 & \le 2 \|f\|_{L_2(\mathcal{O}_I(v))} \|\zeta_v\|_{L_2(\mathcal{O}_I(v))} \\
	& \le 4 M_v \|f\|_{L_2(\mathcal{O}_I(v))} \|\nabla\zeta_v\|_{L_2(\mathcal{O}_I(v))} \\
	& \le \frac{1}{2} \|\nabla\zeta_v\|_{L_2(\mathcal{O}_I(v))}^2 +  8 M_v^2 \|f\|_{L_2(\mathcal{O}_I(v))}^2\ .
	\end{align*}
	Hence,
	\begin{equation*}
	\|\nabla\zeta_v\|_{L_2(\mathcal{O}_I(v))}^2 + 2 \|\sqrt{\sigma} \zeta_v(\cdot,-H)\|_{L_2(I)}^2 \le  16 M_v^2 \|f\|_{L_2(\mathcal{O}_I(v))}^2\ .
	\end{equation*}
	Combining the Poincar\'e inequality \eqref{t300P} and the above inequality completes the proof.
\end{proof}

\subsubsection{$H^2$-regularity of $\zeta_v$}\label{sec.tP}

\newcounter{ncr}

We next investigate the regularity of the variational solution $\zeta_v$ to \eqref{t5P}; that is, we establish a local version of Theorem~\ref{thmt1PP}.

\begin{theorem}\label{thmt1P}
\refstepcounter{ncr}\label{cr0P} Consider a function $v$ satisfying \eqref{t0P} and let $\kappa>0$ be such that 
\begin{equation}
\|v\|_{H^2(I)} \le \kappa\ . \label{t2P} 
\end{equation} 
The variational solution $\zeta_v =\zeta_{I,v}\in \hb{\mathcal{O}_I(v)}$ to \eqref{t5P} given by Lemma~\ref{lemt0P} belongs to $H^2(\mathcal{O}_I(v))$, and there is $C_{\ref{cr0P}}(\kappa)>0$ depending only on $\sigma$ and $\kappa$ such that
\begin{equation}
\|\zeta_v\|_{H^2(\mathcal{O}_I(v))} + \|\partial_x \zeta_v(\cdot,-H)\|_{L_2(I)} \le C_{\ref{cr0P}}(\kappa) \|f\|_{L_2(\mathcal{O}_I(v))}\ . \label{t6P}
\end{equation}
\refstepcounter{ncr}\label{cr0rP} Moreover, there is $C_{\ref{cr0rP}}(\kappa) >0$ depending only on $\sigma$ and $\kappa$ such that, for any $r\in [2,\infty)$,
\begin{equation}
\|\partial_z\zeta_v(\cdot,v)\|_{L_r(I)} \le r C_{\ref{cr0rP}}(\kappa) \|f\|_{L_2(\mathcal{O}_I(v))}\ .  \label{t6rP}
\end{equation}
\end{theorem}

Several difficulties are encountered in the proof of Theorem~\ref{thmt1P}, due to the low regularity of the domain $\mathcal{O}_I(v)$ which has a Lipschitz boundary if $v(a)>-H$ and $v(b)>-H$ but may have cusps otherwise,  see Figure~\ref{F2}, and due to the mixed boundary conditions \eqref{t5bP} and \eqref{t5cP}. As in \cite[Section~3.3]{Gr1985}, to remedy these problems requires to construct suitable approximations of $\mathcal{O}_I(v)$ and to pay special attention in the derivation of functional inequalities and estimates on the dependence of the constants on $v$ and $I$. To be more precise, we shall begin with the case where $v$ satisfies
\begin{equation}
v\in W_\infty^3(I) \;\text{ and }\; \min_{[a,b]} v > -H\ , \label{t7P}
\end{equation}
an assumption which is obviously stronger than \eqref{t0P}. Then $\mathcal{O}_I(v)$ is a Lipschitz domain with a piecewise $W_\infty^3$-smooth boundary and the $H^2$-regularity of $\zeta_v$ is guaranteed by \cite[Theorem~2.2]{BR1991}, see Lemma~\ref{lemt1P} below. Next, transforming $\mathcal{O}_I(v)$ to the rectangle $\mathcal{R}_I := I\times (0,1)$, we shall adapt the proof of \cite[Lemma~4.3.1.3]{Gr1985} to establish the identity
\begin{equation}
\begin{split}
\int_{\mathcal{O}_I(v)} \partial_x^2\zeta_v \partial_z^2\zeta_v\ \mathrm{d}(x,z) & = \int_{\mathcal{O}_I(v)} |\partial_x\partial_z \zeta_v|^2\ \mathrm{d}(x,z) + \int_I \left( \partial_x \zeta_v \partial_x(\sigma\zeta_v) \right)(\cdot,-H)\ \mathrm{d}x \\& \qquad - \frac{1}{2} \int_I \partial_x^2 v |\partial_z\zeta_v(\cdot,v)|^2\ \mathrm{d}x
\end{split} \label{t8P}
\end{equation}
in Lemma~\ref{lemt2Px}. We then shall show that the last two integrals on the right-hand side of \eqref{t8P} are controlled by the $H^2$-norm of $\zeta_v$ with a sublinear dependence, a feature which will allow us to derive \eqref{t6P} when $v$ satisfies \eqref{t7P}. To this end, we shall use the embedding of the subspace
\begin{equation}
H_{WS}^1(\mathcal{O}_I(v)) := \left\{ P\in H^1(\mathcal{O}_I(v))\ :\
\begin{array}{cl} 
P(x,-H) & = 0\ , \qquad x\in I\ , \\
P(a,z) & = 0\ , \qquad z\in (-H,v(a)) \ ,
\end{array}
\right\} \label{t9P}
\end{equation}
of $H^1(\mathcal{O}_I(v))$ in $L_r(\mathcal{O}_I(v))$ and the continuity of the trace operator from $H_{WS}^1(\mathcal{O}_I(v))$ to $L_r(\mathfrak{G}_I(v))$ for $r\in [1,\infty)$, which involves constants that do not depend on $\min_{[a,b]}\{v+H\}$, see Lemmas~\ref{lemt4P}-\ref{lemt6P} in Appendix~\ref{A2}. After this preparation, we will be left with relaxing the assumption \eqref{t7P} to \eqref{t0P} and this will be achieved by an approximation argument, see Section~\ref{sec.t2P}. 

\subsubsection{$H^2$-regularity of $\zeta_v$ when $v$ satisfies \eqref{t7P}} \label{sec.t1P}

Throughout this section, we assume that $v$ satisfies \eqref{t7P} and fix $M>0$ such that
\begin{equation}
M \ge \max\left\{ 1 , \|H+v\|_{L_\infty(I)}, \| \partial_x v\|_{L_\infty(I)} \right\}\, . \label{tmP}
\end{equation}
 We also denote positive constants depending only on $\sigma$ by $C$ and $(C_i)_{i\ge 3}$. The dependence upon additional parameters will be indicated explicitly.

We begin with the $H^2$-regularity of the variational solution $\zeta_v$ to \eqref{t5P}, which follows from the analysis performed in \cite{Ba1992, BR1989, BR1991}.

\begin{lemma}\label{lemt1P}
$\zeta_v\in H^2(\mathcal{O}_I(v))$.
\end{lemma}

\begin{proof}
We first recast the boundary value problem \eqref{t5P} in the framework of \cite{BR1991}. Owing to \eqref{t7P}, the boundary of the domain $\mathcal{O}_I(v)$ includes four $W_\infty^3$-smooth edges $(\Gamma_i)_{1\le i \le 4}$ given by
\begin{align*}
\Gamma_1 := I \times \{-H\}\ , & \qquad \Gamma_3 := \mathfrak{G}_I(v)\ , \\
\Gamma_2 := \{b\}\times (-H,v(b))\ , & \qquad \Gamma_4 := \{a\}\times (-H,v(a))\ ,
\end{align*}
and four vertices $(S_i)_{1\le i\le 4}$ 
\begin{align*}
S_1 := \overline{\Gamma}_1 \cap \overline{\Gamma}_2 = (b,-H)\ , & \qquad S_3 := \overline{\Gamma}_3 \cap \overline{\Gamma}_4 = (a,v(a))\ , \\
S_2 := \overline{\Gamma}_2 \cap \overline{\Gamma}_3 = (b,v(b))\ , & \qquad S_4 := \overline{\Gamma}_4 \cap \overline{\Gamma}_1 = (a,-H)\ .
\end{align*}
We set
\begin{align*}
& \mathcal{D}_\Gamma := \{2,3,4\}\ , \quad \mathcal{N}_\Gamma := \{1\}\ , \\
& \mathcal{D} := \{2,3\}\ , \quad \mathcal{M}_{12} := \{4\}\ , \quad \mathcal{M}_{21} := \{1\}\ , \quad \mathcal{N} := \emptyset\ ,
\end{align*}
and note that $\mathcal{D}_\Gamma\ne \emptyset$ as required in \cite{BR1991}.

Since $v\in W_\infty^3(I)$, the measure $\omega_i$ of the angle at $S_i$ taken towards the interior of $\mathcal{O}_I(v)$ satisfies
\begin{equation}
\omega_1=\omega_4=\frac{\pi}{2}\ , \qquad (\omega_2,\omega_3)\in (0,\pi)^2\ . \label{t11P}
\end{equation}
For $1\le i\le 4$, we denote the outward unit normal vector field and the corresponding unit tangent vector field by $\boldsymbol{\nu}_i$ and $\boldsymbol{\tau}_i$, respectively. According to the geometry of $\mathcal{O}_I(v)$, 
\begin{align*}
\boldsymbol{\nu}_1=(0,-1)\ , \ \boldsymbol{\nu}_2=(1,0)\ , \ \boldsymbol{\nu}_3 = \frac{(- \partial_x v,1)}{\sqrt{1+|\partial_x v|^2}}\ , \ \boldsymbol{\nu}_4 = (-1,0)\ , \\
\boldsymbol{\tau}_1=(1,0)\ , \ \boldsymbol{\tau}_2=(0,1)\ , \ \boldsymbol{\tau}_3 = \frac{(-1, -\partial_x v)}{\sqrt{1+|\partial_x v|^2}}\ , \ \boldsymbol{\tau}_4 = (0,-1)\ . 
\end{align*}
We also define
\begin{equation}
\boldsymbol{\mu}_1 := \boldsymbol{\nu}_1\ , \qquad \boldsymbol{\mu}_i := \boldsymbol{\tau}_i\ , \ i\in\{2,3,4\}\ , \label{t12P}
\end{equation}
and note that the measure $\Psi_i\in [0,\pi]$ of the angle between $\boldsymbol{\mu}_i$ and $\boldsymbol{\tau}_i$, $1\le i \le 4$, is given by
\begin{equation}
\Psi_1 = \frac{\pi}{2}\ , \qquad \Psi_i=0\ , \ i\in\{2,3,4\}\ . \label{t13P}
\end{equation}
We also set 
\begin{equation}
\psi_1 = \phi_2 = \phi_3 = \phi_4 = 0\ . \label{t14P}
\end{equation}
We finally define the boundary operator
\begin{equation*}
\mathcal{B}_1 := - \partial_z + \sigma \mathrm{id} \;\text{ on }\; I \times  \{-H\}\,.
\end{equation*}

Now, on the one hand, the regularity of $\sigma$ implies that \cite[Assumption~(1.5)]{BR1991} is satisfied, while \cite[Assumption~(1.6)]{BR1991} obviously holds since $\mathcal{N}=\emptyset$. On the other hand, we note that $\boldsymbol{\mu}_1(S_1) = - \boldsymbol{\mu}_2(S_1)$ and $\boldsymbol{\mu}_4(S_4)=\boldsymbol{\mu}_1(S_4)$, so that \cite[Assumption~(2.1)]{BR1991} is satisfied for $i\in\{1,4\}$ (but not for $i\in\{2,3\}$). We then set $\varepsilon_1=-1$ and $\varepsilon_4=1$. We are left with checking \cite[Assumptions~(2.3)-(2.4)]{BR1991} but this is obvious due to \eqref{t14P}. We finally observe that
\begin{equation*}
\mathcal{K} := \left\{ (i,m)\in \{1,\ldots,4\}\times \mathbb{Z}\ :\ \lambda_{i,m}\in (-1,0) \right\}
\end{equation*}
is empty, since
\begin{align*}
\lambda_{1,m} & := \frac{\Psi_2-\Psi_1+m\pi}{\omega_1} = 2m-1 \not\in (-1,0)\ , \\
\lambda_{2,m} & := \frac{\Psi_3-\Psi_2+m\pi}{\omega_2} = \frac{m\pi}{\omega_2}\not\in (-1,0)\ , \\
\lambda_{3,m} & := \frac{\Psi_4-\Psi_3+m\pi}{\omega_3} = \frac{m\pi}{\omega_3} \not\in (-1,0)\ , \\
\lambda_{4,m} & := \frac{\Psi_1-\Psi_4+m\pi}{\omega_4} = 2m+1 \not\in (-1,0)\ , 
\end{align*}
for  any $m\in\mathbb{Z}$. We then infer from \cite[Theorem~2.2]{BR1991} that $\zeta_v$ has no singular part and thus belongs to $H^2(\mathcal{O}_I(v))$.
\end{proof}

We now investigate the quantitative dependence of the just established $H^2$-regularity of $\zeta_v$  on $v$ and derive an $H^2$-estimate, which is related to the regularity of $v$. To this end, we need the following identity.

\begin{lemma}\label{lemt2Px}
\begin{align*}
\int_{\mathcal{O}_I(v)} \partial_x^2\zeta_v \partial_z^2\zeta_v\ \mathrm{d}(x,z) & = \int_{\mathcal{O}_I(v)} |\partial_x\partial_z \zeta_v|^2\ \mathrm{d}(x,z) + \int_I \left( \partial_x \zeta_v \partial_x(\sigma\zeta_v) \right)(\cdot,-H)\ \mathrm{d}x \\
& \qquad - \frac{1}{2} \int_I \partial_x^2 v |\partial_z\zeta_v(\cdot,v)|^2\ \mathrm{d}x\ .
\end{align*} 
\end{lemma}

The identity of Lemma~\ref{lemt2Px} is reminiscent of \cite[Lemma 3.5]{LW19}.
Its proof is rather technical and thus postponed to Appendix~\ref{A1}.

The next step of the analysis is to show that the two integrals over $I$ on the right-hand side of the identity stated in Lemma~\ref{lemt2Px} can be controlled by the $H^2$-norm of $\zeta_v$ with a mild dependence on $v$. To this end, we need some auxiliary functional and trace inequalities which are established in Appendix~\ref{A2}. With this in hand, we begin with an estimate of the last integral.

\begin{lemma}\label{lemt7P}
\refstepcounter{ncr}\label{cr1P} There is $C_{\ref{cr1P}}(M)>0$ such that, for any $r\in [2,\infty)$,
\begin{equation}
\|\partial_z\zeta_v(\cdot,v)\|_{L_r(I)} \le r C_{\ref{cr1P}}(M) \|f\|_{L_2(\mathcal{O}_I(v))}^{1/r} \left( \|\nabla\partial_z\zeta_v\|_{L_2(\mathcal{O}_I(v))} + \|f\|_{L_2(\mathcal{O}_I(v))} \right)^{(r-1)/r}\ . \label{z3P}
\end{equation}
\refstepcounter{ncr}\label{cr2P} In particular, there is $C_{\ref{cr2P}}(M)>0$ such that
\begin{equation}
\left| \int_I \partial_x^2 v |\partial_z\zeta_v(\cdot,v)|^2\ \mathrm{d}x \right| \le C_{\ref{cr2P}}(M) \|\partial_x^2 v\|_{L_2(I)} \left[ \|f \|_{L_2(\mathcal{O}_I(v))}^{1/2} \|\nabla \partial_z \zeta_v\|_{L_2(\mathcal{O}_I(v))}^{3/2} + \|f\|_{L_2(\mathcal{O}_I(v))}^2 \right]\ . \label{z4P}
\end{equation}
\end{lemma}

\begin{proof}
To lighten notation, we set $\mathcal{O} := \mathcal{O}_I(v)$ and introduce $P:= \partial_z \zeta_v - \sigma \zeta_v$. Since $\zeta_v\in H^2(\mathcal{O})$ by Lemma~\ref{lemt1P} and $\sigma\in C^2(\bar{I})$, the function $P$ belongs to $H^1(\mathcal{O})$ and satisfies \eqref{t41P} by \eqref{t5bP} and \eqref{t5cP}. In addition, we observe that $P(\cdot,v)=\partial_z \zeta_v(\cdot,v)$ by \eqref{t5bP}. It then follows from Lemma~\ref{lemt6P} that
\begin{equation*}
\|\partial_z\zeta_v(\cdot,v)\|_{L_r(I)}^r = \|P(\cdot,v)\|_{L_r(I)}^r \le  \left( 4r\sqrt{M} \right)^r \|P\|_{L_2(\mathcal{O})} \|\nabla P\|_{L_2(\mathcal{O})}^{r-1}\ .
\end{equation*}
Moreover, by \eqref{s0P} and Lemma~\ref{lemt0P},
\begin{align*}
\|P\|_{L_2(\mathcal{O})} & \le \|\partial_z\zeta_v\|_{L_2(\mathcal{O})} + \bar{\sigma} \|\zeta_v\|_{L_2(\mathcal{O})} \le \left( 1 + \bar{\sigma} \right) \|\zeta_v\|_{H^1(\mathcal{O})} \\
& \le 4 \|H+v\|_{L_\infty(I)} \sqrt{1+4\|H+v\|_{L_\infty(I)}^2} \left( 1 + \bar{\sigma} \right) \|f\|_{L_2(\mathcal{O})} \le C(M) \|f\|_{L_2(\mathcal{O})}
\end{align*}
and
\begin{align*}
\|\nabla P\|_{L_2(\mathcal{O})} & \le \|\partial_x P\|_{L_2(\mathcal{O})} + \|\partial_z P\|_{L_2(\mathcal{O})} \\
& \le \|\partial_x\partial_z \zeta_v\|_{L_2(\mathcal{O})} + \bar{\sigma} \|\partial_x \zeta_v\|_{L_2(\mathcal{O})} + \bar{\sigma} \|\zeta_v\|_{L_2(\mathcal{O})} + \|\partial_z^2 \zeta_v\|_{L_2(\mathcal{O})} + \bar{\sigma} \|\partial_z \zeta_v\|_{L_2(\mathcal{O})} \\
& \le \sqrt{2} \|\nabla\partial_z \zeta_v\|_{L_2(\mathcal{O})} + \bar{\sigma} \left( \sqrt{2} \|\nabla \zeta_v\|_{L_2(\mathcal{O})} +  \|\zeta_v\|_{L_2(\mathcal{O})} \right) \\
& \le \sqrt{2} \|\nabla\partial_z \zeta_v\|_{L_2(\mathcal{O})} +  C(M) \|f\|_{L_2(\mathcal{O})}\ .
\end{align*}
Collecting the previous estimates, we end up with
\begin{align*}
\|\partial_z\zeta_v(\cdot,v)\|_{L_r(I)}^r  & \le \left( 4r\sqrt{M} \right)^r C(M) \|f\|_{L_2(\mathcal{O})} \left( \sqrt{2} \|\nabla\partial_z \zeta_v\|_{L_2(\mathcal{O})} +  C(M) \|f\|_{L_2(\mathcal{O})} \right)^{r-1} \\
& \le (rC(M))^r \|f\|_{L_2(\mathcal{O})} \left( \|\nabla\partial_z \zeta_v\|_{L_2(\mathcal{O})} + \|f\|_{L_2(\mathcal{O})} \right)^{r-1}\ ,
\end{align*}
from which \eqref{z3P} follows. We next deduce from \eqref{z3P} (with $r=4$) and H\"older's inequality that
\begin{align*}
\left| \int_I \partial_x^2 v |\partial_z\zeta_v(\cdot,v)|^2\ \mathrm{d}x \right| & \le \|\partial_x^2 v\|_{L_2(I)} \|\partial_z\zeta_v(\cdot,v)\|_{L_4(I)}^2 \\
& \le 16 C_{\ref{cr1P}}(M)^2  \|\partial_x^2 v\|_{L_2(I)} \|f\|_{L_2(\mathcal{O})}^{1/2} \left( \|\nabla\partial_z \zeta_v\|_{L_2(\mathcal{O})} + \|f\|_{L_2(\mathcal{O})} \right)^{3/2} \\
& \le C(M) \|\partial_x^2 v\|_{L_2(I)} \|f\|_{L_2(\mathcal{O})}^{1/2} \left( \|\nabla\partial_z \zeta_v\|_{L_2(\mathcal{O})}^{3/2} + \|f\|_{L_2(\mathcal{O})}^{3/2} \right)\ ,
\end{align*}
and the proof is complete. 
\end{proof}

We are now in a position to derive quantitative estimates in $H^2$ for $\zeta_v$, which only depends on the $H^2$-norm of $v$, even though $v$ is assumed to be more regular.

\begin{lemma}\label{lemt8P}
\refstepcounter{ncr}\label{cr3P} There is $C_{\ref{cr3P}}(M)>0$ such that
\begin{subequations}\label{t81P}
\begin{align}
\|\nabla \partial_z \zeta_v\|_{L_2(\mathcal{O}_I(v))}^2 + \|\sqrt{\sigma}\partial_x\zeta_v(\cdot,-H)\|_{L_2(I)}^2 & \le C_{\ref{cr3P}}(M) \left( 1 + \|\partial_x^2 v\|_{L_2(I)}^4 \right) \|f \|_{L_2(\mathcal{O}_I(v))}^2\ ,  \label{t81aP} \\
\|\partial_x^2 \zeta_v \|_{L_2(\mathcal{O}_I(v))}^2 & \le C_{\ref{cr3P}}(M) \left( 1 + \|\partial_x^2 v\|_{L_2(I)}^4 \right) \|f\|_{L_2(\mathcal{O}_I(v))}^2\ . \label{t81bP}
\end{align}
\end{subequations}
\end{lemma}

\begin{proof}
To lighten notation, we set $\mathcal{O} := \mathcal{O}_I(v)$. We infer from \eqref{t5aP} and Lemma~\ref{lemt2Px} that
\begin{align*}
- \int_{\mathcal{O}} f\partial_z^2\zeta_v\ \mathrm{d}(x,z) & = \int_{\mathcal{O}} \left( \partial_x^2\zeta_v \partial_z^2\zeta_v + |\partial_z^2\zeta_v|^2 \right)\ \mathrm{d}(x,z) \\
& = \|\nabla\partial_z\zeta_v\|_{L_2(\mathcal{O})}^2 + \int_I \partial_x\zeta_v(\cdot,-H) \partial_x(\sigma\zeta_v)(\cdot,-H)\ \mathrm{d}x \\
& \qquad - \frac{1}{2} \int_I \partial_x^2 v |\partial_z\zeta_v(\cdot,v)|^2\ \mathrm{d}x\ .
\end{align*}
Hence, thanks to \eqref{s0P},  Lemma~\ref{lemt7P}, and H\"older's and Young's inequalities,
\begin{align*}
X & := \|\nabla\partial_z\zeta_v\|_{L_2(\mathcal{O})}^2 + \|\sqrt{\sigma}\partial_x\zeta_v(\cdot,-H)\|_{L_2(I)}^2 \\
& = - \int_{\mathcal{O}} f\partial_z^2\zeta_v\ \mathrm{d}(x,z) - \int_I \partial_x \sigma (\zeta_v\partial_x\zeta_v)(\cdot,-H)\ \mathrm{d}x + \frac{1}{2} \int_I \partial_x^2 v |\partial_z\zeta_v(\cdot,v)|^2\ \mathrm{d}x \\
& \le \|f\|_{L_2(\mathcal{O})} \|\partial_z^2\zeta_v\|_{L_2(\mathcal{O})} + \bar{\sigma} \|\zeta_v(\cdot,-H)\|_{L_2(I)} \|\partial_x\zeta_v(\cdot,-H)\|_{L_2(I)} \\
& \qquad + \frac{C_{\ref{cr2P}}(M)}{2} \|\partial_x^2 v\|_{L_2(I)} \left[ \|f \|_{L_2( \mathcal{O})}^{1/2} \|\nabla \partial_z \zeta_v\|_{L_2( \mathcal{O})}^{3/2} + \|f\|_{L_2(\mathcal{O})}^2 \right] \\
& \le \frac{1}{4} \|\partial_z^2\zeta_v\|_{L_2(\mathcal{O})}^2 + \|f\|_{L_2(\mathcal{O})}^2 + \frac{\bar{\sigma}}{\sqrt{\sigma_{min}}} \|\zeta_v(\cdot,-H)\|_{L_2(I)} \|\sqrt{\sigma}\partial_x\zeta_v(\cdot,-H)\|_{L_2(I)} \\
& \qquad + \frac{1}{4} \|\nabla\partial_z\zeta_v\|_{L_2(\mathcal{O})}^2 + C(M) \left( \|\partial_x^2 v\|_{L_2(I)}^4 + \|\partial_x^2 v\|_{L_2(I)} \right) \|f\|_{L_2(\mathcal{O})}^2 \\
& \le \frac{1}{2} \|\nabla\partial_z\zeta_v\|_{L_2(\mathcal{O})}^2 + \frac{1}{2} \|\sqrt{\sigma}\partial_x\zeta_v(\cdot,-H)\|_{L_2(I)}^2 + \frac{\bar{\sigma}^2}{2\sigma_{min}} \|\zeta_v(\cdot,-H)\|_{L_2(I)}^2 \\
& \qquad + C(M)  \left( 1 + \|\partial_x^2 v\|_{L_2(I)}^4 \right) \|f\|_{L_2(\mathcal{O})}^2\ .
\end{align*}
Consequently, using once more Young's inequality,
\begin{equation*}
X \le \frac{\bar{\sigma}^2}{\sigma_{min}} \|\zeta_v(\cdot,-H)\|_{L_2(I)}^2 + C(M) \left( 1 + \|\partial_x^2 v\|_{L_2(I)}^4 \right) \|f\|_{L_2(\mathcal{O})}^2\ .
\end{equation*}
Now, since $\zeta_v\in H^1_B(\mathcal{O})$, it follows from \eqref{s0P}, \eqref{tmP}, and Lemma~\ref{lemt0P} that
\begin{equation*}
 2 \sigma_{min} \|\zeta_v(\cdot,-H)\|_{L_2(I)}^2 \le  16 M^2 (1+4M^2) \|f\|_{L_2(\mathcal{O})}^2\ .
\end{equation*}
Combining the above two estimates gives \eqref{t81aP}. 

To complete the proof of Lemma~\ref{lemt8P}, we simply notice that \eqref{t5aP} ensures that
\begin{equation*}
\|\partial_x^2\zeta_v\|_{L_2(\mathcal{O})}^2 = \|f + \partial_z^2\zeta_v\|_{L_2(\mathcal{O})}^2 \le 2 \|\partial_z^2\zeta_v\|_{L_2(\mathcal{O})}^2 + 2 \|f\|_{L_2(\mathcal{O})}^2
\end{equation*}
and deduce \eqref{t81bP} from \eqref{t81aP}.
\end{proof}

Summarizing, we have established the following result:

\begin{proposition}\label{propt9b}
Consider $v\in H^2(I)$ satisfying \eqref{t7P}; that is, 
\begin{equation*}
v\in W_\infty^3(I) \;\text{ and }\; \min_{[a,b]} v > -H\, , 
\end{equation*}
 and fix $\kappa>0$ such that
\begin{equation}
\|v\|_{H^2(I)}\le \kappa\ . \label{t83P}
\end{equation}
Then the elliptic boundary value problem \eqref{t5P} has a unique strong solution $\zeta_v\in H^2(\mathcal{O}_I(v))$ which satisfies \refstepcounter{ncr}\label{cr4P} 
\begin{align}
\|\zeta_v\|_{H^2(\mathcal{O}_I(v))} + \|\partial_x\zeta_v(\cdot,-H)\|_{L_2(I)} & \le C_{\ref{cr4P}}(\kappa) \|f\|_{L_2(\mathcal{O}_I(v))}\ , \label{t82aP} \\
\|\partial_z \zeta_v(\cdot,v)\|_{L_r(I)} & \le r C_{\ref{cr4P}}(\kappa) \|f\|_{L_2(\mathcal{O}_I(v))} \ , \qquad r\in [2,\infty)\ . \label{t82bP} 
\end{align}
\end{proposition}

\begin{proof}
The existence and uniqueness of a strong solution $\zeta_v\in H^2(\mathcal{O}_I(v))$ to \eqref{t5P} are consequences of Lemma~\ref{lemt0P} and Lemma~\ref{lemt1P}. Next, it readily follows from \eqref{t83P} and the continuous embedding of $H^2(I)$ in $W_\infty^1(I)$ that there is $M\ge 1$ depending on $\kappa$ such that
\begin{equation}
\|H+v\|_{L_\infty(I)} + \|\partial_x v\|_{L_\infty(I)} \le M\ . \label{t84P}
\end{equation}
Due to \eqref{t84P}, we deduce \eqref{t82aP} from \eqref{s0P}, \eqref{t83P}, Lemma~\ref{lemt0P}, and Lemma~\ref{lemt8P}, while \eqref{t82bP} follows from \eqref{t82aP} and Lemma~\ref{lemt7P}.
\end{proof}

 We emphasize that, though derived for functions $v\in H^2(I)$ satisfying the additional assumption \eqref{t7P}, the estimates stated in Proposition~\ref{propt9b} only depend on the $H^2$-norm of $v$ and, neither on its $W_\infty^2$-norm, nor on the value of its minimum (provided that it stays above $-H$). The outcome of Proposition~\ref{propt9b} is thus likely to extend to any configuration depicted in Figure~\ref{F2} under the sole assumption \eqref{t0P} and this will be shown in the next section by an approximation argument.

\subsubsection{$H^2$-regularity: Proof of Theorem~\ref{thmt1P}} \label{sec.t2P}

We now prove the  $H^2$-regularity of $\zeta_v$ as stated in Theorem~\ref{thmt1P}. We thus assume that $v$ satisfies \eqref{t0P}; that is, 
\begin{equation*}
v\in H^2(I) \;\text{ such that }\; v(x)>-H\, , \qquad x\in I\, , 
\end{equation*}
and fix $\kappa>0$ such that $\|v\|_{H^2(I)}\le \kappa$. Owing to the density of $C^\infty([a,b])$ in $H^2(I)$ and since $v$ satisfies \eqref{t0P}, we employ classical approximation arguments to construct a sequence $(v_n)_{n\ge 1}$ of functions in $C^\infty([a,b])$ with the following properties:
\begin{subequations}\label{s1P}
	\begin{align}
	& \lim_{n\to \infty} \|v_n - v\|_{H^2(I)} = 0\ , \qquad \sup_{n\ge 1}\{\|v_n\|_{H^2(I)}\} \le 1+\kappa\ , \label{s1Pa}\\
	&  v_n\ge v + \frac{1}{n}\ , \qquad n\ge 1\ . \label{s1Pb}
	\end{align}
\end{subequations}
A first consequence of \eqref{s1Pa} and the continuous embedding of $H^2(I)$ in $ W_\infty^1(I)$ is that
\begin{equation}
\begin{split}
& \|H+v_n\|_{L_\infty(I)} + \| \partial_x v_n\|_{L_\infty(I)} \le C(\kappa)\ , \qquad n\ge 1\ , \\
&  \lim_{n\to \infty} \|v_n - v\|_{W_\infty^1(I)} = 0 \ .
\end{split} \label{s1Pc}
\end{equation}
According to \eqref{t0P} and \eqref{s1Pb}, the function $v_n$ satisfies \eqref{t7P} for each $n\ge 1$ and, since $\mathcal{O}_I(v)\subset \mathcal{O}_I(v_n)$, we infer from Proposition~\ref{propt9b} that the strong solution $\zeta_{v_n}$ to \eqref{t5P} with $v_n$ instead of $v$ (and $f$ replaced by its trivial extension to $\mathcal{O}_I(v_n)$) satisfies
\refstepcounter{ncr}\label{cr5P} 
\begin{align}
\|\zeta_{v_n}\|_{H^2(\mathcal{O}_I(v_n))} + \|\partial_x\zeta_{v_n}(\cdot,-H)\|_{L_2(I)} & \le C_{\ref{cr5P}}(\kappa) \|f\|_{L_2(\mathcal{O}_I(v))}\ , \label{s4P} \\
\|\partial_z \zeta_{v_n}(\cdot,v_n)\|_{L_r(I)} & \le r C_{\ref{cr5P}}(\kappa) \|f\|_{L_2(\mathcal{O}_I(v))} \ , \qquad r\in [2,\infty)\ . \label{s5P} 
\end{align}
Using again the inclusion $\mathcal{O}_I(v)\subset \mathcal{O}_I(v_n)$, we deduce from \eqref{s4P} that $(\zeta_{v_n})_{n\ge 1}$ is bounded in $H^2(\mathcal{O}_I(v))$. Consequently, recalling that $H^1(\mathcal{O}_I(v))$ is compactly embedded in $L_2(\mathcal{O}_I(v))$ (despite the non-Lipschitz character of $\mathcal{O}_I(v)$, see \cite[Theorem~11.21]{Leoni17} or \cite[I.Theorem~1.4]{Necas67}), there are a subsequence of $(\zeta_{v_n})_{n\ge 1}$ (not relabeled) and $\phi\in H^2(\mathcal{O}_I(v))$ such that
\begin{equation}
\begin{split}
\zeta_{v_n} \rightharpoonup \phi \;\;\text{ in }\;\; H^2(\mathcal{O}_I(v)) \ , \\
\zeta_{v_n} \longrightarrow \phi \;\;\text{ in }\;\; H^1(\mathcal{O}_I(v)) \ . \\
\end{split} \label{s6P}
\end{equation}
Let us first check that $\phi\in H_B^1(\mathcal{O}_I(v))$. On the one hand, since both $\phi$ and $\zeta_{v_n}$ belong to $H^1(\mathcal{O}_I(v))$, we infer from \eqref{t100P} that
\begin{equation*}
\int_I \left| (\phi-\zeta_{v_n})(\cdot,v) \right|^2 (H+v)\ \mathrm{d}x \le C(\kappa) \|\phi-\zeta_{v_n}\|_{H^1(\mathcal{O}_I(v))}^2\ .
\end{equation*}
Hence, by \eqref{s6P},
\begin{equation*}
\lim_{n\to\infty} \int_I \left| (\phi-\zeta_{v_n})(\cdot,v) \right|^2 (H+v)\ \mathrm{d}x = 0\ .
\end{equation*}
On the other hand, since $\zeta_{v_n}\in H_B^1(\mathcal{O}_I(v_n))$ and $v_n\ge v$, it follows from Lemma~\ref{lemA1} and \eqref{s4P} that
\begin{align*}
\int_I \left| \zeta_{v_n}(\cdot,v)\right|^2 (H+v)\ \mathrm{d}x & = \int_I \left| \zeta_{v_n}(\cdot,v) - \zeta_{v_n}(\cdot,v_n) \right|^2 (H+v)\ \mathrm{d}x \\ 
& \le \|(v-v_n) (H+v)\|_{L_\infty(I)} \|\partial_z \zeta_{v_n} \|_{L_2(\mathcal{O}_I(v_n))} \\
& \le C(\kappa) \|v-v_n\|_{L_\infty(I)} \|f\|_{L_2(\mathcal{O}_I(v))}\ . 
\end{align*}
Hence, by \eqref{s1Pc},
\begin{equation*}
\lim_{n\to \infty} \int_I \left| \zeta_{v_n}(\cdot,v)\right|^2 (H+v)\ \mathrm{d}x = 0\ .
\end{equation*}
Combining the previous two limits, we deduce 
\begin{equation*}
\int_I \left| \phi(\cdot,v) \right|^2 (H+v)\ \mathrm{d}x = 0\ ,
\end{equation*}
so that $\phi\in H_B^1(\mathcal{O}_I(v))$. In particular, for $n\ge 1$, due to the inclusion $\mathcal{O}_I(v)\subset \mathcal{O}_I(v_n)$, the function $\phi$ also belongs to $H_B^1(\mathcal{O}_I(v_n))$ and we infer from \eqref{t400P} and \eqref{s6P} that 
\begin{equation}
\lim_{n\to\infty} \int_I \left| (\zeta_{v_n}-\phi)(\cdot,-H) \right|^2 \rd x = 0\ . \label{s7P}
\end{equation}
We next recall that $\zeta_{v_n}$ is the unique solution in $H_B^1(\mathcal{O}_I(v_n))$ to 
\begin{equation}
\int_{\mathcal{O}_I(v_n)} \nabla \zeta_{v_n}\cdot \nabla \vartheta\ \mathrm{d}(x,z) + \int_I  \sigma \zeta_{v_n}(\cdot,-H)  \vartheta(\cdot,-H)\ \mathrm{d}x= \int_{\mathcal{O}_I(v_n)} f \vartheta\ \mathrm{d}x \label{s8P}
\end{equation}
for all $\vartheta\in H_B^1(\mathcal{O}_I(v_n))$. Now, since $H_B^1(\mathcal{O}_I(v)) \subset H_B^1(\mathcal{O}_I(v_n))$, we can take $\vartheta\in H_B^1(\mathcal{O}_I(v))$ in \eqref{s8P} and use the convergences \eqref{s6P} and \eqref{s7P} to pass to the limit $n\to\infty$ and conclude that $\phi\in H_B^1(\mathcal{O}_I(v))$ satisfies the variational formulation of \eqref{t5P}. Therefore, Lemma~\ref{lemt0P} guarantees that $\phi=\zeta_v$. We have thus shown that $\zeta_v\in H^2(\mathcal{O}_I(v))$ and it follows from \eqref{s4P} and \eqref{s6P} that
\begin{equation}
\|\zeta_{v}\|_{H^2(\mathcal{O}_I(v))} \le \liminf_{n\to\infty} \|\zeta_{v_n}\|_{H^2(\mathcal{O}_I(v))} \le  \liminf_{n\to\infty} \|\zeta_{v_n}\|_{H^2(\mathcal{O}_I(v_n))}  \le C_{\ref{cr5P}}(\kappa) \|f\|_{L_2(\mathcal{O}_I(v))}\ . \label{s10P}
\end{equation}
A further consequence of \eqref{t200P} and \eqref{s6P} is that $(\partial_x\zeta_{v_n}(\cdot,-H))_{n\ge 1}$ converges to $\partial_x\zeta_{v}(\cdot,-H)$ in $L_2(I,(H+v)\mathrm{d}x)$, which, together with the positivity of $H+v$ in $I$, implies that $(\partial_x\zeta_{v_n}(\cdot,-H))_{n\ge 1}$ converges to $\partial_x\zeta_{v}(\cdot,-H)$ in $L_2(a+\varepsilon,b-\varepsilon)$ for any $\varepsilon\in (0,(b-a)/2)$. Combining this convergence with \eqref{s4P} and using Fatou's lemma to take the limit $\varepsilon\to 0$ give
\begin{equation}
\|\partial_x\zeta_{v}(\cdot,-H)\|_{L_2(I)} \le C_{\ref{cr5P}}(\kappa) \|f\|_{L_2(\mathcal{O}_I(v))}\ . \label{s11P}
\end{equation}
Finally, by \eqref{t100P} and \eqref{s4P},
\begin{equation*}
\int_I \left| (\partial_z\zeta_{v_n} - \partial_z \zeta_v)(\cdot,v) \right|^2 (H+v)\ \mathrm{d}x \le C(\kappa) \|\partial_z\zeta_{v_n} - \partial_z \zeta_v\|_{L_2(\mathcal{O}_I(v))}\ .
\end{equation*}
Hence, by \eqref{s6P},
\begin{equation}
\lim_{n\to\infty} \int_I \left| (\partial_z\zeta_{v_n} - \partial_z \zeta_v)(\cdot,v) \right|^2 (H+v)\ \mathrm{d}x = 0\ . \label{s12P}
\end{equation}
Moreover, owing to Lemma~\ref{lemA1}, \eqref{s4P}, and the properties $\zeta_{v_n}\in H_B^1(\mathcal{O}_I(v_n))$ and $v_n\ge v$, 
\begin{align*}
\int_I \left| \partial_z\zeta_{v_n}(\cdot,v) - \partial_z\zeta_{v_n}(\cdot,v_n) \right|^2 (H+v)\ \mathrm{d}x & \le \|(v-v_n)(H+v)\|_{L_\infty(I)} \|\partial_z^2 \zeta_{v_n} \|_{L_2(\mathcal{O}_I(v_n))}^{2} \\
& \le C(\kappa) \|v-v_n\|_{L_\infty(I)}\ ,
\end{align*}
and it follows from \eqref{s1Pc} that 
\begin{equation}
\lim_{n\to\infty} \int_I \left| \partial_z\zeta_{v_n}(\cdot,v) - \partial_z\zeta_{v_n}(\cdot,v_n) \right|^2 (H+v)\ \mathrm{d}x = 0\ . \label{s13P}
\end{equation}
Gathering \eqref{s12P} and \eqref{s13P} leads us to
\begin{equation}
\lim_{n\to\infty} \int_I \left| \partial_z\zeta_{v}(\cdot,v) - \partial_z\zeta_{v_n}(\cdot,v_n) \right|^2 (H+v)\ \mathrm{d}x = 0\ . \label{s14P}
\end{equation}
Since $H+v>0$ in $I$, we may extract a further subsequence (not relabeled) such that $(\partial_z\zeta_{v_n}(\cdot,v_n))_{n\ge 1}$ converges a.e. in $I$ to $\partial_z\zeta_{v}(\cdot,v)$. We then use Fatou's lemma to pass to the limit $n\to\infty$ in \eqref{s5P} and conclude that
\begin{equation*}
\|\partial_z \zeta_{v}(\cdot,v)\|_{L_r(I)} \le r C_{\ref{cr5P}}(\kappa) \|f\|_{L_2(\mathcal{O}_I(v))} \ , \qquad r\in [2,\infty)\ ,
\end{equation*}
thereby completing the proof of Theorem~\ref{thmt1P}.

\subsection{Global $H^2$-regularity  of $\chi_v$: Proof of Theorem~\ref{thmt1PP} and Theorem~\ref{Thmpsi}}\label{Sec3.2}

 Finally, we prove Theorem~\ref{thmt1PP} and Theorem~\ref{Thmpsi} for which we consider an arbitrary function  $v$ in $\bar S$ and $\kappa>0$ satisfying \eqref{t2PP}. According to  \cite[IX.Proposition~1.8]{AEIII} we can write the open set $D\setminus\mathcal{C}(v)$ as a countable union of disjoint open intervals  $(I_j)_{j\in J}$; that is,
$$
D\setminus\mathcal{C}(v)=\bigcup_{j\in J} I_j\,.
$$
Hence, $\Omega(v)$ is the disjoint union of the open domains $\mathcal{O}_{I_j}(v)$. Now recall from  Proposition~\ref{lemt9P} that $\chi_v\in H_B^1(\Omega(v))$ is the unique minimizer on $H_B^1(\Omega(v))$ of the functional 
\begin{equation*}
G_D(v)[\vartheta] = \frac{1}{2} \int_{\Omega(v)} |\nabla\vartheta|^2\ \mathrm{d}(x,z) + \frac{1}{2} \int_D \sigma |\vartheta(\cdot,-H)|^2\ \mathrm{d}x - \int_{\Omega(v)} \vartheta \Delta h_v \ \mathrm{d}(x,z)\,, \qquad \vartheta\in H_B^1(\Omega(v))\,.
\end{equation*}
Furthermore, since $\Delta h_v$ belongs to $L_2(\Omega(v))$ by Lemma~\ref{h}, it follows from the definition of $H_B^1(\Omega(v))$ that
\begin{equation*}
 G_D(v)[\vartheta]=\sum_{j\in J} G_{I_j}(v)[\vartheta]\,, \qquad \vartheta\in H_B^1(\Omega(v))\,,
\end{equation*}
where $G_{I_j}(v)[\vartheta]$ is defined by \eqref{t101P} with $f :=\Delta h_v \mathbf{1}_{\mathcal{O}_{I_j}(v)}$. Restricting to $\vartheta\in H_B^1(\mathcal{O}_{I_j}(v))$, it thus readily follows that  $\chi_v \mathbf{1}_{\mathcal{O}_{I_j}(v)}$ is a minimizer of $G_{I_j}(v)$ on $H_B^1(\mathcal{O}_{I_j}(v))$. Consequently,  $\chi_v \mathbf{1}_{\mathcal{O}_{I_j}(v)} = \zeta_{I_j,v}$ by Lemma~\ref{lemt0P}.
Hence Theorem~\ref{thmt1P} yields
\begin{equation*}
\|\chi_v\|_{H^2(\mathcal{O}_{I_j}(v))} + \|\partial_x  \chi_v(\cdot,-H)\|_{L_2(I_j)} \le C_{\ref{cr0P}}(\kappa) \|\Delta h_v\|_{L_2(\mathcal{O}_{I_j}(v))}
\end{equation*}
and
\begin{equation*}
\|\partial_z\chi_v(\cdot,v)\|_{L_r(I_j)} \le r C_{\ref{cr0rP}}(\kappa) \|\Delta h_v\|_{L_2(\mathcal{O}_{I_j}(v))} \,, \qquad r\in [2,\infty)\,,
\end{equation*}
with constants $C_{\ref{cr0P}}(\kappa)$ and $C_{\ref{cr0rP}}(\kappa) $ not depending on $I_j$. Therefore, summing with respect to $j\in J$, we conclude that $\chi_v\in H^2(\Omega(v))$ and satisfies \eqref{t6PP} and \eqref{t6rPP},  since $\|\Delta h_v\|_{L_2(\Omega(v))}\le c(\kappa)$ by Lemma~\ref{h}. Therefore, as in \cite[Theorem~3.5]{LNW19}, we may use the version of Gau\ss' Theorem stated in \cite[Folgerung~7.5]{Ko63} in the variational characterization of $\chi_v$ featuring $\mathcal{G}(v)$ to deduce that $\chi_v\in H^2(\Omega(v))$ is indeed a strong solution to \eqref{bbb}. This proves Theorem~\ref{thmt1PP}. Owing to \eqref{chi} and Lemma~\ref{h}, this also entails Theorem~\ref{Thmpsi}.

\section{Continuity of $\chi_v$ with Respect to $v$} \label{sec.Cont}

In this section we derive continuity properties of $\chi_v$ and its gradient trace  $\partial_z\chi_v(\cdot, v)$ with respect to $v\in \bar S$. The latter will also yield the continuity of the function $g$ defined in \eqref{GG}. Throughout this section we denote positive constants depending only on $\sigma$ by $C$. The dependence upon additional parameters will be indicated explicitly.

\subsection{$H^1$-Continuity: $\Gamma$-convergence of $\mathcal{G}$} \label{sec.Gc}

Let us recall that, according to Proposition~\ref{lemt9P},  $\chi_v$ is the unique minimizer on $H_B^1(\Omega(v))$ of the functional $\mathcal{G}(v)$ introduced in \eqref{GFunc} as
$$
\mathcal{G}(v)[\vartheta]= \frac{1}{2} \int_{\Omega(v)} |\nabla (\vartheta+h_v)|^2 \,\mathrm{d}(x,z) + \frac{1}{2} \int_D \sigma(x) |\vartheta(x,-H)+h_v(x,-H) - \mathfrak{h}_v(x)|^2   \,\mathrm{d}x
$$
for $\vartheta\in H_B^1(\Omega(v))$. Now, in order to derive continuity properties of $\chi_v$ (and $\psi_v$) with respect to $v\in \bar S$, we first prove a Gamma convergence result for the set of functionals $\{\mathcal{G}(v)\,,\, v\in \bar S\}$. More precisely, given $M>0$ we set as before $\Omega(M) := D\times (-H,M)$ and, for $v\in\bar S$  such that $v\le M-H$, we extend the functional $\mathcal{G}(v)$ to $L_2(\Omega(M))$ by defining
$$
\mathcal{G}(v)[\vartheta]:=
\infty\,, \quad \vartheta\in L_2(\Omega(M))\setminus H_B^1(\Omega(v))\,.
$$
With these notations we have:

\begin{proposition}\label{P3b}
Let $M>0$ and consider a sequence $(v_n)_{n\ge 1}$ in $\bar{S}$ and $v\in \bar S$ such that 
\begin{equation}\label{p9}
-H\le v_n(x)\,,\, v(x)\le M-H\,,\quad x\in D\,, \qquad v_n\rightarrow v \text{ in }\  H^1(D)\,.
\end{equation}
Then
\begin{equation*}
\Gamma-\lim_{n\rightarrow \infty} \mathcal{G}(v_n) =\mathcal{G}(v)\quad\text{in }\ L_2(\Omega(M))\,.
\end{equation*}
\end{proposition}

\begin{proof} The proof is very similar to that of \cite[Proposition~3.11]{LW19}.\\
	
\noindent\textit{(i) Asymptotic weak lower semi-continuity}. Given a sequence $(\vartheta_n)_{n\ge 1}$ in $L_2(\Omega(M))$ and $\vartheta\in L_2(\Omega(M))$ satisfying
\begin{equation}\label{ppp}
\vartheta_n\rightarrow \vartheta\ \text{ in }\ L_2(\Omega(M))\,,
\end{equation}
we shall show that
\begin{equation}\label{1ay}
\mathcal{G}(v)[\vartheta]\le\liminf_{n\rightarrow \infty} \mathcal{G}(v_n)[\vartheta_n]\,.
\end{equation}
We may assume without of loss of generality that
\begin{equation}\label{1a}
\vartheta_n\in H_B^1(\Omega(v_n))\,, \quad n\ge 1\,, \qquad	\mathcal{G}_\infty := \sup_{n\ge 1} \mathcal{G}(v_n)[\vartheta_n]<\infty\,.
\end{equation}
Let $n\ge 1$ and denote the extension by zero of $\vartheta_n$ to $\Omega(M)\setminus\Omega(v_n)$ by $\vartheta_n$. Then $\vartheta_n\in H_B^1(\Omega(M))$ and it follows from \eqref{p9}, \eqref{ppp}, \eqref{1a}, and Lemma~\ref{h}~\textbf{(b)} that the sequence $(\vartheta_n)_{n\ge 1}$ is bounded in $H_B^1(\Omega(M))$. Since $\Omega(M)$ is a Lipschitz domain, the compactness of the embedding of $H^1(\Omega(M))$ in $H^{3/4}(\Omega(M))$ \cite[Theorem~1.4.3.2]{Gr1985}, the continuity of the trace operator from $H^{3/4}(\Omega(M))$ to $L_2(\partial\Omega(M))$ (see, e.g., \cite[Theorem~1.5.1.2]{Gr1985}, \cite{Marschall87}, or \cite[Satz 8.7]{Wloka82}) and \eqref{ppp} ensure that there is a subsequence of $(\vartheta_n)_{n\ge 1}$ (not relabeled) such that 
\begin{align}
\vartheta_n & \rightharpoonup \vartheta \quad\text{in }\ H_B^1(\Omega(M))\,, \label{2aa} \\
\vartheta_n & \rightarrow \vartheta \quad\text{in }\ L_2(\partial\Omega(M))\,. \label{2ba}
\end{align}
In particular, $\vartheta\in H^1(\Omega(v))$ and its trace $\vartheta(\cdot,v)$ is well-defined in $L_2(D,(H+v)\,\mathrm{d}x)$ according to Lemma~\ref{lemT1H}. Similarly, for each $n\ge 1$, $\vartheta\in H^1(\Omega(v_n))$ and its trace $\vartheta(\cdot,v_n)$ is well-defined in $L_2(D,(H+v_n)\,\mathrm{d}x)$. Consequently, for $n\ge 1$,
\begin{equation}
\begin{split}
\int_D (H+v) (H+v_n) |\vartheta(\cdot,v)|^2\,\mathrm{d}x & \le 2 \int_D (H+v) (H+v_n) |\vartheta(\cdot,v)-\vartheta(\cdot,v_n)|^2\,\mathrm{d}x \\
& \quad + 2 \int_D (H+v) (H+v_n) |\vartheta(\cdot,v_n)|^2\,\mathrm{d}x\,.
\end{split} \label{Gcz1}
\end{equation}
On the one hand, by  Lemma~\ref{lemA1} and \eqref{p9},
\begin{align}
\int_D (H+v) (H+v_n) |\vartheta(\cdot,v)-\vartheta(\cdot,v_n)|^2\,\mathrm{d}x & \le \| (H+v) (H+v_n) (v-v_n)\|_{L_\infty(D)} \|\partial_z\vartheta\|_{L_2(\Omega(M))}^2 \nonumber \\
& \le M^2 \|v-v_n\|_{L_\infty(D)} \|\partial_z\vartheta\|_{L_2(\Omega(M))}^2\,. \label{Gcz2}
\end{align}
On the other hand, since $\vartheta_n\in H_B^1(\Omega(v_n))$, we infer from \eqref{p9} and Lemma~\ref{lemT1H} that
\begin{align}
& \int_D (H+v) (H+v_n) |\vartheta(\cdot,v_n)|^2\,\mathrm{d}x \nonumber  \\ 
& \qquad = \int_D (H+v) (H+v_n) |\vartheta(\cdot,v_n)-\vartheta_n(\cdot,v_n)|^2\,\mathrm{d}x \nonumber \\ 
& \qquad \le M \int_D (H+v_n) |\vartheta(\cdot,v_n)-\vartheta_n(\cdot,v_n)|^2\,\mathrm{d}x \nonumber \\ 
&  \qquad \le M \left[ \|\vartheta-\vartheta_n\|_{L_2(\Omega(v_n))}^2 + 2 \|H+v_n\|_{L_\infty(D)} \|\vartheta-\vartheta_n\|_{L_2(\Omega(v_n))} \|\partial_z(\vartheta-\vartheta_n)\|_{L_2(\Omega(v_n))} \right] \nonumber \\
& \qquad \le M \|\vartheta-\vartheta_n\|_{L_2(\Omega(M))} \left[ \sup_{m\ge 1} \|\vartheta-\vartheta_m\|_{L_2(\Omega(M))} + 2 M \sup_{m\ge 1} \|\partial_z(\vartheta-\vartheta_m)\|_{L_2(\Omega(M))} \right] \nonumber \\
& \qquad \le 2M(1+M) \|\vartheta-\vartheta_n\|_{L_2(\Omega(M))} \left[ \|\vartheta\|_{H^1(\Omega(M))} + \sup_{m\ge 1} \|\vartheta_m\|_{H^1(\Omega(M))} \right] \,. \label{Gcz3}
\end{align}
Now, it readily follows from \eqref{p9}, \eqref{ppp}, \eqref{2aa}, \eqref{Gcz2}, \eqref{Gcz3}, and the continuous embedding of $H_0^1(D)$ in $C(\bar{D})$ that the right-hand side of \eqref{Gcz1} converges to zero as $n\to\infty$. Therefore, 
\begin{equation*}
\lim_{n\to\infty} \int_D (H+v) (H+v_n) |\vartheta(\cdot,v)|^2\,\mathrm{d}x = 0\,,
\end{equation*}
and we use Fatou's lemma to conclude that 
\begin{equation*}
\vartheta(\cdot,v)=0 \quad\text{in }\quad L_2(D, (H+v)^2\,\mathrm{d}x)\,. 
\end{equation*}
Combining this result with \eqref{2aa} and \eqref{2ba} implies that 
\begin{equation}
\vartheta\in H_B^1(\Omega(v))\,. \label{Gcz4}
\end{equation}
Now, we infer from \eqref{Gcz5}, \eqref{p9}, \eqref{2aa}, \eqref{Gcz4}, and the continuous embedding of $H_0^1(D)$ in $C(\bar{D})$ that
\begin{align*}
\int_{\Omega(v)} |\nabla (\vartheta+h_v)|^2\,\mathrm{d}(x,z) & = \int_{\Omega(M)} |\nabla (\vartheta+h_v)|^2\,\mathrm{d}(x,z) - \int_{ \Omega(M)\setminus\Omega(v)} |\nabla h_v|^2\,\mathrm{d}(x,z) \\
& \le \liminf_{n\to\infty} \int_{\Omega(M)} |\nabla (\vartheta_n+h_{v_n})|^2\,\mathrm{d}(x,z) - \lim_{n\to\infty} \int_{\Omega(M)\setminus\Omega(v_n)} |\nabla h_{v_n}|^2\,\mathrm{d}(x,z) \\
& = \liminf_{n\to\infty} \int_{\Omega(v_n)} |\nabla (\vartheta_n+h_{v_n})|^2\,\mathrm{d}(x,z)\,.
\end{align*}
Also, from \eqref{2ba} and Lemma~\ref{h} we deduce that
\begin{equation*}
\lim_{n\to\infty} \int_D  \sigma \left| (\vartheta_n + h_{v_n})(\cdot,-H) - \mathfrak{h}_{v_n} \right|^2 \,\mathrm{d}x = \int_D \sigma \left| (\vartheta + h_{v})(\cdot,-H) - \mathfrak{h}_{v} \right|^2  \,\mathrm{d}x\,.
\end{equation*}
Gathering the outcome of the above analysis gives \eqref{1ay}.

\medskip

\noindent\textit{(ii) Recovery sequence}. Consider $\vartheta\in H_B^1(\Omega(v))$ and introduce the function $\bar{\vartheta}$ defined on 
$$
\hat{\Omega}(M) := D\times (-2H-M,M)
$$ 
by
$$
\bar\vartheta(x,z):= \left\{ \begin{array}{ll} 0\,, & x\in D\,,\ v(x)<z<M\,, \\[0.1cm]
\vartheta(x,z)\,, &  x\in D\,,\ -H<z\le v(x)\,,\\[0.1cm]
\vartheta(x,-2H-z)\,, &  x\in D\,,\ -2H-v(x)<z\le -H\,,\\[0.1cm]
0\,, &  x\in D\,,\ -2H-M<z\le -2H-v(x)\,,\\
\end{array} \right.
$$
which is the extension of $\vartheta$ by zero in $\Omega(M)\setminus\Omega(v)$ and the reflection of the thus obtained function to $D\times (-2H-M,-H)$. Then $\bar\vartheta\in H_0^1(\hat\Omega(M))$, so that $F:=-\Delta \bar\vartheta\in H^{-1}(\hat\Omega(M))$. 

Let $n\ge 1$. Since 
\begin{equation*}
\hat\Omega(v_n):=\Omega(v_n) \cup \big(D\times (-2H-M,-H]\big)\subset \hat\Omega(M)\,,
\end{equation*} 
the distribution $F$ can also be considered as an element of $H^{-1}(\hat\Omega(v_n))$ by restriction. Then there is a unique variational solution $\hat\vartheta_n\in H_0^1(\hat\Omega(v_n))\subset H_0^1(\hat\Omega(M))$ to
$$
-\Delta \hat\vartheta_n=F \quad\text{in }\ \hat\Omega(v_n)\,,\qquad \hat\vartheta_n=0 \quad\text{on }\ \partial\hat\Omega(v_n)\,.
$$
Owing to \eqref{p9}  and the continuous embedding of $H_0^1(D)$ in $C(\bar{D})$, 
$$
d_H\left( \hat\Omega(v_n),\hat\Omega(v) \right) \le \|v_n - v\|_{L_\infty(D)}\rightarrow 0\,,
$$
where $d_H$ stands for  the Hausdorff distance in $\hat\Omega(M)$, see \cite[Section~2.2.3]{HP05}. Since $\overline{\hat\Omega(M)}\setminus\hat\Omega(v_n)$ has a single connected component for all $n\ge 1$, it follows from \cite[Theorem~4.1]{Sv93} and \cite[Theorem~3.2.5]{HP05} that $\hat\vartheta_n\rightarrow \hat\vartheta$ in $H_0^1(\hat\Omega(M))$, where $\hat\vartheta\in H_0^1(\hat\Omega(M))$ is the unique variational solution to
$$
-\Delta \hat\vartheta=F \quad\text{in }\ \hat\Omega(M)\,,\qquad \hat\vartheta=0 \quad\text{on }\ \partial\hat\Omega(M)\,.
$$
Clearly, $\hat\vartheta=\bar\vartheta$ by uniqueness, so that  $ \hat{\vartheta}_n  \rightarrow \bar\vartheta$ in $H_0^1(\hat{\Omega}(M))$. Setting $\vartheta_n := \hat\vartheta_n \mathbf{1}_{\Omega(v_n)}\in H^1(\Omega(M))$, $n\ge 1$, this convergence implies that
\begin{equation}\label{ui}
\vartheta_n\rightarrow \bar\vartheta\quad \text{ in }\ H^1(\Omega(M))\,.
\end{equation}
Since $\vartheta_n=0$ in $\Omega(M)\setminus\Omega(v_n)$ we obtain from \eqref{Gcz5}, \eqref{p9}, and \eqref{ui} that
\begin{equation*}
\begin{split}
\int_{\Omega(v)} \vert \nabla(\vartheta+h_v)\vert^2\,\mathrm{d} (x,z) &= \int_{\Omega(v)} \left(\vert \nabla\vartheta\vert^2 + 2 \nabla\vartheta\cdot \nabla h_v + \vert \nabla h_v\vert^2 \right)\,\mathrm{d} (x,z) \\
& = \int_{\Omega(M)} \left(\vert \nabla\vartheta\vert^2 + 2 \nabla\vartheta\cdot \nabla h_v\right)\,\mathrm{d} (x,z) + \int_{\Omega(v)} \vert \nabla h_v\vert^2\,\mathrm{d} (x,z) \\
& = \lim_{n\rightarrow\infty} \int_{\Omega(M)} \left( \vert \nabla\vartheta_n\vert^2 + 2 \nabla\vartheta_n\cdot \nabla h_{v_n}\right)\,\mathrm{d} (x,z) \\
&\qquad + \lim_{n\rightarrow\infty} \int_{\Omega(v_n)} \vert \nabla h_{v_n}\vert^2\,\mathrm{d} (x,z)\\
&= \lim_{n\rightarrow\infty} \int_{\Omega(v_n)} \vert \nabla(\vartheta_n +h_{v_n})\vert^2\,\mathrm{d} (x,z)\,.
\end{split}
\end{equation*}
Moreover, the continuity of the trace from $H^1(\Omega(M))$ to $L_2(D\times\{-H\})$ and \eqref{ui} entail that  
$$\vartheta_n(\cdot,-H)\rightarrow \bar\vartheta(\cdot,-H)=\vartheta(\cdot,-H)\quad \text{in}\quad L_2(D)\,.
$$
These two properties, along with \eqref{Gcz6} and \eqref{Gcz7}, imply that 
$$
\mathcal{G}(v)[\vartheta]=\lim_{n\rightarrow\infty} \mathcal{G}(v_n)[\vartheta_n]\,;
$$
that is, $(\vartheta_n)_{n\ge 1}$ is a recovery sequence for $\vartheta$ and the claim is proved.
\end{proof}

The Fundamental Theorem of $\Gamma$-convergence, see \cite[Corollary~7.20]{DaM93}, then yields the following continuous dependence of $\chi_v$ on $v\in \bar S$: 

\begin{corollary}\label{L3}
Suppose \eqref{p9} and assume further that there is $\kappa>0$ such that 
\begin{equation}
\|v\|_{H^2(D)}\le \kappa \;\text{ and }\; \|v_n\|_{H^2(D)}\le \kappa\,, \qquad n\ge 1\, . \label{x1P}
\end{equation} 
Then
\begin{equation}
\lim_{n\to\infty} \mathcal{G}(v_n)[\chi_{v_n}] = \mathcal{G}(v)[\chi_{v}] \label{x2P}
\end{equation}
and
\begin{equation}
\lim_{n\to\infty} \|\chi_{v_n}-\chi_{v}\|_{{H^1}(\Omega(M))} =\lim_{n\to\infty} \|\chi_{v_n}(\cdot,-H)-\chi_{v}(\cdot,-H)\|_{L_r(D)} = 0\,, \quad r\in [1,\infty)\,. \label{x3P}
\end{equation}
\end{corollary}

\begin{proof}
It readily follows  from \eqref{p9}, \eqref{x1P}, and Theorem~\ref{thmt1PP}  that
\begin{equation}
(\chi_{v_n})_{n\ge 1} \;\text{ is bounded in }\; H^1(\Omega(M)) \label{x4P}
\end{equation} 
and thus relatively compact in $L_2(\Omega(M))$ by \cite[Theorem~1.4.5.2]{Gr1985}. According to Proposition~\ref{P3b}, we deduce from the Fundamental Theorem of $\Gamma$-convergence, see \cite[Corollary~7.20]{DaM93}, that any cluster point of $(\chi_{v_n})_{n\ge 1}$ in $L_2(\Omega(M))$ is a minimizer of $\mathcal{G}(v)$ and thus coincides with $\chi_v$ by Proposition~\ref{lemt9P}. Therefore, 
\begin{equation}
\lim_{n\to\infty} \|\chi_{v_n} - \chi_v \|_{L_2(\Omega(M))} = 0\,, \label{x5P}
\end{equation}
and, using once more \cite[Corollary~7.20]{DaM93}, we obtain \eqref{x2P}.

We are left with proving \eqref{x3P}. To this end, we first observe that, since $\Omega(M)$ is a Lipschitz domain, \cite[Theorem~1.4.3.2, Theorem~1.4.5.2]{Gr1985} imply that $H^1(\Omega(M))$ compactly embeds in $W_q^{3/2q}(\Omega(M))$ for $q\ge 2$. Thus, the continuity of the trace operator from $W_q^{3/2q}(\Omega(M))$ to $L_q(\partial\Omega(M))$ (see \cite[Theorem~1.5.1.2]{Gr1985} and \cite{Marschall87}), along with \eqref{x4P} and \eqref{x5P}, ensure that there is a subsequence of $(\chi_{v_n})_{n\ge 1}$ (not relabeled) such that 
\begin{align}
\chi_{v_n} & \rightharpoonup \chi_v \quad\text{in }\ H_B^1(\Omega(M))\,, \label{x6P} \\
\chi_{v_n}(\cdot,-H) & \rightarrow \chi_v(\cdot,-H) \quad\text{in }\  L_q(D)\,, \quad q\ge 2\,. \label{x7P}
\end{align}
Notice that \eqref{x7P} yields the second assertion of \eqref{x3P}. It now follows from  \eqref{Gcz5}, \eqref{Gcz6}, \eqref{Gcz7}, \eqref{x2P}, and \eqref{x7P} that
\begin{align*}
\lim_{n\to\infty} \|\nabla (\chi_{v_n} + h_{v_n})\|_{L_2(\Omega(M))}^2 & = \lim_{n\to\infty} \|\nabla (\chi_{v_n} + h_{v_n})\|_{L_2(\Omega(v_n))}^2  + \lim_{n\to\infty} \|\nabla h_{v_n}\|_{L_2(\Omega(M)\setminus\Omega(v_n))}^2 \\
& = \|\nabla (\chi_{v} + h_{v})\|_{L_2(\Omega(v))}^2  +  \|\nabla h_{v}\|_{L_2(\Omega(M)\setminus\Omega(v))}^2 \\
& =  \|\nabla (\chi_{v} + h_{v})\|_{L_2(\Omega(M))}^2\, .
\end{align*}
This property, along with \eqref{Gcz5} and \eqref{x6P}, guarantees that $(\nabla \chi_{v_n})_{n\ge 1}$ converges to $\nabla\chi_v$ in $L_2(\Omega(M))$ and the proof of \eqref{x3P} is complete.
\end{proof}

\subsection{Continuity of $\partial_z\chi_v(\cdot, v)$ with Respect to $v$}
Finally, in order to establish the continuity of the function $g$ defined in \eqref{GG} we need also to investigate the continuous dependence of the gradient trace $\partial_z\chi_v(\cdot, v)$ on $v\in \bar S$, the main difficulty arising when $\mathcal{C}(v)\not=\emptyset$. In this regard we note:

\begin{proposition}\label{L2}
Consider $v \in \bar{S}$ and a sequence $(v_n)_{n\ge 1}$ in $\bar{S}$ such that
\begin{equation}
 \|v\|_{H^2(D)} + \sup_{n\ge 1} \|v_n\|_{H^2(D)} \le \kappa \;\text{ and }\; \lim_{n\to\infty} \| v_n - v\|_{H^1(D)} = 0\,. \label{cv1P}
\end{equation}
Then
\begin{equation}\label{ooo}
\ell(v_n)\rightarrow \ell(v) \quad\text{ in }\ L_r(D) \;\ \text{ for }\ \; r\in [1,\infty)\,,
\end{equation}
where $\ell(v)$ is given by
$$
\ell(v)(x):=\left\{
\begin{array}{ll} 
\partial_z\chi_v(x,v(x))\,, &  x\in D\setminus\mathcal{C}(v)\,,\\[0.1cm]
0\,, &x\in \mathcal{C}(v)\,.
\end{array}
\right.
$$
\end{proposition}

\begin{proof}
Thanks to \eqref{cv1P} and the continuous embedding of $H^2(D)$ in $L_\infty(D)$, we may fix $M>H$ (only depending on $\kappa$) such that
\begin{equation}
- H \le v_n(x) , v(x) \le M-H\ , \qquad x\in\bar{D}\ , \quad n\ge 1\, . \label{cv2P}
\end{equation}

\noindent\textbf{Step~1.} We first establish an estimate ensuring that there is no concentration of $\partial_z\chi_v(\cdot,v)$ on small subsets of $D\setminus\mathcal{C}(v)$. Indeed, since $\chi_v\in H^2(\Omega(v))$ we have $\chi_v(x,\cdot)\in H^2((-H,v(x)))$ for a.a. $x\in D\setminus\mathcal{C}(v)$, so that it follows from the boundary conditions \eqref{t5bP} and \eqref{t5cP} that
\begin{align*}
\partial_z \chi_v(x,v(x)) &= \partial_z \chi_v(x,-H) + \int_{-H}^{v(x)} \partial_z^2\chi_v(x,z)\, \mathrm{d}z
\\
&= \sigma(x) \, \chi_v(x,-H) + \int_{-H}^{v(x)} \partial_z^2\chi_v(x,z)\, \mathrm{d}z
\\
&=\sigma(x) \left( \chi_v(x,v(x)) - 
\int_{-H}^{v(x)} \partial_z\chi_v(x,z)\mathrm{d}z\right) 
+ \int_{-H}^{v(x)} \partial_z^2\chi_v(x,z)\, \mathrm{d}z
\\
&= \int_{-H}^{v(x)} \left(\partial_z^2\chi_v(x,z) - \sigma(x) \partial_z\chi_v(x,z)\right)\mathrm{d}z
\end{align*}
for a.a. $x\in D\setminus\mathcal{C}(v)$. Thus, for an arbitrary measurable subset $E \subset D\setminus\mathcal{C}(v)$, we infer from H\"older's inequality that 
\begin{subequations}\label{y0P}
\begin{align}
& \int_E \vert\partial_z \chi_v(x,v(x))\vert \, \mathrm{d}x \nonumber \\
& \qquad \leq \int_E \int_{-H}^{v(x)} \left( \vert \partial_z^2\chi_v(x,z) \vert + \sigma(x)\vert \partial_z\chi_v(x,z)\vert\right) \mathrm{d}z \mathrm{d}x \nonumber
\\
&\qquad \leq \left( \int_E (H+v)(x) \, \mathrm{d}x  \right)^{1/2} \bigg(\int_{\Omega(v)} \left( 2 \vert \partial_z^2\chi_v(x,z) \vert^2 + 2
\Vert \sigma\Vert_{\infty}^2 \vert \partial_z\chi_v(x,z) \vert^2 \right) \mathrm{d}(x,z) \bigg)^{1/2} \nonumber
\\
&\qquad \leq C \left( \int_E (H+v)(x) \, \mathrm{d}x  \right)^{1/2} \Vert \chi_v \Vert_{H^2(\Omega(v))}\,. \label{y0Pa}
\end{align}
Clearly, the same proof implies that, for any $n\ge 1$ and arbitrary measurable subset $E \subset D\setminus\mathcal{C}(v_n)$, 
\begin{equation}
\int_E \vert\partial_z \chi_{v_n}(x,v_n(x))\vert \, \mathrm{d}x \leq C \left( \int_E (H+v_n)(x) \, \mathrm{d}x  \right)^{1/2} \Vert \chi_{v_n} \Vert_{H^2(\Omega(v_n))}\,. \label{y0Pb}
\end{equation}
\end{subequations}

\noindent\textbf{Step~2.} We next handle the behavior of $\partial_z \chi_v(\cdot,v)$ where $v$ stays away from $-H$. To this end, let $\varepsilon\in (0,H/2)$ and define
\begin{equation}
\Lambda(\varepsilon) := \{ x \in D\ :\ v(x)>-H+2\varepsilon\}\,, \label{y1P}
\end{equation} 
which is a non-empty open subset of $D$, since $v\in C(\bar{D})$ with $v(\pm L)=0$. We can thus write it as a countable union of disjoint open intervals $(\Lambda_j(\varepsilon))_{j\in J}$, see \cite[IX.Proposition~1.8]{AEIII}. Also, owing to \eqref{cv1P} and the continuous embedding of $H^1(D)$ in $C(\bar{D})$, there is $n_\varepsilon\ge 1$ such that
\begin{equation}
v(x)-\varepsilon \le v_n(x) \le v(x)+\varepsilon\ , \qquad x\in\bar{D}\ , \quad n\ge n_\varepsilon\, . \label{y2P}
\end{equation}
A straightforward consequence of \eqref{y1P} and \eqref{y2P} is that
\begin{equation}
\{ (x,z) \in \Lambda(\varepsilon)\times [-H,\infty)\ :\ -H< z < v(x)-\varepsilon\} \subset \Omega(v_n)\ , \qquad n\ge n_\varepsilon\,. \label{y3P}
\end{equation}
Therefore, the function $X_{n,\varepsilon}$, given by
\begin{equation*}
X_{n,\varepsilon}(x) := \partial_z \chi_v(x,v(x)-\varepsilon) - \partial_z \chi_{v_n}(x,v(x)-\varepsilon), \qquad x \in \Lambda(\varepsilon)\, , \quad n\ge n_\varepsilon\, ,
\end{equation*}
is well-defined. Let $j\in J$ and $n\ge n_\varepsilon$. Since $\partial_z \chi_v$ and $\partial_z \chi_{v_n}$ both belong to $H^1(\mathcal{O}_{\Lambda_j(\varepsilon)}(v-\varepsilon))$, the set $\mathcal{O}_{\Lambda_j(\varepsilon)}(v-\varepsilon)$ being defined in \eqref{t3P}, it follows from \eqref{t100P}, \eqref{cv2P}, and the definition of $\Lambda(\varepsilon)$ that
\begin{align*}
\varepsilon \int_{\Lambda_j(\varepsilon)} |X_{n,\varepsilon}(x)|^2\ \mathrm{d}x & \le \int_{\Lambda_j(\varepsilon)} |X_{n,\varepsilon}(x)|^2 (H+v(x)-\varepsilon)\ \mathrm{d}x \\
& \le \|\partial_z (\chi_v - \chi_{v_n})\|_{L_2(\mathcal{O}_{\Lambda_j(\varepsilon)}(v-\varepsilon))}^2 \\
& \quad + 2 \|H+v-\varepsilon\|_{L_\infty(\Lambda_j(\varepsilon))} \|\partial_z (\chi_v - \chi_{v_n})\|_{L_2(\mathcal{O}_{\Lambda_j(\varepsilon)}(v-\varepsilon))} \|\partial_z^2 (\chi_v - \chi_{v_n})\|_{L_2(\mathcal{O}_{\Lambda_j(\varepsilon)}(v-\varepsilon))} \\
& \le \|\partial_z (\chi_v - \chi_{v_n})\|_{L_2(\mathcal{O}_{\Lambda_j(\varepsilon)}(M))}^2 \\
& \quad + C(\kappa) \|\partial_z (\chi_v - \chi_{v_n})\|_{L_2(\mathcal{O}_{\Lambda_j(\varepsilon)}(M))} \left( \|\partial_z^2 \chi_v\|_{L_2(\mathcal{O}_{\Lambda_j(\varepsilon)}(v))} + \|\partial_z^2 \chi_{v_n}\|_{L_2(\mathcal{O}_{\Lambda_j(\varepsilon)}(v_n))} \right) \,.
\end{align*}
Summing the above inequality over $j\in J$ and noticing that
\begin{align*}
& \sum_{j\in J} \|\partial_z (\chi_v - \chi_{v_n})\|_{L_2(\mathcal{O}_{\Lambda_j(\varepsilon)}(M))} \left( \|\partial_z^2 \chi_v\|_{L_2(\mathcal{O}_{\Lambda_j(\varepsilon)}(v))} + \|\partial_z^2 \chi_{v_n}\|_{L_2(\mathcal{O}_{\Lambda_j(\varepsilon)}(v_n))} \right) \\
& \quad \le \left( \sum_{j\in J} \|\partial_z (\chi_v - \chi_{v_n})\|_{L_2(\mathcal{O}_{\Lambda_j(\varepsilon)}(M))}^2 \right)^{1/2} \left( \sum_{j\in J} \left( \|\partial_z^2 \chi_v\|_{L_2(\mathcal{O}_{\Lambda_j(\varepsilon)}(v))} + \|\partial_z^2 \chi_{v_n}\|_{L_2(\mathcal{O}_{\Lambda_j(\varepsilon)}(v_n))} \right)^2 \right)^{1/2} \\
& \quad \le \sqrt{2} \|\partial_z (\chi_v - \chi_{v_n})\|_{L_2(\Omega(M))} \left( \sum_{j\in J} \left( \|\partial_z^2 \chi_v\|_{L_2(\mathcal{O}_{\Lambda_j(\varepsilon)}(v))}^2 + \|\partial_z^2 \chi_{v_n}\|_{L_2(\mathcal{O}_{\Lambda_j(\varepsilon)}(v_n))}^2 \right) \right)^{1/2} \\
& \quad \le \sqrt{2} \|\partial_z (\chi_v - \chi_{v_n})\|_{L_2(\Omega(M))} \left( \|\partial_z^2 \chi_v\|_{L_2(\Omega(v))} + \|\partial_z^2 \chi_{v_n}\|_{L_2(\Omega(v_n))} \right) \\
& \quad \le  C(\kappa) \|\partial_z (\chi_v - \chi_{v_n})\|_{L_2(\Omega(M))}
\end{align*}
by the Cauchy-Schwarz inequality, \eqref{cv1P},  Theorem~\ref{thmt1PP}, we obtain
\begin{equation*}
\varepsilon \int_{\Lambda(\varepsilon)} |X_{n,\varepsilon}(x)|^2\ \mathrm{d}x \le \|\partial_z (\chi_v - \chi_{v_n})\|_{L_2(\Omega(M))}^2 + C(\kappa) \|\partial_z (\chi_v - \chi_{v_n})\|_{L_2(\Omega(M))} \,.
\end{equation*}
We now infer from \eqref{x3P} and the above inequality that
\begin{equation}
\lim_{n\to\infty} \int_{\Lambda(\varepsilon)} |X_{n,\varepsilon}(x)|^2\ \mathrm{d}x = 0\, . \label{x8P}
\end{equation}

We next set
\begin{equation*}
Y_n (x) := \partial_z \chi_v(x,v(x)) - \partial_z \chi_{v_n}(x,v_n(x)), \qquad x \in \Lambda(\varepsilon)\ , \quad n\ge n_\varepsilon\, .
\end{equation*}
Using \eqref{y2P} and H\"older's and Young's inequalities, we obtain, for $j\in J$, 
\begin{align*}
\Vert Y_n \Vert_{L_1(\Lambda_j(\varepsilon))}  & \leq \Vert X_{n,\varepsilon}\Vert_{L_1(\Lambda_j(\epsilon))} 
+ \int_{\Lambda_j(\varepsilon)} \left\vert \int_{v-\varepsilon}^v \partial_z^2 \chi_v(\cdot,z)\,\mathrm{d}z 
- \int_{v-\varepsilon}^{v_n}\partial_z^2 \chi_{v_n}(\cdot,z)\,\mathrm{d}z\right\vert \, \mathrm{d}x 
\\
& \leq  \Vert X_{n,\varepsilon}\Vert_{L_1(\Lambda_j(\varepsilon))}  
 +\int_{\Lambda_j(\varepsilon)} \int_{v-\varepsilon}^v \vert \partial_z^2 \chi_v(\cdot,z)\vert \,\mathrm{d}z \mathrm{d}x
 + \int_{\Lambda_j(\varepsilon)} \int_{v-\varepsilon}^{v_n} \vert \partial_z^2 \chi_{v_n}(\cdot,z)\vert \,\mathrm{d}z \mathrm{d}x 
\\
& \leq \Vert X_{n,\varepsilon} \Vert_{L_1(\Lambda_j(\varepsilon))} 
+\sqrt{\varepsilon \vert \Lambda_j(\varepsilon)\vert}\left(\int_{\Lambda_j(\varepsilon)} 
\int_{v-\varepsilon}^v \vert \partial_z^2 \chi_v(\cdot,z)\vert^2 \,\mathrm{d}z \mathrm{d}x \right)^{1/2}
\\
& \quad +\sqrt{2\varepsilon \vert \Lambda_j(\varepsilon)\vert}\left(\int_{\Lambda_j(\varepsilon)} 
\int_{v-\varepsilon}^{v_n} \vert \partial_z^2 \chi_{v_n}(\cdot,z)\vert^2 \,\mathrm{d}z \mathrm{d}x \right)^{1/2}
\\
&\leq \Vert X_{n,\varepsilon} \Vert_{L_1(\Lambda_j(\epsilon))} + \frac{\sqrt{\varepsilon}}{2} \, \vert\Lambda_j(\varepsilon) \vert 
+  \frac{\sqrt{\varepsilon}}{2} \int_{\Lambda_j(\varepsilon)} \int_{-H}^v \vert \partial_z^2 \chi_v(\cdot,z)\vert^2 \,\mathrm{d}z \mathrm{d}x
\\
&\quad + \frac{\sqrt{\varepsilon}}{2} \, \vert\Lambda_j(\varepsilon) \vert 
+  \sqrt{\varepsilon} \int_{\Lambda_j(\varepsilon)} \int_{-H}^{v_n} \vert \partial_z^2 \chi_{v_n}(\cdot,z)\vert^2 \,\mathrm{d}z \mathrm{d}x\,.
\end{align*}
Summing over $j\in J$ and using \eqref{cv1P} and Theorem~\ref{thmt1PP} give
\begin{align*}
\Vert Y_n \Vert_{L_1(\Lambda(\varepsilon))} & \leq \Vert X_{n,\varepsilon} \Vert_{L_1(\Lambda(\varepsilon))} + \sqrt{\varepsilon}  \vert \Lambda (\varepsilon)\vert  +  \sqrt{\varepsilon} \Vert \chi_v \Vert_{H^2(\Omega(v))}  + \sqrt{\varepsilon} \Vert \chi_{v_n} \Vert_{H^2(\Omega(v_n))} \\
& \le \Vert X_{n,\varepsilon} \Vert_{L_1(\Lambda(\varepsilon))} +  C(\kappa) \sqrt{\varepsilon}  \,.
\end{align*}
Owing to \eqref{x8P}, we may take the limit $n\to \infty$ in the previous inequality and obtain
\begin{equation*}
\limsup_{n\to \infty} \Vert Y_n \Vert_{L_1(\Lambda(\varepsilon))}  \leq C(\kappa) \sqrt{\varepsilon} .
\end{equation*}
Since $\Lambda(\varepsilon)\subset \Lambda(\delta)$ for all $\delta \in (0,\varepsilon)$, we infer from the above inequality that
\begin{equation*}
\limsup_{n\to \infty} \Vert Y_n \Vert_{L_1(\Lambda(\varepsilon))} \le \limsup_{n\to \infty} \Vert Y_n \Vert_{L_1(\Lambda(\delta))} \le  C(\kappa) \sqrt{\delta} 
\end{equation*}
and we may pass to the limit $\delta\to 0$ to conclude that
\begin{equation}
\lim_{n\to \infty} \Vert Y_n \Vert_{L_1(\Lambda(\varepsilon))} =0, \qquad \varepsilon \in (0,H/2). \label{x9P}
\end{equation}

\noindent\textbf{Step~3.} Finally, we infer from \eqref{cv1P}, \eqref{cv2P}, \eqref{y0P}, and Theorem~\ref{thmt1PP} that 
\begin{align*}
&\Vert \ell(v_n)- \ell(v)\Vert_{L_1(D)} 
\\
& \;\leq  \int_{\Lambda(\varepsilon)} \vert \ell(v_n)- \ell(v)\vert \, \mathrm{d}x 
+ \int_{D\setminus\Lambda(\varepsilon)} \vert \ell(v_n)\vert\, \mathrm{d}x  
+\int_{D\setminus\Lambda(\varepsilon)} \vert \ell(v)\vert\, \mathrm{d}x
\\
& \; = \Vert Y_n \Vert_{L_1(\Lambda(\varepsilon))}
+ \int_{(D\setminus\Lambda(\varepsilon))\setminus \mathcal{C}(v_n)} \vert \partial_z \chi_{v_n}(\cdot, v_n)\vert \,\mathrm{d}x +  \int_{(D\setminus\Lambda(\varepsilon))\setminus \mathcal{C}(v)} \vert \partial_z \chi_{v}(\cdot, v)\vert \,\mathrm{d}x 
\\
& \;\leq \Vert Y_n \Vert_{L_1(\Lambda(\varepsilon))} + C \left( \int_{(D\setminus\Lambda(\varepsilon))\setminus \mathcal{C}(v_n)} (H + v_n)(x) \, \mathrm{d}x \right)^{1/2} \Vert \chi_{v_n} \Vert_{H^2(\Omega(v_n))} 
\\
& \qquad\qquad\qquad 
+ C \left(\int_{(D\setminus\Lambda(\varepsilon))\setminus\mathcal{C}(v)} (H+v)(x) \, \mathrm{d}x \right)^{1/2} \Vert \chi_{v} \Vert_{H^2(\Omega(v))} \\
& \;\leq \Vert Y_n \Vert_{L_1(\Lambda(\varepsilon))} +  C(\kappa)  \left( \int_{D\setminus\Lambda(\varepsilon)} (H + v)(x) \, \mathrm{d}x \right)^{1/2}  + C(\kappa) \left( \int_{D\setminus\Lambda(\varepsilon)} (H + v_n)(x) \, \mathrm{d}x \right)^{1/2} \,.
\end{align*}
Since $0\le H+v\le 2\varepsilon$ and $0\le H+v_n\le 3\varepsilon$ in $D\setminus\Lambda(\varepsilon)$ for $n\ge n_\varepsilon$ by \eqref{y1P} and \eqref{y2P}, we further obtain
\begin{equation*}
\|\ell(v_n)-\ell(v)\|_{L_1(D)} \le \Vert Y_n \Vert_{L_1(\Lambda(\varepsilon))} + C(\kappa)\sqrt{\varepsilon}\,.
\end{equation*}
We now first let $n\to\infty$ with the help of \eqref{x9P} and then take the limit $\varepsilon\to 0$ to conclude that
\begin{equation}
\lim_{n\to\infty} \|\ell(v_n)-\ell(v)\|_{L_1(D)} = 0\,. \label{x10P}
\end{equation}
Finally, given $r\in [1,\infty)$, we infer from H\"older's inequality, Lemma~\ref{h},  \eqref{t6rPP}, and \eqref{cv1P} that
\begin{align*}
\|\ell(v_n)-\ell(v)\|_{L_r(D)} & \le \|\ell(v_n)-\ell(v)\|_{L_1(D)}^{1/(2r-1)} \|\ell(v_n)-\ell(v)\|_{L_{2r}(D)}^{2(r-1)/(2r-1)} \\
& \le \|\ell(v_n)-\ell(v)\|_{L_1(D)}^{1/(2r-1)} \left( \|\ell(v_n)\|_{L_{2r}(D)}^{2(r-1)/(2r-1)}  + \|\ell(v)\|_{L_{2r}(D)}^{2(r-1)/(2r-1)} \right) \\
& \le C(\kappa, r) \|\ell(v_n)-\ell(v)\|_{L_1(D)}^{1/(2r-1)}
\end{align*}
and the assertion follows from \eqref{x10P}.
\end{proof}

Summarizing the outcome of this section, we have obtained continuity properties of the electrostatic energy $E_e$ and the function $g$ introduced in \eqref{GG}.

\begin{theorem}\label{CC}
The electrostatic energy $E_e: \bar{S} \rightarrow \mathbb{R}$ is continuous for the weak topology of $H^2(D)$. The function $g: \bar{S} \rightarrow L_r(D)$ is continuous for each $r \in [1,\infty)$, the set $\bar{S}$ being still endowed with the weak topology of $H^2(D)$.
\end{theorem}

\begin{proof}
Let us first recall that, if $(v_n)_{n\ge 1}$ is a sequence in $\bar{S}$ converging weakly in $H^2(D)$ to $v\in\bar{S}$, then there is $\kappa>0$ such that \eqref{x1P} and \eqref{cv1P} hold true. Consequently, we infer from Corollary~\ref{L3} that
\begin{equation*}
\lim_{n\to\infty} E_e(v_n) = - \lim_{n\to\infty} \mathcal{G}(v_n)[\chi_{v_n}] = - \mathcal{G}(v)[\chi_v] = E_e(v)\,,
\end{equation*}
thereby establishing the stated continuity of $E_e$. Next, let $v\in \bar{S}$. Since $\partial_x v=0$ a.e. in $\mathcal{C}(v)$, it follows from \eqref{GG} and Proposition~\ref{L2} that
\begin{align*}
g(v)(x) & = \frac{1}{2} (1+|\partial_x v(x)|^2)\,\big[\ell(v)(x) - (\partial_w h)_{v}(x,v(x))\big]^2\\
& \qquad + \sigma(x)\big[\chi_{v}(x,-H)+ h_v(x,-H)-\mathfrak{h}_{v}(x)\big](\partial_w \mathfrak{h})_{v}(x)\\
& \qquad -\frac{1}{2} \left[ \big\vert(\partial_x h)_v\big\vert^2+ \big((\partial_z h)_v+(\partial_w h)_v\big)^2 \right](x, v(x))
\end{align*}
for $x\in D$. The stated continuity of $g$ then readily follows from Proposition~\ref{L2} and the $C^1$-regularity of $h$ and $\mathfrak{h}$ (see also Lemma~\ref{h}{\bf (b)}).
\end{proof}

\section{Shape Derivative of the Electrostatic Energy}\label{Sec5}

In this section we investigate differentiability properties of the 
 electrostatic energy
$$
E_e(u)=-\dfrac{1}{2}\displaystyle\int_{\Omega(u)}  \big\vert\nabla \psi_u\big\vert^2\,\rd (x,z)\\
 -\dfrac{1}{2}\displaystyle\int_{ D} \sigma(x) \big\vert \psi_u(x,-H)- \mathfrak{h}_{u}(x)\big\vert^2\,\rd x
$$
with respect to $u\in \bar S$, where $\psi_u$ is the strong solution to \eqref{MBP0}, see Theorem~\ref{Thmpsi}. Owing to the dependence of $\psi_u$ on the domain $\Omega(u)$ this resembles the computation of a shape derivative, a topic which has received considerable attention in recent years, see \cite{BB05,HP05,SZ92} and the references therein. Note that we may write alternatively $E_e(u)=-\mathcal{G}(u)[\psi_u-h_u]$, since $\chi_u= \psi_u-h_u$ is the strong solution to \eqref{bbb} (with $v=u$) given by Theorem~\ref{thmt1PP}.

As might be expected, the switch between boundary conditions for $\psi_u$ when $\mathcal{C}(u)\ne \emptyset$ generates additional difficulties and we begin with the differentiability  of $\psi_u$ with respect to $u\in S$.

\begin{lemma}\label{P316}
Let $u\in S$ be fixed and define, for $v\in S$, the transformation 
$
\Theta_v:\Omega(u)\rightarrow \Omega(v)
$  
by 
\begin{equation*}
\Theta_{v}(x,z):=\left(x,z+\frac{v(x)-u(x)}{H+u(x)}(z+H)\right)\,,\quad (x,z)\in\Omega(u) \,.
\end{equation*}
Then there exists a neighborhood $U$ of $u$ in $S$ such that the mapping
$$
U\rightarrow H_B^1(\Omega(u)),\quad v\mapsto \chi_v\circ \Theta_v
$$
is continuously differentiable, where $\chi_v=\psi_v-h_v\in H_B^1(\Omega(v))$ solves \eqref{bbb}, see Theorem~\ref{thmt1PP}, and $S$ is endowed with the $H^2(D)$-topology.
\end{lemma}

\begin{proof}
The proof follows the lines of \cite[Theorem~5.3.2]{HP05}, a similar proof is given in \cite[Lemma 4.1]{LW19}. We thus only provide a very brief sketch here. Let $u\in S$ and $v\in S$. Setting $\xi_v:=\chi_v\circ \Theta_v$ and performing a change of variables $(\bar x,\bar z)=\Theta_v(x,z)$, the weak formulation \eqref{yes1} satisfied by $\chi_v$ (as a critical point of $\mathcal{G}(v)$) can be written in the form
\begin{equation}\label{w6}
\begin{split}
\int_{\Omega(u)} &J_v\, \big((D\Theta_v)^{-1} (D\Theta_v^T)^{-1}\nabla\xi_v \big)\cdot\nabla\phi\,\rd (x,z) +\int_D \sigma \big(\xi_v \phi\big)(\cdot,-H)\,\rd x\\
&=  
-\int_{\Omega(u)} J_v\, \big((D\Theta_v)^{-1} (\nabla h_v\circ \Theta_v) \big)\cdot\nabla\phi\,\rd (x,z) +\int_D \sigma \big[\mathfrak{h}_v-h_v(\cdot,-H)\big]\phi(\cdot,-H)\, \rd x
\end{split}
\end{equation}
for $\phi\in H_B^1(\Omega(u))$,
where $J_v:= |\mathrm{det}(D\Theta_v)|$. Therefore, \eqref{w6} is equivalent to 
\begin{equation}\label{w7}
F(v,\xi_v)=0\,,\quad v\in S\,,
\end{equation}
for some Fr\'echet differentiable function
$$
F: S\times H_B^1(\Omega(u))\rightarrow (H_B^1(\Omega(u)))'\,.
$$ 
One then uses the Implicit Function Theorem to derive that $\xi_v$ depends smoothly on $v$.
\end{proof}

As a next step  we establish the Fr\'echet differentiability of $E_e$ on the open set $S$. For $u\in S$ recall that $g(u)$ is given by \eqref{g} since $\mathcal{C}(u)=\emptyset$ in this case.

\begin{proposition} \label{P416}
Let $S$ be endowed with the $H^2(D)$-topology. Then the electrostatic energy 
$E_e: S \rightarrow \mathbb{R}$ is continuously Fr\'echet differentiable with 
\begin{equation*}
\partial_u  E_e(u) [\vartheta] =\int_D g(u)\vartheta\,\rd x
\end{equation*}
for $u\in S$ and $\vartheta \in H^2(D)\cap H^1_0(D)$.
\end{proposition}

\begin{proof}
In this proof we shall use the notation from Lemma~\ref{P316}. We fix $u\in S$ and recall from Lemma~\ref{P316} that the mapping $v \mapsto \xi_v = \chi_v \circ \Theta_v$ is continuously differentiable with respect to $v$ in a neighborhood $U$ of $u$ in $S$ and takes values in $H^1_B(\Omega(u))$. With $\psi_v = \chi_v + h_v$, $J_v=\vert\mathrm{det}(D\Theta_v)\vert$, and the change of variables $(\bar x,\bar z)=\Theta_v(x,z)$, we obtain that, for $v \in U$, 
\begin{align*}
E_e(v) &= -\frac{1}{2} \int_{\Omega(v)} \vert \nabla \psi_v\vert^2 \, 
\rd (\bar x,\bar z) 
- \frac{1}{2} \int_D \sigma \big\vert \psi_v(\bar x, -H)  - \mathfrak{h}_v(\bar{x})\big\vert^2 \, \rd \bar x
\\
& = -\frac{1}{2} \int_{\Omega(u)} \left\vert (D\Theta_v^T)^{-1}\nabla \xi_v + 
\nabla h_v \circ \Theta_v \right\vert^2 \, J_v \, \rd (x,z) 
- \frac{1}{2} \int_D  \sigma \vert (\xi_v + h_v)(x,-H) - \mathfrak{h}_v(x)\vert^2  \, \rd x\,.
\end{align*}	
We introduce the functions 
\begin{align*}
j(v) &:= (D\Theta_v^T)^{-1}  \nabla \xi_v + 
\nabla h_v \circ \Theta_v  \quad \text{ in } \Omega(u),
\\[0.2cm]
 m(v) &:=  \big( \xi_v + h_v\big)(\cdot,-H) - \mathfrak{h}_v \quad\text{ in } D\, .
\end{align*}
Then, recalling that $h$ and $\mathfrak{h}$ are $C^1$-functions in all their arguments by \eqref{zucchero}, we conclude that the Fr\'echet derivative of $E_e$ at $u$ applied to 
$\vartheta \in H^2(D)\cap H^1_0(D)$ is given by 
\begin{align*}
\partial_u  E_e(u) [\vartheta] 
&= \partial_v  E_e(v) [\vartheta]\vert_{v=u} 
= - \int_{\Omega(u)} j(u) \cdot (\partial_v j(v)[\vartheta]\vert_{v=u}) \, 
J_u \, \rd (x,z)
\\
& \qquad 
- \frac{1}{2}\int_{\Omega(u)} \vert j(u)\vert^2 \, (\partial_v J_v[\vartheta]\vert_{v=u}) \, \rd (x,z) - \int_D \sigma \, m(u)\, (\partial_v m(v)[\vartheta]\vert_{v=u}) \, \rd x\,.
\end{align*}
Using $J_u=1$, $j(u)= \nabla \chi_u+ \nabla h_u =\nabla \psi_u$ in $\Omega(u)$, 
and  $m(u) = \psi_u - \mathfrak{h}_u$ in $D$, we see that 
\begin{align*}
\partial_u  E_e(u) [\vartheta] 
&= - \int_{\Omega(u)} \nabla \psi_u \cdot \big(\partial_v j(v)[\vartheta]\vert_{v=u}\big) \, \rd (x,z)
- \frac{1}{2}\int_{\Omega(u)} \vert \nabla \psi_u \vert^2 \, \big(\partial_v J_v[\vartheta]\vert_{v=u}\big) \, \rd (x,z)
\\
&\quad - \int_D \sigma \big[ \psi_u(\cdot,-H) - \mathfrak{h}_u\big] \, (\partial_v m(v)[\vartheta]\vert_{v=u}) \, \rd x\,.
\end{align*}
Since 
\begin{equation*}
\partial_v J_v[\vartheta]\vert_{v=u} = \frac{\vartheta}{H+u} \quad \text{ in } D
\end{equation*}
and 
\begin{equation*}
\partial_v m(v)[\vartheta]\vert_{v=u} = 
(\partial_v \xi_v[\vartheta]\vert_{v=u})(\cdot,-H)  + (\partial_w h)_u(\cdot,-H) \, \vartheta - (\partial_w \mathfrak{h})_u \,\vartheta \quad \text{ in } D\,,
\end{equation*}
it follows that 
\begin{align*}
\partial_u  E_e(u) [\vartheta] 
&= -  \int_{\Omega(u)} \nabla \psi_u 
\cdot \big(\partial_v j(v)[\vartheta]\vert_{v=u}\big) \, \rd (x,z)
- \frac{1}{2}\int_{\Omega(u)} \vert \nabla \psi_u \vert^2 
\frac{\vartheta}{H+u} \, \rd (x,z)
\\
& \quad - \int_D   \sigma \big[ \psi_u(\cdot,-H) - \mathfrak{h}_u\big]
\big[ (\partial_v \xi_v[\vartheta]\vert_{v=u})(\cdot,-H)  + (\partial_w h)_u(\cdot,-H) \, \vartheta - (\partial_w \mathfrak{h})_u \,\vartheta \big] \, \rd x\,.
\end{align*}
Using that $\Theta_u$ is the identity on $\Omega(u)$, $D\Theta_u=\mathrm{id}$, and that $\xi_u= \chi_u$, we 
compute from the definition of $j(v)$ that 
\begin{align*}
\partial_v j(v)[\vartheta]\vert_{v=u} 
= -  \partial_v (D\Theta_v^T)[\vartheta]\vert_{v=u} \, \nabla \chi_u + 
\partial_v (\nabla \xi_v)[\vartheta]\vert_{v=u} 
 + \partial_v (\nabla h_v \circ \Theta_v) [\vartheta]\vert_{v=u}  
\end{align*}
in $\Omega(u)$. Now,
\begin{equation*}
- \partial_v (D\Theta_v^T)[\vartheta]\vert_{v=u} \, \nabla \chi_u  = - \partial_z \chi_u \nabla \left( \frac{(z+H)\vartheta}{H+u}\right) \quad \text{ in } \Omega(u)
\end{equation*}
and 
\begin{equation*}
\partial_v (\nabla \xi_v)[\vartheta]\vert_{v=u} 
= \nabla \left( \partial_v  \xi_v [\vartheta]\vert_{v=u} \right)
\quad \text{ in } \Omega(u)\,. 
\end{equation*}
Moreover, 
\begin{equation*}
\partial_v (\nabla h_v \circ \Theta_v) [\vartheta]\vert_{v=u}   
= \nabla \big((\partial_w h)_u \vartheta \big) + 
\frac{(z+H)\vartheta}{H+u} \, \nabla \big((\partial_z h)_u \big)
\quad \text{ in } \Omega(u)\,.
\end{equation*}
The above three identities yield
\begin{equation}\label{w10}
\begin{split}
\partial_u  E_e(u) [\vartheta] 
&=- \int_{\Omega(u)}  \nabla \psi_u \cdot 
\nabla \big( \partial_v  \xi_v [\vartheta]\vert_{v=u} + (\partial_w h)_u \vartheta
\big) \,\rd (x,z) 
\\
& \quad + \int_{\Omega(u)}  \nabla \psi_u \cdot 
\left[
\partial_z \chi_u \, \nabla \left( \frac{(z+H)\vartheta}{H+u}\right) - 
\frac{(z+H)\vartheta}{H+u} \nabla \big((\partial_z h)_u \big)
\right] \rd (x,z) \\
&\quad - \frac{1}{2} \int_{\Omega(u)} \vert \nabla \psi_u \vert^2 \frac{\vartheta}{H+u}
\, \rd (x,z)
 \\
& \quad - \int_D   \sigma \big[ \psi_u(\cdot,-H) - \mathfrak{h}_u\big]
	\big[ (\partial_v \xi_v[\vartheta]\vert_{v=u})(\cdot,-H)  + (\partial_w h)_u(\cdot,-H) \, \vartheta - (\partial_w \mathfrak{h})_u \,\vartheta \big] \, \rd x\,. 
\end{split}
\end{equation}
Next we shall simplify the right-hand side of \eqref{w10}. Using Gau\ss' Theorem, the fact that $\psi_u$ is a strong solution to \eqref{MBP1}, $\vartheta =0$ on $\partial D$, and the fact that 
$\partial_v \xi_v [\vartheta]\vert_{v=u}$ belongs to $H^1_B(\Omega(u))$, the first integral on the right-hand side of \eqref{w10} can be rewritten in the form 
\begin{align*}
- \int_{\Omega(u)}  &\nabla \psi_u \cdot 
\nabla \big( \partial_v  \xi_v [\vartheta]\vert_{v=u} + (\partial_w h)_u \vartheta
\big) \,\rd (x,z) 
\\
&= - \int_D (\partial_w h)_u (x,u(x)) \, \vartheta(x) 
\big[ \partial_z \psi_u - \partial_x u \,\partial_x \psi_u \big](x,u(x))\, \rd x
\\
& \quad + \int_D \big[ (\partial_v \xi_v[\vartheta]\vert_{v=u})(x,-H)  + (\partial_w h)_u(x,-H) \, \vartheta(x) \big] \, \partial_z \psi_u(x,-H) \, \rd x\,.
\end{align*}
Since, due to \eqref{MBP3},
\begin{equation*}
\partial_z \psi_u(x,-H) = \sigma(x) \big[\psi_u(x,-H) - \mathfrak{h}_u(x)\big]\,, 
\quad x \in D, 
\end{equation*}
it follows that 
\begin{align}
- \int_{\Omega(u)}  &\nabla \psi_u \cdot 
\nabla \big( \partial_v  \xi_v [\vartheta]\vert_{v=u} + (\partial_w h)_u \vartheta
\big) \,\rd (x,z) \nonumber
\\
&= - \int_D  \vartheta(x) 
\Big[ (\partial_w h)_u \big( \partial_z \psi_u - \partial_x u \,\partial_x \psi_u \big) \Big](x,u(x))\, \rd x
\label{w11} \\
& \quad + 
\int_D  \sigma(x) \big[\psi_u(x,-H) - \mathfrak{h}_u(x)\big]
\big[ (\partial_v \xi_v[\vartheta]\vert_{v=u}) (x,-H) + (\partial_w h)_u(x,-H) \, \vartheta(x) \big] \, \rd x\,. \nonumber
\end{align}
On account of $(\partial_z h)_u = \partial_z \psi_u - \partial_z \chi_u$ in $ \Omega(u)$, the second integral on the right-hand side of \eqref{w10} can be written as 
\begin{equation}
\label{w12}
\begin{split}
&\int_{\Omega(u)}  \nabla \psi_u \cdot 
\left[
\partial_z \chi_u \, \nabla \left( \frac{(z+H)\vartheta}{H+u}\right) - 
\frac{(z+H)\vartheta}{H+u} \nabla \big((\partial_z h)_u \big)
\right] \rd (x,z) 
\\[0.1cm]
& = \int_{\Omega(u)} 
\nabla \psi_u \cdot \nabla \left( \partial_z \chi_u \, \frac{(z+H)\vartheta}{H+u}
\right) \rd (x,z) - \int_{\Omega(u)} 
\nabla \psi_u \cdot \nabla \big( \partial_z \psi_u\big) \, \frac{(z+H)\vartheta}{H+u}\, \rd (x,z)\, . 
\end{split}
\end{equation}
Thanks again to Gau\ss' Theorem and using \eqref{MBP1} and the fact that 
\begin{equation}
\label{w13}
(x,z) \mapsto (z+H) \vartheta(x) \text{ vanishes on } \partial D \times (-H,0) 
\text{ and on } D\times \{-H\},
\end{equation}
we obtain 
\begin{align*}
\int_{\Omega(u)} 
&\nabla \psi_u \cdot \nabla \left( \partial_z \chi_u \, \frac{\vartheta(z+H)}{H+u}
\right) \rd (x,z) 
\\
&= \int_D \vartheta(x) \left(\partial_z \chi_u
\big[ \partial_z \psi_u - \partial_x u \,\partial_x \psi_u \big]\right)(x,u(x))\, \rd x 
\\
&= \int_D \vartheta(x) \, \Big[ \big( \partial_z \psi_u -(\partial_z h)_u\big)
\big( \partial_z \psi_u - \partial_x u \,\partial_x \psi_u \big) \Big](x,u(x))\, \rd x\, .
\end{align*}
We write the second integral in \eqref{w12} in the form 
\begin{equation*}
- \int_{\Omega(u)} 
\nabla \psi_u \cdot \nabla \big( \partial_z \psi_u\big) \, \frac{(z+H)\vartheta}{H+u}\, \rd (x,z) 
= - \frac{1}{2} \int_{\Omega(u)} \partial_z (\vert \nabla \psi_u \vert^2) \, 
\frac{(z+H)\vartheta}{H+u}\, \rd (x,z) 
\end{equation*}
and use integration by parts and \eqref{w13} to get 
\begin{align*}
- \int_{\Omega(u)} &
\nabla \psi_u \cdot \nabla \big( \partial_z \psi_u\big) \, \frac{\vartheta(z+H)}{H+u}\, \rd (x,z) 
\\ 
& = \frac{1}{2} \int_{\Omega(u)} \frac{\vartheta}{H+u}\, \vert \nabla \psi_u \vert^2 \,  \rd (x,z) 
- \frac{1}{2} \int_D \vartheta(x) \, \vert \nabla \psi_u(x,u(x)) \vert^2 \,  \rd x\,.
\end{align*}
Therefore, we deduce from \eqref{w12} that 
\begin{align*}
\int_{\Omega(u)}  &\nabla \psi_u \cdot 
\left[
\partial_z \chi_u \, \nabla \left( \frac{(z+H)\vartheta}{H+u}\right) - 
\frac{(z+H)\vartheta}{H+u} \nabla \big((\partial_z h)_u \big)
\right] \rd (x,z) 
\\
&= \frac{1}{2} \int_{\Omega(u)} \frac{\vartheta}{H+u}\, \vert \nabla \psi_u \vert^2 \,  \rd (x,z) 
- \frac{1}{2} \int_D \vartheta(x) \, \vert \nabla \psi_u(x,u(x)) \vert^2 \,  \rd x
\\
& \quad + \int_D \vartheta(x) \, \Big[ \big( \partial_z \psi_u -(\partial_z h)_u\big)
\big( \partial_z \psi_u - \partial_x u \,\partial_x \psi_u \big) \Big](x,u(x))\, \rd x\, .
\end{align*}
Combining this identity with \eqref{w10} and \eqref{w11} yields
\begin{equation}
\label{w14}
\begin{split}
\partial_u  E_e(u) [\vartheta] &= 
\int_D \vartheta(x) \,\Big[ \big( \partial_z \psi_u -(\partial_z h)_u - (\partial_wh)_u\big) \big( \partial_z \psi_u - \partial_x u \,\partial_x \psi_u \big) \Big](x,u(x))\, \rd x
\\
& \quad - \frac{1}{2}\int_D  \vartheta(x) \, \vert \nabla \psi_u(x,u(x)) \vert^2 \,  \rd x + \int_D \sigma(x) \big[\psi_u(x,-H) - \mathfrak{h}_u(x)\big] \, (\partial_w \mathfrak{h})_u(x) \,\vartheta(x)  \, \rd x\, .
\end{split}
\end{equation}
Since \eqref{MBP2} entails $\psi_u(x,u(x))= h(x,u(x),u(x))$, $x\in D$, we have 
\begin{equation*}
\partial_x \psi_u(x,u(x)) = (\partial_x h)_u(x,u(x)) - \partial_xu(x) 
\big[ \partial_z \psi_u - (\partial_z h)_u-(\partial_w h)_u
\big](x,u(x)), \quad x \in D,
\end{equation*}
and hence, for $x\in D$, 
\begin{equation*}
\begin{split}
\frac{1}{2}&\big\vert\nabla\psi_{u}(x,u(x))\big\vert^2-\Big[ \big( \partial_z\psi_{u}-(\partial_z h)_u-(\partial_w h)_u\big) \big( \partial_z\psi_{u} - \partial_x u\, \partial_x\psi_{u}\big) \Big](x, u(x))\\
& = -\frac{1}{2}(1+\vert\partial_x u(x)\vert^2)\,\big[\partial_z\psi_{u}-(\partial_z h)_u-(\partial_w h)_u\big]^2(x, u(x))\\
&\quad +\frac{1}{2}\big[\big\vert(\partial_x h)_u\big\vert^2+ \big((\partial_z h)_u+(\partial_w h)_u\big)^2\big](x, u(x))\,.
\end{split}
\end{equation*}
Inserting this identity into \eqref{w14} gives 
\begin{align*}
\partial_u  E_e(u) [\vartheta] 
&= \frac{1}{2} \int_D (1+ \vert \partial_x u(x)\vert^2) 
\big[ \partial_z \psi_u -(\partial_z h)_u - (\partial_w h)_u \big]^2 (x,u(x)) \, 
\vartheta(x) \,  \rd x
\\
&\quad - \frac{1}{2} \int_D \big[\vert (\partial_x h)_u\vert^2 + 
\left((\partial_z h)_u + (\partial_w h)_u\right)^2\big](x,u(x)) 
\, \vartheta(x) \,  \rd x
\\
&\quad + \int_D \sigma(x) \big[\psi_u(x,-H) - \mathfrak{h}_u(x)\big] \ (\partial_w \mathfrak{h})_u(x) \,  \vartheta(x) \,  \rd x\\
& =\int_D g(u)(x)\,\vartheta(x) \,  \rd x\,,
\end{align*}
according to \eqref{g}. Finally,  the continuity of 
$$
\partial_u E_e : S \rightarrow \mathcal{L}\big(H^2(D)\cap H_0^1(D),\mathbb{R}\big)
$$
readily follows from Theorem~\ref{CC}.
\end{proof}

We finally provide the differentiability property of $E_e$ on the closed set $\bar S$. More precisely, we show that $E_e$ admits a directional derivative at a point $u\in \bar S$ in any direction of $-u+S$, which is given by $g(u)$ defined in~\eqref{GG}. Recall that $\mathcal{C}(u)$ may be non-empty in this case.

\begin{proposition}
	\label{P516}
	Let $u\in \bar S$ and $w\in S$. Then
	\begin{equation*}
	\lim_{s\rightarrow 0^+} \frac{1}{s}\Big[ E_e(u+s(w-u))-E_e(u) \Big] =\int_D g(u) (w-u)\, \rd x\,.
	\end{equation*} 
\end{proposition}

\begin{proof}
	Fix $w\in S$ and note that
	\begin{equation*}
	u_s:= u+s(w-u)=(1-s)u+sw \in S\,,\quad s\in (0,1)\,.
	\end{equation*}
Since $u_s\in S$ for $s\in (0,1)$, we obtain from Proposition~\ref{P416} that
\begin{equation*}
	\frac{\rd}{\rd s} E_e(u_s) = \int_D g(u_s) (w-u)\,\rd x
\end{equation*}
for $s\in (0,1)$. Therefore, letting $s\rightarrow 0$, we derive with the help of Theorem~\ref{CC} that
	\begin{equation}
	\label{pa}
	\begin{split}
	\lim_{s\rightarrow 0^+}\frac{\rd}{\rd s} E_e(u_s)=&\int_D g(u)\, (w-u)\, \rd x\,.
	\end{split}
	\end{equation}
Now, Theorem~\ref{CC} guarantees that $E_e(u_s)\rightarrow E_e(u)$ as $s\rightarrow 0$,  so that
	\begin{equation*}
	\begin{split}
	E_e(u_t)-E_e(u)= \int_0^t \frac{\rd}{\rd s} E_e(u_s)\,\rd s\,,\quad t\in (0,1)\,,
	\end{split}
	\end{equation*}
	and we conclude from \eqref{pa} that
	\begin{equation*}
	\begin{split}
	\lim_{t\rightarrow 0^+} \frac{1}{t}\big(E_e(u_t)-E_e(u)\big)&= \lim_{t\rightarrow 0^+} \frac{1}{t}\int_0^t \frac{\rd}{\rd s} E_e(u_s)\,\rd s=\int_D g(u) (w-u)\, \rd x\,,
	\end{split}
	\end{equation*} 
as claimed.
\end{proof}

\section{Proofs of Theorem~\ref{Thm22} and Theorem~\ref{Thm33} for $\alpha=0$}\label{proof}

In this section we deal with the case $\alpha=0$ and recall that the total energy is then given by
$$
E(u)= E_m(u)+E_e(u)
$$ 
with mechanical energy 
$$
E_m(u)=\frac{\beta}{2}\|\partial_x^2u\|_{L_2(D)}^2 +  \frac{\tau}{2} \|\partial_x u\|_{L_2(D)}^2
$$
and electrostatic energy 
$$
E_e(u)=-\dfrac{1}{2}\displaystyle\int_{\Omega(u)}  \big\vert\nabla \psi_u\big\vert^2\,\rd (x,z)\\
 -\dfrac{1}{2}\displaystyle\int_{ D} \sigma(x) \big\vert \psi_u(x,-H)-\mathfrak{h}_{u}(x)\big\vert^2\,\rd x\,.
$$

\subsection{Existence of a Minimizer of a Regularized Energy}\label{sec61}

As already noted in \cite{LW19}, the boundedness from below of the functional $E$ is \textit{a priori} unclear since $\alpha=0$. To cope with this issue, we work with the regularized functional given by
\begin{equation}\label{Ek}
\mathcal{E}_k(u):= E(u) + \frac{A}{2} \|(u-k)_+\|_{L_2(D)}^2\,,\quad u\in \bar S_0\,,
\end{equation}
for $k\ge  H$, where 
$$
A:= 8\left(\frac{K^4}{\beta}  + 2 K^2 \right)\,,
$$
and the constant $K$ is introduced in \eqref{hbound}.

\begin{lemma}\label{L333}
For each $k\ge H$ the functional $\mathcal{E}_k$ is bounded from below with
$$
\mathcal{E}_k(u)\ge \frac{\beta}{4}\|\partial_x^2 u\|_{L_2(D)}^2+\frac{A}{4}\|(u-k)_+\|_{L_2(D)}^2- c(k)
$$
for some constant $c(k)>0$.
\end{lemma}

\begin{proof}
By \eqref{bb5}, \eqref{s0P}, and Proposition~\ref{lemt9P}, 
\begin{align*}
-E_e(u)&=\mathcal{G}(u)[\psi_u-h_u] \le \mathcal{G}(u)[0] \\
& =\frac{1}{2}\int_{\Omega(u)} \vert \nabla h_u\vert^2\, \rd (x,z) +\frac{1}{2}\int_D \sigma(x)\big[h_u(x,-H)-\mathfrak{h}_u(x)\big]^2\,\rd x\\
& \le \int_{\Omega(u)} \left[ (\partial_x h)_u^2 + |\partial_x u|^2 (\partial_w h)_u^2 + (\partial_z h)_u^2 \right]\,\rd x + \bar{\sigma} \int_D \Big\{ \big[h_u(x,-H)\big]^2 + \big[\mathfrak{h}_u(x)\big]^2 \Big\}\,\rd x \\
&\le K^2\int_{\Omega(u)} \left( 2 \frac{1+u(x)^2}{H+u(x)}+\frac{\vert\partial_x u(x)\vert^2}{H+u(x)}\right)\,\rd (x,z) + 2 \bar\sigma K^2|D|\\
& \le K^2 \left( 2|D| + 2 \|u\|_{L_2(D)}^2 + \|\partial_x u\|_{L_2(D)}^2 \right) + 2 \bar\sigma K^2|D|\\
& =2(1+\bar{\sigma}) |D| K^2 + 2K^2 \|u\|_{L_2(D)}^2 + K^2 \|\partial_x u\|_{L_2(D)}^2\,.
\end{align*}
Now, since $u\in \bar S$,
$$
\int_D \vert\partial_x u\vert^2\,\rd x=-\int_D u\partial_x^2u\,\rd x\le \|u\|_{L_2(D)} \|\partial_x^2u\|_{L_2(D)}\,,
$$
and we further obtain with the help of Young's inequality that
\begin{align*}
-E_e(u)&\le  2(1+\bar{\sigma}) |D| K^2 + 2 K^2 \|u\|_{L_2(D)}^2 + K^2 \|u\|_{L_2(D)} \|\partial_x^2u\|_{L_2(D)}\\
&\le 2(1+\bar{\sigma}) |D| K^2 +\left(\frac{K^4}{\beta}  +2K^2 \right) \| u\|_{L_2(D)}^2+\frac{\beta}{4} \|\partial_x^2u\|_{L_2(D)}^2 \,.
\end{align*}
Using this estimate in the definition of $\mathcal{E}_k(u)$ along with
\begin{align*}
\|u\|_{L_2(D)}^2 & = \int_D u^2 \mathbf{1}_{(k,\infty)}(u)\,\rd x + \int_D u^2 \mathbf{1}_{[-H,k]}(u)\,\rd x \\
& \le 2 \|(u-k)_+\|_{L_2(D)}^2 + 2 k^2 \int_D \mathbf{1}_{(k,\infty)}(u)\,\rd x + k^2 \int_D \mathbf{1}_{[-H,k]}(u)\,\rd x \\
& \le 2 \|(u-k)_+\|_{L_2(D)}^2 + 2 k^2 |D|\,,
\end{align*}
 we derive
\begin{align*}
\mathcal{E}_k(u)&\ge  \frac{\beta}{4} \|\partial_x^2u\|_{L_2(D)}^2  - 2(1+\bar{\sigma}) |D| K^2 - \left(\frac{K^4}{\beta}  +2K^2 \right) \| u\|_{L_2(D)}^2 + \frac{A}{2}\|(u-k)_+\|_{L_2(D)}^2
\\
&\ge \frac{\beta}{4} \|\partial_x^2u\|_{L_2(D)}^2  - c(k)  + \left[\frac{A}{2}-2\left(\frac{K^4}{\beta}  + 2 K^2 \right)\right] \|(u-k)_+\|_{L_2(D)}^2\\
&\ge 
\frac{\beta}{4} \|\partial_x^2u\|_{L_2(D)}^2 + \frac{A}{4} \|(u-k)_+\|_{L_2(D)}^2 - c(k)\,,
\end{align*}
thereby completing the proof.
\end{proof}

Due to the weak lower semicontinuity of $E_m$ in $H^2(D)$ and the continuity of $E_e$ with respect to the weak topology of $H^2(D)$ (see Theorem~\ref{CC}), Lemma~\ref{L333} allows us to apply the direct method of the calculus of variations to derive the existence of a minimizer of $\mathcal{E}_k$ in $\bar{S}_0$. 

\begin{corollary}\label{L334}
 For each $k\ge H$, the functional $\mathcal{E}_k$ has at least one minimizer $u_k\in \bar S_0$; that is,
\begin{equation}\label{mini}
\mathcal{E}_k(u_k)=\min_{\bar S_0} \mathcal{E}_k\,.
\end{equation}
\end{corollary}

\subsection{Derivation of the Euler-Lagrange Equation  for the Regularized Energy}\label{sec62}

We shall next identify the Euler-Lagrange equation satisfied by a minimizer of the regularized energy $\mathcal{E}_k$ on $\bar S_0$.

\begin{proposition}\label{prop3}
 Let $k\ge H$ and let $u\in \bar{S_0}$ be a minimizer of $\mathcal{E}_k$ on $\bar S_0$. Then $u$ is an $H^2$-weak solution to the variational inequality
\begin{subequations}\label{bennygoodman}
\begin{equation}
\beta\partial_x^4u-  \tau \partial_x^2 u+
A (u-k)_++\partial\mathbb{I}_{\bar{S_0}}(u) \ni -g(u) \;\;\text{ in }\;\; D\,, \label{bennygoodman1}
\end{equation}
where $\partial\mathbb{I}_{\bar{S_0}}$ is the subdifferential of the indicator function $\mathbb{I}_{\bar S_0}$ of the closed convex subset $\bar{S_0}$ of $H^2(D)$; that is, 
\begin{equation}
\begin{split}
\int_D &\left\{\beta\partial_x^2 u\,\partial_x^2 (w-u) + \tau \partial_x u\, \partial_x(w-u)  +A(u-k)_+ (w-u)\right\}\,\rd x \ge -\int_D g(u) (w-u)\, \rd x  
\end{split} \label{bennygoodman2}
\end{equation}
\end{subequations}
for all $w\in \bar{S_0}$. 
\end{proposition}

\begin{proof}
Let $k\ge  H$ be fixed. Consider a minimizer $u\in\bar{S_0}$ of $\mathcal{E}_k$ on $\bar{S_0}$ and fix $w\in S_0 := \bar{S}_0\cap S$. Owing to the convexity of $\bar{S_0}$, the function $u+s(w-u)=(1-s)u+sw$ belongs to $S_0$  for all $s\in (0,1]$ and the minimizing property of $u$ guarantees that
$$
0\le \liminf_{s\rightarrow 0^+} \frac{1}{ s}\big(\mathcal{E}_k(u+s(w-u))-\mathcal{E}_k(u)\big)\,.
$$
Since $u\in \bar{S}_0\subset \bar{S}$ and $w\in S_0\subset S$, Proposition~\ref{P516} implies that
\begin{align*}
0 & \le  \int_D \left\{\beta\partial_x^2 u\,\partial_x^2 (w-u) +  \tau \partial_x u\, \partial_x (w-u)+A(u-k)_+ (w-u)\right\}\,\rd x  + \int_D g(u) (w-u)\, \rd x
\end{align*}
for all $w\in S_0$. Since $S_0$ is dense in $\bar{S_0}$ and $(u,g(u))$ belongs to $H^2(D)\times L_2(D)$,  this inequality also holds for any $w\in \bar{S_0}$.
\end{proof}

\begin{proposition}\label{prop4}
There is $\kappa_0 \ge H$ depending only on $K$ such that, if $u\in \bar{S_0}$ is any solution to the variational inequality \eqref{bennygoodman} with  $k\ge H$, then $\|u\|_{L_\infty(D)}\le \kappa_0$.
\end{proposition}

\begin{proof} Owing  to the continuous embedding of $H_0^1(D)$ in $C(\bar{D})$, the function $u$ belongs to $C(\bar{D})$ with $u(\pm L)=0$. Consequently, the set $\{ x \in D\,:\, u(x)>-H\}$ is a non-empty open subset of $D$ and we can write it as a countable union of disjoint open intervals $(I_j)_{j\in J}$, see \cite[IX.Proposition~1.8]{AEIII}. Using once more the property $u(\pm L)=0>-H$, we may assume without loss of generality that $I_0=(-L,a_0)$ and $I_1=(b_0,L)$ for some $-L<a_0<b_0<L$, and $\bar{I}_j\subset (-L,L)$ for $j\in J$ with $j\ge 2$.
	
\noindent {\bf Step 1:} Thanks to \eqref{bb7} and \eqref{hbound1}, we infer from Lemma~\ref{MP} that $|\psi_u|\le K$ in $\Omega(u)$. Combining this bound with \eqref{bb5}, \eqref{hbound}, \eqref{GG}, and \eqref{s0P} readily gives
\begin{equation}\label{bono}
g(u)(x)\ge  -2\bar\sigma K^2 - K^2=:-G_0\,, \quad x\in D\,.
\end{equation}

\noindent {\bf Step 2:} Consider first $j\in J$ with $j\ge 2$ and let $\theta\in \mathcal{D}(I_j)$. Since $u>-H$ in the support of $\theta$, the function $u\pm \delta\theta$ belongs to $S_0$ for $\delta>0$ small enough. We thus infer from \eqref{bennygoodman2} that
\begin{equation*}
\pm\delta\int_{I_j} \left\{\beta\partial_x^2 u\,\partial_x^2 \theta+\tau \partial_x u\, \partial_x\theta  +A(u-k)_+ \theta\right\}\,\rd x \ge \mp\delta\int_{I_j} g(u) \theta\, \rd x  \,,
\end{equation*}
hence
\begin{equation*}
\int_{I_j} \left\{\beta\partial_x^2 u\,\partial_x^2 \theta+\tau \partial_x u\, \partial_x\theta  +A(u-k)_+ \theta\right\}\,\rd x =-\int_{I_j} g(u) \theta\, \rd x  \,.
\end{equation*}
Consequently, using the function $S_{I_j}$ defined in Proposition~\ref{propA3.1}, we realize that $u-S_{I_j} \in H^2(I_j)$ is a weak solution to the boundary value problem
\begin{subequations}\label{ep}
\begin{align}
\beta\partial_x^4 w-\tau\partial_x^2 w& = -G_0-g(u)-A (u-k)_+\;\;\text{ in }\;\; I_j\,,\label{ep1}\\
&  w=\partial_x w=0 \;\;\text{ in }\;\; \partial I_j\,, \label{ep2}
\end{align}
\end{subequations}
the boundary conditions \eqref{ep2} being a consequence of the definition of $I_j$, $j\ge2$,  the $H^2(D)$-regularity of $u$, and the constraint $u\ge -H$.
Taking into account that $g(u)+A (u-k)_+\in L_2(I_j)$ by Theorem~\ref{CC}, classical elliptic regularity theory implies that $ u-S_{I_j}\in H^4(I_j)$ is a strong solution to \eqref{ep}. Since the right hand side of \eqref{ep1} is non-positive due to \eqref{bono}, it now follows from a version of Boggio's comparison principle \cite{Bo05, Gr02, LW15, Ow97} that $u-S_{I_j} < 0$ in $I_j$, so that $u(x)\le  \kappa_0$ for $x\in \bar I_j$ and $j\ge 2$ by Proposition~\ref{propA3.1}.

\noindent {\bf Step 3:} We next handle the case $j=0$ in which $I_0=(-L,a_0)$. We first argue as in the previous step to conclude that
\begin{equation}
\int_{I_0} \left\{\beta\partial_x^2 u\,\partial_x^2 \theta+\tau \partial_x u\, \partial_x\theta + A(u-k)_+ \theta \right\}\,\rd x =-\int_{I_0} g(u) \theta\, \rd x \label{fp3}
\end{equation}
for all $\theta\in \mathcal{D}(I_0)$ and that $u(-L)=\partial_x u(-L)=u(a_0)+H = \partial_x u(a_0)=0$. Consequently, we infer from \eqref{fp3} and Proposition~\ref{propA3.1} that $u-S_{I_0}\in H^2(I_0)$ is a weak solution to the boundary value problem
\begin{align*}
\beta\partial_x^4 w-\tau\partial_x^2 w & = -G_0 -g(u)-A (u-k)_+\;\;\text{ in }\;\; I_0\,,\\
&  w=\partial_x w=0 \;\;\text{ on }\;\; \partial I_0\,. 
\end{align*}
We then argue as in \textbf{Step~2} to establish that $u-S_{I_0} < 0$ in $I_0=(-L,a_0)$. Hence, $u\le \kappa_0$ in $[-L,a_0]$ by Proposition~\ref{propA3.1}.

\noindent {\bf Step 4:} For the case $I_1$ we proceed as in {\bf Step~3} using Proposition~\ref{propA3.1} to deduce that $u\le \kappa_0$ in $[b_0,L]$. This completes the proof.
\end{proof}

\subsection{Proof of Theorem~\ref{Thm22} for $\alpha=0$}\label{sec63}

Let $k\ge H$ and consider a minimizer $u_k\in \bar S_0$ of the functional $\mathcal{E}_k$ on $\bar S_0$ as provided by Corollary~\ref{L334}. Then, $-H\le u_k\le \kappa_0$ in $D$ according to Proposition~\ref{prop4}. Therefore, if $k\ge \kappa_0$, then
\begin{equation}\label{888}
E(u_k)=\mathcal{E}_{\kappa_0}(u_k)=\mathcal{E}_k(u_k)\le \mathcal{E}_k(v) = E(v) + \frac{A}{2} \|(v-k)_+\|_{L_2(D)}^2\,,\quad v\in \bar S_0\,.
\end{equation}
Now, it follows from Lemma~\ref{L333} and the fact that $0\in \bar S_0$ that, for $k\ge \kappa_0$,
$$
\frac{\beta}{4}\|\partial_x^2 u_k\|_{L_2(D)}^2 \le \mathcal{E}_{\kappa_0}(u_k)+ c(\kappa_0)\le  \mathcal{E}_{k}(0)+ c(\kappa_0)=E(0)+ c(\kappa_0)\,.
$$
Therefore, $(u_k)_{k\ge \kappa_0}$ is bounded in $H^2(D)$ and there is a subsequence  of $(u_k)_{k\ge \kappa_0}$ (not relabeled) which converges weakly in $H^2(D)$ and strongly in $H^1(D)$ towards some $u_*\in \bar S_0$. Due to the weak lower semicontinuity of $E_m$ in $H^2(D)$ and the continuity of $E_e$ with respect to the weak topology of $H^2(D)$ (see Theorem~\ref{CC}), we readily infer from \eqref{888} that
\begin{equation*}
E(u_*)\le E(v)\,,\quad v\in \bar S_0\,,
\end{equation*}
after taking into account that
\begin{equation*}
\lim_{k\to\infty} \|(v-k)_+\|_{L_2(D)} = 0\,, \qquad v\in L_2(D)\,.
\end{equation*}
Consequently, $u_*\in \bar S_0$ is a minimizer of $E$ on $\bar S_0$. This proves Theorem~\ref{Thm22}.

\subsection{Proof of Theorem~\ref{Thm33} for $\alpha=0$}\label{sec64}

Let $u\in \bar S_0$ be any minimizer of $E$ on $\bar S_0$. 
Proceeding as in the proof of Proposition~\ref{prop3}, this implies that $u\in \bar S_0$
is an $H^2$-weak solution to the variational inequality
$$
\beta\partial_x^4u- \tau \partial_x^2 u+
+\partial\mathbb{I}_{\bar{S_0}}(u) \ni -g(u) \;\;\text{ in }\;\; D\,, 
$$
which completes the proof of Theorem~\ref{Thm33}.

\section{Proofs of Theorem~\ref{Thm22} and Theorem~\ref{Thm33} for $\alpha>0$}\label{proof2}

Consider now $\alpha>0$. In that case, the total energy is given by
$$
E(u)= E_m(u)+E_e(u)
$$ 
with mechanical energy 
$$
E_m(u)=\frac{\beta}{2}\|\partial_x^2u\|_{L_2(D)}^2 +\left(\frac{\tau}{2}+\frac{\alpha}{4}\|\partial_x u\|_{L_2(D)}^2\right)\|\partial_x u\|_{L_2(D)}^2
$$
and electrostatic energy 
$$
E_e(u)=-\dfrac{1}{2}\displaystyle\int_{\Omega(u)}  \big\vert\nabla \psi_u\big\vert^2\,\rd (x,z)\\
-\dfrac{1}{2}\displaystyle\int_{ D} \sigma(x) \big\vert \psi_u(x,-H)-\mathfrak{h}_{u}(x)\big\vert^2\,\rd x\,.
$$
Observe that, since $\alpha>0$, the mechanical energy $E_m$ features a super-quadratic term in $\|\partial_x u\|_{L_2(D)}$ which has the following far-reaching consequence.

\begin{lemma}\label{kraftwerk}
The functional $E$ is bounded from below with
\begin{equation*}
E(u) \ge \frac{\beta}{4} \|\partial_x^2 u\|_{L_2(D)}^2 -c
\end{equation*}
for some constant $c>0$. 
\end{lemma}

\begin{proof}
As in the proof of Lemma~\ref{L333}, we deduce from \eqref{bb5}, \eqref{s0P}, and Proposition~\ref{lemt9P} that
\begin{equation*}
-E_e(u) \le 2(1+\bar{\sigma}) |D| K^2 +\left(\frac{K^4}{\beta}  +2K^2 \right) \| u\|_{L_2(D)}^2+\frac{\beta}{4} \|\partial_x^2u\|_{L_2(D)}^2 \,.
\end{equation*}
Therefore, since $\|u\|_{L_2(D)} \le 2 |D| \|\partial_x u\|_{L_2(D)}$ by Poincar\'e's inequality, it follows fromYoung's inequality that
\begin{align*}
E(u) & \ge \frac{\beta}{4} \|\partial_x^2u\|_{L_2(D)}^2 + \frac{\alpha}{4} \|\partial_x u\|_{L_2(D)}^4 - 2(1+\bar{\sigma}) |D| K^2 - \left(\frac{K^4}{\beta}  + 2K^2 \right) \| u\|_{L_2(D)}^2 \\
& \ge  \frac{\beta}{4} \|\partial_x^2u\|_{L_2(D)}^2 + \frac{\alpha}{4} \|\partial_x u\|_{L_2(D)}^4 - 2(1+\bar{\sigma}) |D| K^2 - 4 |D|^2\left(\frac{K^4}{\beta}  + 2K^2 \right) \|\partial_x u\|_{L_2(D)}^2 \\
& \ge  \frac{\beta}{4} \|\partial_x^2u\|_{L_2(D)}^2 + \frac{\alpha}{8} \|\partial_x u\|_{L_2(D)}^4 - c\,,
\end{align*}
and the proof is complete.
\end{proof}
Once Lemma~\ref{kraftwerk} is established, the existence of a minimizer of $E$ on $\bar S_0$ follows from the weak lower semicontinuity of $E_m$ in $H^2(D)$ and the continuity of $E_e$ with respect to the weak topology of $H^2(D)$ (see Corollary~\ref{L3}) by the direct method of the calculus of variations, hence Theorem~\ref{Thm22} for $\alpha>0$ (see also \cite[Theorem~5.1]{LW19}). As for the proof of Theorem~\ref{Thm33} for $\alpha>0$, it is the same as that for $\alpha=0$, see Section~\ref{sec64}.

\appendix

\section{A technical lemma}

\begin{lemma}\label{lemA1}
Let $I$ and $J$ be two bounded intervals in $\mathbb{R}$, and let $U$ be a bounded open subset of $I \times J$. 
Consider $\vartheta \in H^1(U)$ and  functions $v\in C(\bar{I})$, $w\in C(\bar{I})$, and $\rho \in C(\bar{I})$, 
$\rho \geq 0$, such that 

\begin{itemize}
		\item [(a)] $x \mapsto \vartheta(x,v(x))$ and $x \mapsto \vartheta(x,w(x))$ are well-defined and belong 
		to $L_2(I,\rho \, \mathrm{d}x)$;
		\item [(b)] $\{ (x,z) \in I \times J : \min\{v(x),w(x)\} < z< \max\{v(x),w(x)\}\} \subset \bar{U}$.
\end{itemize}
Then
\begin{equation*}
\int_I  \vert \vartheta(\cdot, v) - \vartheta(\cdot, w)\vert^2 \rho\, \mathrm{d}x \leq \Vert  (v-w) \rho\Vert_{L_{\infty}(I)} \Vert \partial_z \vartheta \Vert_{L_2(U)}^2. 
\end{equation*}
\end{lemma}
\begin{proof}
Owing to (b) we have, for a.a. $x \in I$,
\begin{equation*}
 \vert \vartheta(x, v(x)) - \vartheta(x, w(x))\vert^2 = \left( \int_{w(x)}^{v(x)} \partial_z \vartheta(x,z)  \, \mathrm{d}z \right)^2\,.
\end{equation*}
Integrating with respect to $x\in I$ after multiplication by $\rho (x)$ and using H\"older's inequality give
\begin{align*}
\int_I \vert \vartheta(x, v(x)) - \vartheta(x, w(x))\vert^2 \rho(x)\, \mathrm{d}x 
&\leq \int_I  \vert v(x)-w(x)\vert \left\vert \int_{w(x)}^{v(x)} \vert\partial_z \vartheta(x,z)\vert^2   \mathrm{d}z \right\vert \,\rho(x)\,  \mathrm{d}x
\\
&\leq \Vert \rho (v-w) \Vert_{L_{\infty}(I)} \int_U  \vert \partial_z \vartheta(x,z)\vert^2 \, \mathrm{d}(x,z)
\end{align*}
and the proof is complete.
\end{proof}

\section{Proof of Lemma~\ref{lemt2Px}} \label{A1}

The proof of Lemma~\ref{lemt2Px} relies on the following result, which can be seen as an extension of \cite[Lemma~4.3.1.3]{Gr1985} to include Robin boundary conditions.

\begin{lemma}\label{lemt3P}
Let $I:=(a,b)$ and set $\mathcal{R}_I = I\times (0,1)$. Consider $\varphi\in H^2(\mathcal{R}_I)$ and $\mu\in C^2(\bar{I})$ such that
\begin{subequations}\label{t31P}
\begin{align}
\varphi(a,\eta) = \varphi(b,\eta) & = 0\ , \qquad \eta\in (0,1)\ , \label{t31aP} \\
\varphi(x,1) = - \partial_\eta\varphi(x,0) + \mu(x) \varphi(x,0) & = 0\ , \qquad x\in I\ . \label{t31bP}
\end{align}
\end{subequations}
Then
\begin{equation*}
\int_{\mathcal{R}_I} \partial_x^2\varphi \partial_\eta^2\varphi\ \mathrm{d}(x,\eta) = \int_{\mathcal{R}_I} |\partial_x \partial_\eta\varphi|^2\ \mathrm{d}(x,\eta) + \int_I \left( \partial_x\varphi \partial_x(\mu\varphi) \right)(\cdot,0)\ \mathrm{d}x\ .
\end{equation*}
\end{lemma}

\begin{proof}
We put $\xi(x,\eta) := e^{-\eta\mu(x)} \varphi(x,\eta)$ and $\rho(x,\eta) := e^{\eta\mu(x)}$ for $(x,\eta)\in \mathcal{R}_I$. Owing to the regularity of $\varphi$ and $\mu$, the function $\xi$ belongs to $H^2(\mathcal{R}_I)$ and, for $(x,\eta)\in \mathcal{R}_I$, 
\begin{align*}
\partial_x \xi(x,\eta) & = e^{-\eta\mu(x)} \left[ \partial_x\varphi(x,\eta) - \eta\partial_x \mu(x) \varphi(x,\eta) \right]\ , \\
\partial_\eta \xi(x,\eta) & = e^{-\eta\mu(x)} \left[ \partial_\eta\varphi(x,\eta) - \mu(x) \varphi(x,\eta) \right]\ .
\end{align*}
Consequently, the functions $F$ and $G$,  defined for $(x,\eta)\in \mathcal{R}_I$ by
\begin{align*}
F(x,\eta) & := \rho(x,\eta) \partial_x \xi(x,\eta) = \partial_x\varphi(x,\eta) - \eta\partial_x \mu(x) \varphi(x,\eta) \ , \\
G(x,\eta) & := \rho(x,\eta) \partial_\eta \xi(x,\eta) = \partial_\eta\varphi(x,\eta) - \mu(x) \varphi(x,\eta) \ ,
\end{align*}
 satisfy
\begin{align*}
G(a,\eta) = G(b,\eta) & = 0\ , \qquad \eta\in (0,1)\ , \\
F(x,1) = G(x,0) & = 0\ , \qquad x\in I\ ,
\end{align*}
since, by \eqref{t31P},
\begin{equation}
\begin{split}
\partial_\eta\varphi(a,\eta) = \partial_\eta\varphi(b,\eta) & = 0\ , \qquad \eta\in (0,1)\ , \\
\partial_x\varphi(x,1) & = 0\ , \qquad  x\in I\,.
\end{split}\label{t33P}
\end{equation}
 We then infer from \cite[Lemma~4.3.1.3]{Gr1985} that
\begin{align*}
\int_{\mathcal{R}_I} \partial_x(\rho\partial_x\xi) \partial_\eta(\rho\partial_\eta\xi)\ \mathrm{d}(x,\eta) & = \int_{\mathcal{R}_I} \partial_x F \partial_\eta G\ \mathrm{d}(x,\eta) = \int_{\mathcal{R}_I} \partial_\eta F \partial_x G\ \mathrm{d}(x,\eta) \\
& = \int_{\mathcal{R}_I} \partial_\eta(\rho\partial_x\xi) \partial_x(\rho\partial_\eta\xi)\ \mathrm{d}(x,\eta)\ ;
\end{align*}
that is, 
\begin{equation}
0 = \int_{\mathcal{R}_I} \left[ \partial_x^2\varphi \partial_\eta^2\varphi - |\partial_x\partial_\eta\varphi|^2 \right]\ \mathrm{d}(x,\eta) + \sum_{j=1}^3 I_j\ , \label{t32P}
\end{equation}
where
\begin{align*}
I_1 & := \int_{\mathcal{R}_I} \left[ - \partial_\eta(\mu\varphi) \partial_x^2\varphi + \partial_x(\mu\varphi)  \partial_x\partial_\eta\varphi \right]\ \mathrm{d}(x,\eta)\, , \\
I_2 & := \int_{\mathcal{R}_I} \left[ - \partial_x(\eta \varphi\partial_x \mu) \partial_\eta^2\varphi + \partial_\eta(\eta\varphi\partial_x \mu)  \partial_x\partial_\eta\varphi \right]\ \mathrm{d}(x,\eta)\ , \\
I_3 & := \int_{\mathcal{R}_I} \left[ \partial_x(\eta \varphi\partial_x \mu) \partial_\eta(\tau\varphi) - \partial_\eta(\eta\varphi\partial_x \mu)  \partial_x(\tau\varphi) \right]\ \mathrm{d}(x,\eta)\,.
\end{align*}
First, integrating by parts and using the boundary values \eqref{t31P} of $\varphi$ give
\begin{equation*}
I_3 = \int_0^1 \Big[ \eta  \varphi\partial_x \mu \partial_\eta(\mu\varphi) \Big]_{x=a}^{x=b}\ \mathrm{d}\eta - \int_I \Big[ \eta \varphi\partial_x \mu  \partial_x(\mu\varphi) \Big]_{\eta=0}^{\eta=1}\ \mathrm{d}x = 0
\end{equation*}
and
\begin{align*}
I_2 & = - \int_I \Big[ \partial_x(\eta \varphi \partial_x \mu) \partial_\eta\varphi \Big]_{\eta=0}^{\eta=1}\ \mathrm{d}x + \int_0^1 \Big[ \partial_\eta(\eta \varphi \partial_x \mu) \partial_\eta\varphi \Big]_{x=a}^{x=b}\ \mathrm{d}\eta \\
& = - \int_I \partial_x \mu \partial_x\varphi(\cdot,1) \partial_\eta\varphi(\cdot,1)\ \mathrm{d}x + \int_0^1 \partial_\eta(\eta\varphi\partial_x \mu)(b,\cdot) \partial_\eta\varphi(b,\cdot)\ \mathrm{d}\eta \\
& \quad\ - \int_0^1 \partial_\eta(\eta\varphi\partial_x \mu)(a,\cdot) \partial_\eta\varphi(a,\cdot)\ \mathrm{d}\eta\ .
\end{align*}
Owing to \eqref{t33P} we conclude that $I_2=0$. Finally, we deduce from \eqref{t31P} and \eqref{t33P} after integrating by parts that
\begin{align*}
I_1 & = - \int_0^1 \Big[ \partial_\eta(\mu \varphi) \partial_x\varphi \Big]_{x=a}^{x=b}\ \mathrm{d}\eta + \int_I \Big[ \partial_x(\mu \varphi) \partial_x\varphi \Big]_{\eta=0}^{\eta=1}\ \mathrm{d}x \\
& = - \int_0^1 \mu(b) \partial_x\varphi(b,\cdot) \partial_\eta\varphi(b,\cdot)\ \mathrm{d}\eta + \int_0^1  \mu(a) \partial_x\varphi(a,\cdot) \partial_\eta\varphi(a,\cdot)\ \mathrm{d}\eta \\
& \qquad + \int_I \partial_x(\mu\varphi)(\cdot,1) \partial_x\varphi(\cdot,1)\ \mathrm{d}x - \int_I \partial_x(\mu\varphi)(\cdot,0) \partial_x\varphi(\cdot,0)\ \mathrm{d}x \\
& = - \int_I \partial_x(\mu\varphi)(\cdot,0) \partial_x\varphi(\cdot,0)\ \mathrm{d}x\ .
\end{align*}
Collecting \eqref{t32P} and the formulas for $I_j$, $1\le j\le 3$, completes the proof.
\end{proof}

\begin{proof}[Proof of Lemma~\ref{lemt2Px}]
For $(x,\eta)\in \mathcal{R}_I$, we define 
\begin{equation}
\Phi(x,\eta) := \zeta_v(x,-H+\eta(H+v(x)))\,, \label{t25P}
\end{equation}
or, equivalently,
\begin{equation*}
\zeta_v(x,z) = \Phi\left( x , \frac{H+z}{H+v(x)} \right)\ , \qquad (x,z)\in \mathcal{O}_I(v)\,.
\end{equation*}
Since $\zeta_v\in H^2(\mathcal{O}_I(v))$ by Lemma~\ref{lemt1P} and $v\in H^2(I)$, the function $\Phi$ belongs to $H^2(\mathcal{R}_I)$ and we infer from \eqref{t5bP} and \eqref{t5cP} that 
\begin{equation}
\begin{split}
\Phi(a,\eta) = \Phi(b,\eta) & = 0\ , \qquad \eta\in (0,1)\ , \\
\Phi(x,1) = -\partial_\eta\Phi(x,0) + \sigma(x) (H+v)(x) \Phi(x,0) & = 0\ , \qquad x\in I\ .
\end{split} \label{t21P}
\end{equation}

Next,
\begin{equation}
J := \int_{\mathcal{O}_I(v)} \partial_x^2 \zeta_v \partial_z^2 \zeta_v\ \mathrm{d}(x,z) = \sum_{i=1}^3 J_i\ , \label{t24P}
\end{equation}
where
\begin{align*}
J_1 & := \int_{\mathcal{R}_I} \partial_x^2\Phi \partial_\eta^2\Phi \frac{\mathrm{d}(x,\eta)}{H+v}\ , \\
J_2 & := \int_{\mathcal{R}_I} \left[ - 2 \eta \frac{\partial_x v}{H+v} \partial_x\partial_\eta\Phi + \eta^2 \left( \frac{\partial_x v}{H+v} \right)^2 \partial_\eta^2\Phi \right] \partial_\eta^2\Phi \frac{\mathrm{d}(x,\eta)}{H+v}\ , \\
J_3 & := \int_{\mathcal{R}_I} \eta \left[ 2 \left( \frac{\partial_x v}{H+v} \right)^2 - \frac{\partial_x^2 v}{H+v} \right] \partial_\eta\Phi \partial_\eta^2\Phi \frac{\mathrm{d}(x,\eta)}{H+v}\ .
\end{align*}
Since
\begin{equation*}
\partial_x^2 \left( \frac{\Phi}{\sqrt{H+v}} \right) = \frac{\partial_x^2\Phi}{\sqrt{H+v}} - \frac{\partial_x v}{(H+v)^{3/2}} \partial_x \Phi - \frac{1}{2} \partial_x\left( \frac{\partial_x v}{(H+v)^{3/2}} \right) \Phi\ ,
\end{equation*}
we further obtain
\begin{equation*}
J_1 := \sum_{i=1}^{3} J_{1,i}\ ,
\end{equation*}
where
\begin{align*}
J_{1,1} & := \int_{\mathcal{R}_I} \partial_x^2 \left( \frac{\Phi}{\sqrt{H+v}} \right) \partial_\eta^2\left( \frac{\Phi}{\sqrt{H+v}} \right)\ \mathrm{d}(x,\eta)\,, \\
J_{1,2} & :=  \int_{\mathcal{R}_I} \frac{\partial_x v}{(H+v)^{3/2}} \,\partial_x \Phi \partial_\eta^2\left( \frac{\Phi}{\sqrt{H+v}} \right)\ \mathrm{d}(x,\eta)\,, \\
J_{1,3} & := \frac{1}{2} \int_{\mathcal{R}_I} \partial_x\left( \frac{\partial_x v}{(H+v)^{3/2}} \right) \Phi \partial_\eta^2\left( \frac{\Phi}{\sqrt{H+v}} \right)\ \mathrm{d}(x,\eta)\ .
\end{align*}
We first infer from \eqref{t21P} and Lemma~\ref{lemt3P} (with $\varphi=\Phi/\sqrt{H+v}$ and $\mu=\sigma (H+v)$) that
\begin{align*}
J_{1,1} & = \int_{\mathcal{R}_I} \left| \partial_x\partial_\eta \left( \frac{\Phi}{\sqrt{H+v}} \right) \right|^2\ \mathrm{d}(x,\eta) \\
& \qquad + \int_I \partial_x \left( \frac{\Phi}{\sqrt{H+v}} \right)(\cdot,0)\, \partial_x \left( \sigma \sqrt{H+v} \Phi \right)(\cdot,0)\ \mathrm{d}x \\
& = \int_{\mathcal{R}_I} \frac{\left| \partial_x\partial_\eta\Phi \right|^2}{H+v}\ \mathrm{d}(x,\eta) - \int_{\mathcal{R}_I} \frac{\partial_x v}{(H+v)^2} \partial_\eta\Phi \partial_x\partial_\eta\Phi\ \mathrm{d}(x,\eta) \\
& \qquad + \frac{1}{4} \int_{\mathcal{R}_I} \frac{(\partial_x v)^2}{(H+v)^3} |\partial_\eta\Phi|^2\ \mathrm{d}(x,\eta) \\
& \qquad + \int_I \partial_x \left( \frac{\Phi}{\sqrt{H+v}} \right)(\cdot,0) \partial_x \left( \sigma \sqrt{H+v} \Phi \right)(\cdot,0)\ \mathrm{d}x \ .
\end{align*}
Next, we integrate by parts and use the boundary values \eqref{t21P} of $\Phi$ to obtain
\begin{align*}
J_{1,2} & = \int_{\mathcal{R}_I} \frac{\partial_x v}{(H+v)^2} \partial_x \Phi \partial_\eta^2\Phi\ \mathrm{d}(x,\eta) \\
& = \int_I \frac{\partial_x v}{(H+v)^2} \Big[ \partial_x \Phi \partial_\eta\Phi \Big]_{\eta=0}^{\eta=1}\ \mathrm{d}x - \int_{\mathcal{R}_I} \frac{\partial_x v}{(H+v)^2} \partial_\eta \Phi \partial_x\partial_\eta\Phi\ \mathrm{d}(x,\eta) \\
& = - \int_I \frac{\sigma \partial_x v}{(H+v)} \Phi(\cdot,0) \partial_x \Phi(\cdot,0)\ \mathrm{d}x - \int_{\mathcal{R}_I} \frac{\partial_x v}{(H+v)^2} \partial_\eta \Phi \partial_x\partial_\eta\Phi\ \mathrm{d}(x,\eta)
\end{align*}
and
\begin{align*}
J_{1,3} & = \int_{\mathcal{R}_I} \left( \frac{\partial_x^2 v}{2(H+v)^2} - \frac{3}{4} \frac{(\partial_x v)^2}{(H+v)^3} \right) \Phi \partial_\eta^2\Phi\ \mathrm{d}(x,\eta) \\
& = \int_I \left( \frac{\partial_x^2 v}{2(H+v)^2} - \frac{3}{4} \frac{(\partial_x v)^2}{(H+v)^3} \right) \Big[ \Phi \partial_\eta\Phi \Big]_{\eta=0}^{\eta=1}\ \mathrm{d}x \\
& \qquad - \int_{\mathcal{R}_I} \left( \frac{\partial_x^2 v}{2(H+v)^2} - \frac{3}{4} \frac{(\partial_x v)^2}{(H+v)^3} \right) |\partial_\eta\Phi|^2\ \mathrm{d}(x,\eta) \\
& = - \int_I \sigma \left( \frac{\partial_x^2 v}{2(H+v)} - \frac{3}{4} \frac{(\partial_x v)^2}{(H+v)^2} \right) |\Phi(\cdot,0)|^2\ \mathrm{d}x \\
& \qquad - \int_{\mathcal{R}_I} \left( \frac{\partial_x^2 v}{2(H+v)^2} - \frac{3}{4} \frac{(\partial_x v)^2}{(H+v)^3} \right) |\partial_\eta\Phi|^2\ \mathrm{d}(x,\eta)\ .
\end{align*}
Next,
\begin{equation*}
J_3 = \int_{\mathcal{R}_I} 2\eta \frac{(\partial_x v)^2}{(H+v)^3} \partial_\eta\Phi \partial_\eta^2\Phi \mathrm{d}(x,\eta) + J_{3,2}
\end{equation*}
with
\begin{equation*}
J_{3,2} := - \int_{\mathcal{R}_I} \frac{\eta \partial_x^2 v}{2(H+v)^2} \partial_\eta\left( |\partial_\eta\Phi|^2 \right)  \mathrm{d}(x,\eta) \ .
\end{equation*}
Integrating by parts and using \eqref{t21P} give
\begin{align*}
J_{3,2} & = - \int_I \frac{\partial_x^2 v}{2(H+v)^2} \Big[ \eta |\partial_\eta\Phi|^2 \Big]_{\eta=0}^{\eta=1}\ \mathrm{d}x + \int_{\mathcal{R}_I} \frac{\partial_x^2 v}{2(H+v)^2} |\partial_\eta\Phi|^2 \mathrm{d}(x,\eta) \\
& = - \int_I \frac{\partial_x^2 v}{2(H+v)^2} |\partial_\eta\Phi(\cdot,1)|^2\ \mathrm{d}x + \int_{\mathcal{R}_I} \frac{\partial_x^2 v}{2(H+v)^2} |\partial_\eta\Phi|^2 \mathrm{d}(x,\eta)\ .
\end{align*}
Consequently,
\begin{equation}
\begin{split}
J & = \int_{\mathcal{R}_I} (H+v) \left( \frac{\partial_x\partial_\eta\Phi}{H+v} - \frac{\eta \partial_x v}{(H+v)^2} \partial_\eta^2\Phi - \frac{\partial_x v}{(H+v)^2} \partial_\eta\Phi \right)^2\ \mathrm{d}(x,\eta) \\
& \qquad + J_4 - \frac{1}{2} \int_I \frac{\partial_x^2 v}{(H+v)^2} |\partial_\eta\Phi(\cdot,1)|^2\ \mathrm{d}x\ ,
\end{split}\label{t22P}
\end{equation}
where
\begin{align*}
J_4 & := \int_I \left( \partial_x\Phi \partial_x(\sigma\Phi) - \frac{\sigma  \partial_x v}{2(H+v)} \Phi \partial_x\Phi - \frac{\partial_x v}{2(H+v)} \Phi \partial_x(\sigma\Phi) \right)(\cdot,0)\ \mathrm{d}x \\
& \qquad + \frac{1}{2} \int_I \sigma \left( \frac{(\partial_x v)^2}{(H+v)^2} - \frac{\partial_x^2 v}{H+v} \right) |\Phi(\cdot,0)|^2\ \mathrm{d}x\ .
\end{align*}
Now, since $H^2(\mathcal{R}_I)$ is continuously embedded in $C(\overline{\mathcal{R}_I})$ by \cite[Theorem~1.4.5.2]{Gr1985}, we infer from \eqref{t21P} that
\begin{equation*}
\Phi(a,0)=\Phi(b,0)=0\ .
\end{equation*}
Using this property along with an integration by parts, we obtain
\begin{align*}
& - \int_I \left( \frac{\sigma \partial_x v}{2(H+v)} \Phi \partial_x\Phi + \frac{\partial_x v}{2(H+v)} \Phi \partial_x(\sigma\Phi) \right)(\cdot,0)\ \mathrm{d}x \\
& \qquad = - \frac{1}{2} \int_I \frac{\partial_x v}{H+v} \partial_x(\sigma\Phi^2)\ \mathrm{d}x \\
& \qquad = - \frac{1}{2} \Big[ \frac{\partial_x v}{H+v} \sigma |\Phi(\cdot,0)|^2 \Big]_{x=a}^{x=b} + \frac{1}{2} \int_I \left( \frac{\partial_x^2 v}{H+v} - \frac{(\partial_x v)^2}{(H+v)^2} \right) \sigma |\Phi(\cdot,0)|^2\ \mathrm{d}x \\
& \qquad = \frac{1}{2} \int_I \left( \frac{\partial_x^2 v}{H+v} - \frac{(\partial_x v)^2}{(H+v)^2} \right) \sigma |\Phi(\cdot,0)|^2\ \mathrm{d}x \ ,
\end{align*}
so that $J_4$ reduces to
\begin{equation}
J_4 = \int_I \left( \partial_x\Phi \partial_x(\sigma\Phi) \right)(\cdot,0)\ \mathrm{d}x\ . \label{t23P}
\end{equation}
We then infer from \eqref{t25P}, \eqref{t22P}, and \eqref{t23P} that
\begin{equation*}
J = \int_{\mathcal{O}_I(v)} |\partial_x\partial_z\zeta_v|^2\ \mathrm{d}(x,z) + \int_I \partial_x \zeta_v(\cdot,-H) \partial_x(\sigma \zeta_v)(\cdot,-H)\ \mathrm{d}x - \frac{1}{2} \int_I \partial_x^2 v |\partial_z \zeta_v(\cdot,v)|^2\ \mathrm{d}x\ .
\end{equation*} 
Combining \eqref{t24P} and the above identity completes the proof.
\end{proof}

\section{Some Functional Inequalities}\label{A2}

Let $I=(a,b)\subset D$ be an open interval and consider $v\in W_\infty^3(I)$ such that $\min_{[a,b]} v >-H$. Let $M>0$ be such that
\begin{equation}
M \ge \max\left\{ 1, \|H+v\|_{L_\infty(I)} , \|\partial_x v\|_{L_\infty(I)} \right\}\,. \label{hello}
\end{equation}
We derive in this section functional inequalities for functions in the subspace $H^1_{WS}(\mathcal{O}_I(v))$ of $H^1(\mathcal{O}_I(v))$ introduced in \eqref{t9P}. Recall that $P\in H^1_{WS}(\mathcal{O}_I(v))$ if and only if $P\in H^1(\mathcal{O}_I(v))$ satisfies
\begin{subequations}\label{t41P}
	\begin{align}
	P(x,-H) & = 0\ , \qquad x\in I\ , \label{t41aP} \\
	P(a,z) & = 0\ , \qquad z\in (-H,v(a)) \ , \label{t41bP}
	\end{align}
\end{subequations} 
We begin with Poincar\'e and Sobolev inequalities and pay special attention to the dependence of the constants on $v$.

\begin{lemma}\label{lemt4P}
Let $P\in H^1_{WS}(\mathcal{O}_I(v))$. Then
\begin{equation*}
\|P\|_{L_2(\mathcal{O}_I(v))}^2 \le 2 M \|\nabla P\|_{L_1(\mathcal{O}_I(v))} \|\partial_z P\|_{L_1(\mathcal{O}_I(v))}\ ,
\end{equation*}
where $M$ is given by \eqref{hello}.
\end{lemma}

\begin{proof}
For $(x,\eta)\in \mathcal{R}_I=I\times (0,1)$, we define
\begin{equation}
Q(x,\eta) := P(x,-H+\eta(H+v(x)))\ , \label{t40P}
\end{equation}
and observe that the regularity of $v$ and $P$ implies that $Q\in H^1(\mathcal{R}_I)$. In addition, we deduce from \eqref{t41P} that
\begin{subequations}\label{t42P}
\begin{align}
Q(x,0) & = 0\ , \qquad x\in I\ , \label{t42aP} \\
Q(a,\eta) & = 0\ , \qquad \eta\in (0,1)\ . \label{t42bP}
\end{align}
\end{subequations}
On the one hand, it follows from \eqref{t42bP} that, for a.a. $(x,\eta)\in \mathcal{R}_I$,
\begin{align}
|(H+v)(x) Q(x,\eta)| & = \left| \int_a^x  \left[ (H+v(x_*)) \partial_x Q(x_*,\eta) + \partial_x v(x_*) Q(x_*,\eta) \right]\ \mathrm{d}x_* \right| \nonumber \\
& \le \mathcal{J}_1(\eta) := \int_I \left| (H+v(x_*)) \partial_x Q(x_*,\eta) + \partial_x v(x_*) Q(x_*,\eta) \right|\ \mathrm{d}x_*\ . \label{t43P}
\end{align}
On the other hand, by \eqref{t42aP}, we obtain, for a.a. $(x,\eta)\in \mathcal{R}_I$,
\begin{equation}
|Q(x,\eta)| = \left| \int_0^\eta \partial_\eta Q(x,\eta_*)\ \mathrm{d}\eta_* \right| \le \mathcal{J}_2(x) := \int_0^1 \left| \partial_\eta Q(x,\eta_*) \right|\ \mathrm{d}\eta_*\ . \label{t44P}
\end{equation}
We then infer from \eqref{t40P}, \eqref{t43P}, and \eqref{t44P} that
\begin{align}
\|P\|_{L_2(\mathcal{O}_I(v))}^2 & = \int_{\mathcal{R}_I} (H+v(x)) |Q(x,\eta)|^2\ \mathrm{d}(x,\eta) \nonumber \\
& \le \int_{\mathcal{R}_I} \mathcal{J}_1(\eta) \mathcal{J}_2(x)\ \mathrm{d}(x,\eta) = \left( \int_I \mathcal{J}_2(x)\ \mathrm{d}x \right) \left( \int_0^1 \mathcal{J}_1(\eta)\ \mathrm{d}\eta \right)\ . \label{t45P}
\end{align}
Now, 
\begin{equation}
\int_I \mathcal{J}_2(x)\ \mathrm{d}x = \int_{\mathcal{R}_I} |\partial_\eta Q(x,\eta)|\ \mathrm{d}(x,\eta) = \int_{\mathcal{O}_I(v)} |\partial_z P(x,z)|\ \mathrm{d}(x,z) \label{t46P} 
\end{equation}
and
\begin{align}
\int_0^1 \mathcal{J}_1(\eta)\ \mathrm{d}\eta & = \int_{\mathcal{R}_I} \left| (H+v(x)) \partial_x Q(x,\eta) + \partial_x v(x) Q(x,\eta) \right|\ \mathrm{d}(x,\eta) \nonumber \\
& = \int_{\mathcal{O}_I(v)} \left| \partial_x P(x,z) + \frac{H+z}{H+v(x)} \partial_x v(x) \partial_z P(x,z) + \frac{\partial_x v(x)}{(H+v)(x)} P(x,z) \right|\ \mathrm{d}(x,z)\ . \label{t47P}
\end{align}
It further follows from \eqref{t41aP} that, for a.a. $(x,z)\in \mathcal{O}_I(v)$,
\begin{equation*}
|P(x,z)| = \left| \int_{-H}^z \partial_z P(x,z_*)\ \mathrm{d}z_* \right| \le \int_{-H}^{v(x)} |\partial_z P(x,z_*)|\ \mathrm{d}z_*\ .
\end{equation*}
Hence,
\begin{equation}
\int_{\mathcal{O}_I(v)} \left| \frac{\partial_x v(x)}{(H+v)(x)} P(x,z) \right|\ \mathrm{d}(x,z) \le \int_{\mathcal{O}_I(v)} |\partial_x v(x)| |\partial_z P(x,z_*)|\ \mathrm{d}(x,z_*)\ . \label{t48P}
\end{equation}
Since $0\le H+z \le H+v(x)$ for $(x,z)\in \mathcal{O}_I(v)$, we deduce from \eqref{hello}, \eqref{t47P}, and \eqref{t48P} that
\begin{equation}
\int_0^1 \mathcal{J}_1(\eta)\ \mathrm{d}\eta \le \int_{\mathcal{O}_I(v)} \left( |\partial_x P(x,z)| + 2 |\partial_x v(x)| |\partial_z P(x,z)| \right)\ \mathrm{d}(x,z) \le 2 M \|\nabla P\|_{L_1(\mathcal{O}_I(v))}\ . \label{t49P}
\end{equation}
Collecting \eqref{t45P}, \eqref{t46P}, and \eqref{t49P} completes the proof.
\end{proof}

Since $\mathcal{O}_I(v)$ is a two-dimensional domain, a classical consequence of Lemma~\ref{lemt4P} is the continuous embedding of $H_{WS}^1(\mathcal{O}_I(v))$ in $L_r(\mathcal{O}_I(v))$ for $r\in [1,\infty)$.  We stress here once more that our main concern is the precise dependence of the embedding constant on $v$.

\begin{lemma}\label{lemt5P}
Let $P\in H^1_{WS}(\mathcal{O}_I(v))$ and $r\in [2,\infty)$. Then
\begin{equation*}
\|P\|_{L_r(\mathcal{O}_I(v))}^r \le \left( 2r \sqrt{M} \right)^{r-2} \|P\|_{L_2(\mathcal{O}_I(v))}^2 \|\nabla P\|_{L_2(\mathcal{O}_I(v))}^{(r-2)/2} \|\partial_z P\|_{L_2(\mathcal{O}_I(v))}^{(r-2)/2}\ ,
\end{equation*}
where $M$ is given by  \eqref{hello}.
\end{lemma}

\begin{proof}
\noindent\textbf{Step~1.} Assume first that $r\ge 4$. For $n\ge 1$, we define the truncation $\mathcal{T}_n$ by $\mathcal{T}_n(s):= s$ for $s\in [-n,n]$ and $\mathcal{T}_n(s) := n\, \mathrm{sign}(s)$ for $s\in (-\infty,-n)\cup (n,\infty)$. Since $\mathcal{T}_n$ is a Lipschitz continuous function on $\mathbb{R}$ with $|\mathcal{T}_n'|\le 1$ and vanishes at zero, the function $\mathcal{T}_n(P)^{r/2}$ also belongs to $H^1_{WS}(\mathcal{O}_I(v))$. We then infer from Lemma~\ref{lemt4P}, the bound $|\mathcal{T}_n'|\le 1$, and H\"older's inequality that
\begin{align*}
\|\mathcal{T}_n(P)\|_{L_r(\mathcal{O}_I(v))}^r & \le \frac{M r^2}{2} \left\| \mathcal{T}_n(P)^{(r-2)/2} \nabla P \right\|_{L_1(\mathcal{O}_I(v))} \left\| \mathcal{T}_n(P)^{(r-2)/2} \partial_z P \right\|_{L_1(\mathcal{O}_I(v))} \\
& \le M r^2 \|\mathcal{T}_n(P)\|_{L_{r-2}(\mathcal{O}_I(v))}^{r-2} \|\nabla P\|_{L_2(\mathcal{O}_I(v))} \|\partial_z P\|_{L_2(\mathcal{O}_I(v))}\ .
\end{align*}
Using again H\"older's inequality, as well as the property $|\mathcal{T}_n(s)|\le |s|$ for $s\in \mathbb{R}$, gives
\begin{align*}
\|\mathcal{T}_n(P)\|_{L_{r-2}(\mathcal{O}_I(v))}^{r-2} & \le \|\mathcal{T}_n(P)\|_{L_r(\mathcal{O}_I(v))}^{r(r-4)/(r-2)} \|\mathcal{T}_n(P)\|_{L_2(\mathcal{O}_I(v))}^{4/(r-2)} \\
& \le \|\mathcal{T}_n(P)\|_{L_r(\mathcal{O}_I(v))}^{r(r-4)/(r-2)} \|P\|_{L_2(\mathcal{O}_I(v))}^{4/(r-2)}\ ,
\end{align*}
since $r\ge 4$. Combining the above two inequalities leads us to
\begin{equation*}
\|\mathcal{T}_n(P)\|_{L_r(\mathcal{O}_I(v))}^r \le (M r^2)^{(r-2)/2} \|P\|_{L_2(\mathcal{O}_I(v))}^{2} \|\nabla P\|_{L_2(\mathcal{O}_I(v))}^{(r-2)/2} \|\partial_z P\|_{L_2(\mathcal{O}_I(v))}^{(r-2)/2}\ .
\end{equation*}
Since the right-hand side of the above inequality does not depend on $n\ge 1$, we may take the limit $n\to\infty$ and deduce from Fatou's lemma that $P\in L_r(\mathcal{O}_I(v))$ and satisfies the stated bound for $r\ge 4$.

\smallskip

\noindent\textbf{Step~2.} Consider now $r\in [2,4]$. By H\"older's inequality and Lemma~\ref{lemt5P} for $r=4$,
\begin{align*}
\|P\|_{L_r(\mathcal{O}_I(v))}^r & \le \|P\|_{L_4(\mathcal{O}_I(v))}^{2(r-2)} \|P\|_{L_2(\mathcal{O}_I(v))}^{4-r} \\
& \le (16M)^{(r-2)/2} \|P\|_{L_2(\mathcal{O}_I(v))}^{2} \|\nabla P\|_{L_2(\mathcal{O}_I(v))}^{(r-2)/2} \|\partial_z P\|_{L_2(\mathcal{O}_I(v))}^{(r-2)/2}\ ,
\end{align*}
and we complete the proof by noticing that $4\le 2r$.
\end{proof}

In the same vein, we derive an estimate for the trace of $P\in H_{WS}^1(\mathcal{O}_I(v))$ on the graph $\mathfrak{G}_I(v)$ of $v$, the trace being here well-defined since the assumption $\min_{[a,b]} v >-H$ guarantees that $\mathcal{O}_I(v)$ is a Lipschitz domain.

\begin{lemma}\label{lemt6P}
Let $P\in H^1_{WS}(\mathcal{O}_I(v))$ and $r\in [2,\infty)$. Then
\begin{equation*}
\|P(\cdot,v)\|_{L_r(I)}^r \le \left( 4r\sqrt{M} \right)^r \|P\|_{L_2(\mathcal{O}_I(v))} \|\nabla P\|_{L_2(\mathcal{O}_I(v))}^{(r-2)/2} \|\partial_z P\|_{L_2(\mathcal{O}_I(v))}^{r/2}\ ,
\end{equation*}
where $M$ is given by \eqref{hello}.
\end{lemma}

\begin{proof}
By \eqref{t41aP} we have, for a.a. $x\in I$, 
\begin{align*}
|P(x,v(x))|^r & \le r \int_{-H}^{v(x)} |P(x,z)|^{r-1} |\partial_z P(x,z)|\ \mathrm{d}z \ .
\end{align*}
Integrating over $I$ and using H\"older's inequality lead us to 
\begin{equation*}
\|P(\cdot,v)\|_{L_r(I)}^r \le r \|P\|_{L_{2(r-1)}(\mathcal{O}_I(v))}^{r-1} \|\partial_z P\|_{L_2(\mathcal{O}_I(v))}\,.
\end{equation*}
Since $2(r-1)\ge 2$ as $r\ge 2$, we deduce from Lemma~\ref{lemt5P} and the above inequality that
\begin{equation*}
\|P(\cdot,v)\|_{L_r(I)}^r \le r \left( 4(r-1) \sqrt{M} \right)^{r-2} \|P\|_{L_2(\mathcal{O}_I(v))}  \|\nabla P\|_{L_2(\mathcal{O}_I(v))}^{(r-2)/2} \|\partial_z P\|_{L_2(\mathcal{O}_I(v))}^{r/2}\ ,
\end{equation*}
from which Lemma~\ref{lemt6P} follows, after using that $r \left(4(r-1) \sqrt{M} \right)^{r-2} \le \left(4r\sqrt{M} \right)^r$. 
\end{proof}

\section{A uniform bound for an auxiliary  stationary problem}\label{A3}

\begin{proposition}\label{propA3.1}
Consider $G_0> 0$, $\beta>0$, and $\tau\ge 0$. Let $I=(a,b)\subset (-L,L)$ be an open interval and let $S_{I}$ be the unique solution to the boundary value problem
\begin{equation}
\beta S_{I}'''' - \tau S_{I}'' = G_0\ , \qquad x\in I\ , \label{eA3.1}
\end{equation}
supplemented with inhomogeneous Dirichlet boundary conditions 
\begin{align}
& S_{I}(a) + H = S'_{I}(a)= S_{I}(b) + H = S'_{I}(b) = 0 \quad\text { if }\quad -L<a<b<L\ , \label{eA3.2} \\
& S_{I}(-L) = S'_{I}(-L)= S_{I}(b) + H= S'_{I}(b) = 0 \quad\text { if }\quad -L=a<b<L\ , \label{eA3.3} \\
& S_{I}(a)+H = S'_{I}(a)= S_{I}(L) = S'_{I}(L) = 0 \quad\text { if }\quad -L<a<b=L\ , \label{eA3.4} 
\end{align}
or clamped boundary conditions
\begin{equation}
S_{I}(-L) = S'_{I}(-L)= S_{I}(L) = S'_{I}(L) = 0 \quad\text { if }\quad -L=a<b=L\ . \label{eA3.5}
\end{equation}
There is $\kappa_0>0$ depending only on $G_0$, $\beta$, $L$, $H$, and $\tau$ such that 
\begin{equation*}
\left| S_{I}(x) \right| \le \kappa_0\ , \qquad x\in [a,b]\ , \qquad -L \le a < b \le L\ .
\end{equation*}
\end{proposition}

\begin{proof}
\noindent\textbf{Case~1: $-L<a<b<L$.} We set $P(y) := S_{I}(a+(b-a)y)+H$ for $y\in [0,1]$ and deduce from \eqref{eA3.1} and \eqref{eA3.2} that $P$ solves the boundary-value problem 
\begin{equation}
\begin{split}
& \beta P'''' - \tau (b-a)^2 P'' = (b-a)^4 G_0\ , \qquad y\in (0,1)\ , \\
& P(0) = P'(0) = P(1) = P'(1) = 0\ . 
\end{split} \label{eA3.6}
\end{equation}
We first infer from \eqref{eA3.6}, the positivity of $G_0$, and a version of Boggio's comparison principle \cite{Bo05, Gr02, LW15, Ow97} that $P>0$ in $(0,1)$. We next multiply \eqref{eA3.6} by $P$ and integrate over $(0,1)$. After integrating by parts and using the boundary conditions, we obtain
\begin{equation*}
\beta \|P''\|_{L_2(0,1)}^2 + \tau (b-a)^2 \|P'\|_{L_2(0,1)}^2 = (b-a)^4 G_0 \int_0^1 P(y)\ \rd y \ .
\end{equation*}
Since
\begin{equation*}
|P(y)| = \left| \int_0^y (y-y_*) P''(y_*) \ \rd y_* \right|  \le \| P''\|_{L_2(0,1)}\ , \qquad y\in (0,1)\ ,
\end{equation*}
by \eqref{eA3.6}, we infer from these observations that
\begin{equation*}
\beta \|P\|_{L_\infty(0,1)}^2 \le \beta \|P''\|_{L_2(0,1)}^2  \le (b-a)^4 G_0 \| P\|_{L_\infty(0,1)} \le 16L^4  G_0 \| P\|_{L_\infty(0,1)}\ .
\end{equation*}
Consequently, $0\le P \le 16L^4 G_0/ \beta$ in $[0,1]$, hence $-H \le S_{I} \le 16 L^4 G_0/ \beta - H$ in $[a,b]$.

\medskip

\noindent\textbf{Case~2: $-L=a<b<L$.} Let $Q\in \mathbb{R}_4[X]$ be such that $Q(0)=Q'(0) = Q(1)+H = Q'(1)= 0$; that is, $Q(y) = y^2(y^2 +2(H-1)y +1- 3H)$. We set $P(y) := S_{I}(-L+(b+L)y)-Q(y)$ for $y\in [0,1]$ and deduce from \eqref{eA3.1} and \eqref{eA3.3} that $P$ solves the boundary value problem 
\begin{equation}
\begin{split}
& \beta P'''' - \tau (b+L)^2 P'' = (b+L)^4 G_0 - \beta Q'''' + \tau (b+L)^2 Q'' \ , \qquad y\in (0,1)\ , \\
& P(0) = P'(0) = P(1) = P'(1) = 0\ . 
\end{split} \label{eA3.7}
\end{equation}
Arguing as in Case~1, we are led to
\begin{align*}
\beta \|P\|_{L_\infty(0,1)}^2 & \le \beta \|P''\|_{L_2(0,1)}^2 + \tau (b+L)^2 \|P'\|_{L_2(0,1)}^2 \\ 
& \le \left[ (b+L)^4 G_0 + 24\beta +14 \tau (H+1) (b+L)^2 \right] \|P\|_{L_\infty(0,1)} \\
& \le \left[ 16L^4 G_0 +  24\beta + 56 \tau (H+1) L^2 \right] \|P\|_{L_\infty(0,1)}\ ,
\end{align*}
since $Q''''=24$ and 
\begin{equation*}
-14(H+1) \le -12y - 6H \le Q''(y) = 12 y^2 + 12(H-1)y + 2(1-3H) \le 14 + 12H \le 14(H+1)\ .
\end{equation*}
 Consequently, 
\begin{equation*}
\|S_{I}\|_{L_\infty(I)} \le \|P\|_{L_\infty(0,1)} + \|Q\|_{L_\infty(0,1)} \le \frac{16L^4 G_0 + 24\beta + 56 \tau (H+1) L^2}{\beta} + \|Q\|_{L_\infty(0,1)}\ .
\end{equation*}
\medskip

\noindent\textbf{Case~3: $-L<a<b=L$.} We set $P(y) := S_{I}(a+y(L-a))-Q(1-y)$ for $y\in [0,1]$ and proceed as in the previous case to derive the same bound for $\|S_{I}\|_{L_\infty(I)}$.
\medskip

\noindent\textbf{Case~4: $-L=a<b=L$.} We set $P(y) := S_{I}(-L+2Ly)$ for $y\in [0,1]$ and deduce from \eqref{eA3.1} and \eqref{eA3.5} that $P$ solves the boundary value problem 
\begin{equation*}
\begin{split}
& \beta P'''' - 4\tau L^2 P'' = 16L^4 G_0\ , \qquad y\in (0,1)\ , \\
& P(0) = P'(0) = P(1) = P'(1) = 0\ . 
\end{split} 
\end{equation*}
We then argue as in Case~1 to conclude that $0\le S_{I} \le 16L^4 G_0/\beta$ in $[-L,L]$.
\end{proof}


\bibliographystyle{siam}
\bibliography{LNW}

\end{document}